\documentclass[12pt]{article}
\usepackage{amsmath,amsfonts}
\usepackage{verbatim}
\usepackage{latexsym}
\usepackage{graphicx}
\usepackage{mathrsfs} 
\usepackage{amsmath}
\usepackage{stmaryrd}
\usepackage{xcolor}
\usepackage{bm}
\usepackage{cite}
\usepackage{amsthm,amsfonts,amssymb,amsmath,oldgerm}
\usepackage{epsfig}
\numberwithin{equation}{section}
\usepackage[thinlines]{easybmat}
\usepackage{graphicx}

\makeatletter
\@addtoreset{equation}{section}

\begin{document}

\newcommand{\E}{\mathbb{E}}
\newcommand{\PP}{\mathbb{P}}
\newcommand{\RR}{\mathbb{R}}

\newtheorem{theorem}{Theorem}[section]
\newtheorem{remark}[theorem]{Remark}
\newtheorem{lemma}[theorem]{Lemma}
\newtheorem{coro}[theorem]{Corollary}
\newtheorem{defn}[theorem]{Definition}
\newtheorem{assp}[theorem]{Assumption}
\newtheorem{expl}[theorem]{Example}
\newtheorem{prop}[theorem]{Proposition}
\newtheorem{rmk}[theorem]{Remark}
\newcommand\tq{{\scriptstyle{3\over 4 }\scriptstyle}}
\newcommand\qua{{\scriptstyle{1\over 4 }\scriptstyle}}
\newcommand\hf{{\textstyle{1\over 2 }\displaystyle}}
\newcommand\hhf{{\scriptstyle{1\over 2 }\scriptstyle}}
\newcommand{\LA}[1]{\ref{lemma:#1}}

\newcommand{\eproof}{\hfill $\Box$} 
\def\tl{\tilde}
\def\trace{\hbox{\rm trace}}
\def\diag{\hbox{\rm diag}}
\def\for{\quad\hbox{for }}
\def\refer{\hangindent=0.3in\hangafter=1}
\newcommand{\vvec}[1]{{\mathbf{#1}}}
\newcommand\wD{\widehat{\D}}
\def\FGhat{{\widehat{F}_{\Gamma}}}
\def\FGtilde{{\widetilde{F}_{\Gamma}}}
\def\FG{{F_{\Gamma}}}
\def\FGhatc{{\widehat{F}_{c,\widehat{\Gamma}}}}
\def\FGtildec{{\widetilde{F}_{c,\widehat{\Gamma}}}}
\def\FGc{{F_{c,\widehat{\Gamma}}}}
\renewcommand{\div}{\mathop{\sf div  }}
\def\MQ{{\mathcal Q}}
\def\MU{{\mathcal U}}
\def\A{{\cal A}}
\def\B{{\cal B}}
\def\Bhat{{\hat{\cal B}}}
\def\C{{\cal C}}
\def\D{{\cal D}}
\def\E{{\cal E}}
\def\Real{{\bf R}}
\def\T{{\cal T}}
\def\G{{\cal G}}
\def\F{{\cal F}}
\def\H{{\cal H}}
\def\L{{\cal L}}
\def\U{{\cal U}}
\def\uhat{{\widehat{u}}}
\def\qhat{{\widehat{\bf q}}}
\def\Hhat{{\hat{\cal H}}}
\def\Htilde{{\tilde{\cal H}}}
\def\I{{\cal I}}
\def\Ktilde{{\widetilde{K}}}
\def\Khat{{\widehat{K}}}
\def\Kbar{{\overline{K}}}
\def\barlambda{{\overline{\lambda}}}
\def\Kbreve{{\breve{K}}}
\def\M{{\cal M}}
\def\Mhat{{\widehat{M}}}
\def\N{{\cal N}}
\def\Khat{{\widehat{K}}}
\def\PDr{{P_{\Delta}}}
\def\PBD{{P_{_D}}}
\def\Q{{\cal Q}}
\def\bfQ{{\bf {\cal Q}}}
\def\RBM{{\rm RBM}}
\def\V{{\cal V}}
\def\W{{\cal W}}
\def\What{{\widehat{W}}}
\def\PDr{{P_{\Delta}}}
\def\Stilde{{\widetilde{S}}}
\def\Stildeeps{{\widetilde{S}}_{\large\varepsilon}}
\def\Seps{S_{\large\varepsilon}}
\def\Sepsdelta{S_{{\varepsilon},\Delta}}
\def\Sdelta{S_{\Delta}}
\def\Stildeepsdelta{{\widetilde{S}}_{{\large\varepsilon},\Delta}}
\def\Wtilde{{\widetilde{W}}}
\def\Rtilde{{\widetilde{R}}}
\def\Rbar{{\widebar{R}}}
\def\ker{{\rm \bf ker \, }}
\def\range{{\rm \bf range \, }}
\def\diver{{\rm div \, }}
\def\rot{{\rm rot \, }}
\def\>{\raisebox{-1ex}{$\; \stackrel{\textstyle >}{\sim } \; $}}
\def\<{\raisebox{-1ex}{$ \; \stackrel{\textstyle <} {\sim } \; $}}
\def\endproof{\begin{flushright} $\Box $ \end{flushright}}
\def\uG{{\mu_\Gamma}}
\def\uE{{\mu_{{\mathcal A},\Gamma}}}
 \def\pE{{p_{{\mathcal A},\Gamma}}}
\def\vE{{v_{{\mathcal A},\Gamma}}}
\def\wE{{w_{{\mathcal A},\Gamma}}}
\def\uH{{u_{{\mathcal H},\Gamma}}}
\def\wG{{w_\Gamma}}
\def\Imap{{I_H^{\Omega^i}}}
\def\Imapj{{I_H^{\Omega^j}}}
\def\Ibmap{{I_H^{\partial\Omega^i}}}
\def\Ibmapj{{I_H^{\partial\Omega^j}}}
\def\Ibmapk{{I_H^{\partial\Omega^k}}}
\def\Ibmapl{{I_H^{\partial\Omega^l}}}
\def\Ktilde{{\widetilde{K}}}
\def\Btilde{{\widetilde{B}}}
\def\phitilde{{\widetilde{\varphi}}}
\def\Ntilde{{\widetilde{N}}}
\def\Ztilde{{\widetilde{Z}}}
\def\Atilde{{\widetilde{A}}}
\def\Khat{{\widehat{K}}}
\def\Kbar{{\overline{K}}}
\def\Kbreve{{\breve{K}}}
\def\M{{\cal M}}
\def\Mhat{{\widehat{M}}}
\def\N{{\mathcal N}}
\def\Khat{{\widehat{K}}}
\def\PDr{{P_{\Delta}}}
\def\PBD{{P_{_D}}}
\def\Q{{\cal Q}}
\def\RBM{{\rm RBM}}
\def\V{{\cal V}}
\def\W{{\cal W}}
\def\What{{\widehat{W}}}
\def\wbar{{\overline{w}}}
\def\PDr{{P_{\Delta}}}
\def\Stilde{{\widetilde{S}}}
\def\STG{{\Stilde_\Gamma}}
\def\KTG{{\widetilde{K}_\Gamma}}
\def\BTG{{\widetilde{B}_\Gamma}}
\def\NTG{{\widetilde{N}_\Gamma}}
\def\ZTG{{\widetilde{Z}_\Gamma}}
\def\Utilde{{\widetilde{U}}}
\def\Stildeeps{{\widetilde{S}}_{\large\varepsilon}}
\def\Seps{S_{\large\varepsilon}}
\def\Sepsdelta{S_{{\varepsilon},\Delta}}
\def\Sdelta{S_{\Delta}}
\def\Stildeepsdelta{{\widetilde{S}}_{{\large\varepsilon},\Delta}}
\def\Wtilde{{\widetilde{W}}}
\def\Wbar{{\overline{W}}}
\def\ker{{\rm \bf ker \, }}
\def\range{{\rm \bf range \, }}
\def\diver{{\rm div \, }}
\def\rot{{\rm rot \, }}
\def\>{\raisebox{-1ex}{$\; \stackrel{\textstyle >}{\sim } \; $}}
\def\<{\raisebox{-1ex}{$ \; \stackrel{\textstyle <} {\sim } \; $}}
\def\u{{{\bf u}}}
\def\v{{{\bf v}}}
\def\q{{{\bf q}}}
\def\SGS{{B_\Gamma}}
\def\SGSS{{Z_\Gamma}}
\def\gtilde{{\widetilde{g}}}
\def\ctilde{{\tilde{c}}}
\def\ytilde{{\widetilde{y}}}
\def\Ttilde{{\widetilde{T}}}
\def\Mtilde{{\widetilde{M}}}
\def\Vtilde{{\widetilde{V}}}
\def\hatI{{\widehat{I}}}
\def\hatGamma{{\widehat{\Gamma}}}
\def\barR{{\overline{R}}}
\def\hatbarR{\widehat{{\overline{R}}}}
\def\Rtilde{{\widetilde{R}}}
\def\RtildeG{{\widetilde{R}_\Gamma}}
\def\hatRtilde{{\widehat{\widetilde{R}}}}
\def\hatDelta{{\widehat{\Delta}}}
\def\hatD{{\widehat{D}}}
\def\hatE{{\widehat{E}}}
\def\hatB{{\widehat{B}}}
\def\hatP{{\widehat{P}}}
\def\hatR{{\widehat{R}}}
\def\hatM{{\widehat{M}}}
\def\hatN{{\widehat{N}}}
\def\hatA{{\widehat{A}}}
\def\hatPi{{\widehat{\Pi}}}
\def\hatH{{\hat{H}}}
\def\hatu{{\hat{u}}}
\def\hata{{\hat{a}}}
\def\hatdelta{{\widehat{\delta}}}
\def\hatS{{\widehat{S}}}
\def\K{\left(\hatRtilde_{\hatGamma}^{T}\Ttilde\hatRtilde_{\hatGamma}\right)}
\def\logHh{\left(1+\log\frac{H}{h}\right)}
\def\q{{{\bf q}}}
\def\vvecV{{{\bf V}}}
\def\vvecn{{{\bf n}}}
\def\P{{{\bf P}}}
\def\r{{{\bf r}}}
\def\vbeta{{{\bm \beta}}}
\def\Lhat{\widehat{\Lambda}}
\def\Ltilde{\widetilde{\Lambda}}
\newcommand{\EQ}[1]{(\ref{equation:#1})}
\def\beginproof{\indent {\it Proof:~}}
\def\endproof{\qquad $\Box $}
\def\mymax{{C_L(\epsilon,H,h)}}
\def\SG{{S_\Gamma}}
\def\SGS{{B_\Gamma}}

\def\Cone{\nu_{1,\epsilon,\tau_K,h}}

\def\Ctwo{\nu_{\epsilon,\tau_K,h}}

\def\Cthree{\nu{3,\epsilon,\tau_K,h}}

\def\C4{{C_{4,\epsilon,{\bm\beta},\tau_K,h}}}

\title{BDDC Algorithms for Advection-diffusion problems with HDG Discretizations}

\author {Xuemin Tu\thanks{Department of Mathematics, University of Kansas, 1460 Jayhawk Blvd, Lawrence, KS 66045-7594, U.S.A, E-mail:
{xuemin@ku.edu}} 
\and  Jinjin Zhang\thanks{Department of Mathematics, University of Kansas, 1460 Jayhawk Blvd, Lawrence, KS 66045-7594, U.S.A, E-mail:
{jinjinzhang@ku.edu}}
}

\maketitle

\begin{abstract}

The balancing domain decomposition methods (BDDC) are originally introduced for
symmetric positive definite systems and have been extended to the
nonsymmetric positive definite system from the linear finite element
discretization of advection-diffusion equations. In this paper, the
convergence of the GMRES method is analyzed  for the BDDC preconditioned linear system from
advection-diffusion equations with the hybridizable discontinuous
Galerkin  (HDG) discretization. Compared to the finite element discretizations, 
several additional norms for the numerical trace have to be used and the
equivalence between the bilinear forms and norms needs to be
established.  For large viscosity, if the subdomain size is small
enough, the number of iterations is independent of the number of
subdomains and depends only slightly on the sudomain problem
size. The convergence deteriorates when the
viscosity decreases. These results are similar to those with the finite element
discretizations.  Moreover, the effects of the additional  primal
constraints used in the BDDC algorithms are more significant with the
higher degree HDG discretizations. The results of two two-dimensional
examples are provided to confirm our theory.

\end{abstract}
\medskip \noindent
{\small\bf Key words:} 
Discontinuous Galerkin,  HDG, domain decomposition, BDDC, advection-diffusion
\section{Introduction}

General hybridizable discontinuous
Galerkin (HDG)  methods were
introduced for second order elliptic problems in \cite{CGL2009}, where
the only global coupled degrees of freedom are
 called ``numerical traces'', a scalar variable.  This significantly
 reduces the number of the degrees of freedom in the resulting global
 system  comparing to other traditional
  DG
methods.  The HDG has been used for many different applications  such
as  Stokes equation
\cite{NPC2010,HDGS2011}, Maxwell's equation \cite{HDGM},  Helmholtz
equation \cite{HDGH}, Oseen equation \cite{HDGOseen}, and
advection-diffusion equation \cite{nguyen_implicit_2009,fu_analysis_2015-1,QS2016}.  The
resulting linear system from the HDG discretization of
advection-diffusion equation is nonsymmetric but usually positive definite and 
the generalized minimal residual methods
(GMRES) can be  used to solve the linear system.   To reach
a given accuracy,  many GMRES iterations might be needed, especially for
HDG discretizations with high order basis functions.  Preconditioners
are necessary to accelerate the convergence of GMRES. To our
best knowledge, there are very few
fast solvers for the advection-diffusion problems with HDG
discretizations.

Domain decomposition methods are widely used preconditioner techniques for
solving large sparse linear systems arising from finite element
discretization of partial differential equations (PDEs). A number of domain decomposition methods have been
proposed for solving advection-diffusion problems. Overlapping Schwarz
methods were studied 
in \cite{3cai,1caiwid,2caiwid}, where the rate of convergence was proved to be independent of the
number of  subdomains if the
coarse mesh is fine enough. For nonoverlapping domain decomposition
methods,  the standard 
 Dirichlet and
Neumann boundary conditions for the subdomain local problems cannot
ensure  the positive definiteness of the local bilinear forms and
therefore are not appropriate, see \cite[Chapter 6]{Quarteroni:1999:DDM}
and references therein for more details. A Robin boundary condition
for the subdomain local problems 
was used in \cite{Achdou:1997:RPA} to overcome this difficulty,
see also~\cite{Achdou:1999:DDN, Achdou:2000:DDAD}. In \cite{006tos}, 
the one-level and two-level FETI algorithms were proposed, with the
same subdomain Robin boundary conditions, for solving
advection-diffusion problems.

The balancing
domain decomposition by constraints (BDDC) methods, one of the most
popular nonoverlapping domain decomposition methods,  were introduced
in \cite{Dohrmann:2003:PSC} and  analyzed in \cite{Jan1,
Jan2} for symmetric positive definite problems. The BDDC methods have
also been extended to solving those  linear systems resulting from the discretization of
Stokes equations ~\cite{LiBDDCS,TW:2017:WGS}, the flow in
porous media ~\cite{Tu:2005:BPP, Tu:2005:BPD,
Tu_thesis}, and  the Helmholtz equations ~\cite{BDDCH}. The BDDC
algorithms have also been applied to second order elliptic problems
with HDG
discretization \cite{TW:2016:HDG} and the Stokes equation 
\cite{TW:2017:WGS,TWZstokes2020}. 
 In \cite{Conceicao_thesis},
extensive numerical experiments for 
some additive and multiplicative BDDC algorithms with vertex constraints 
or edge
average constraints have
been studied for advection-diffusion problems with stabilized finite
element discretizations.  A convergence rate  estimate of the GMRES was established in
\cite{TuLi:2008:BDDCAD} for solving the
the BDDC preconditioned advection-diffusion problems with linear
finite element discretization, stabilized by the Galerkin/least squares
methods  \cite{Hughes:1989:NFE}. The Robin boundary conditions were
used for subdomain local problems and two flux-based primal
constraints are added in additional to standard vertex and edge/face
average constraints. Similar to \cite{1caiwid}, a perturbation
approach  is used  to handle the asymmetry  of
the problem in the analysis and the partially sub-assembled finite
element problem in the preconditioner is treated as a  non-conforming finite element approximation.

In this paper, following the approach used in \cite{TuLi:2008:BDDCAD
}, we develop and analyze the BDDC algorithms for
advection-diffusion problems with the HDG discretizations. Compared to
the standard linear finite element discretizations, the coarse problems
in the BDDC algorithms are simpler since only edge/face related
constraints are needed for two/three dimensions. However, there are several
difficulties in the analysis for the HDG discretizations and different
technical tools are needed. First, since the BDDC preconditioned
system is a reduced system for the ``numerical
trace''  $\lambda$ in the HDG discretization, we need to handle and estimate
$\Q\lambda$ and $\U\lambda$ these two operators in the HDG
discretizations appropriately. Second,  we need to establish
the equivalent relation between the reduced norm from the bilinear form of the HDG
discretization of  the
advection-diffusion PDE and a norm defined for $\lambda$. This norm
plays a similar rule as the
$H^1$ semi-norm in \cite{TuLi:2008:BDDCAD}, where the relation is
easily obtained by the definitions. Third, in the HDG
discretization, since $\lambda$ is defined only on the mesh element
boundaries, appropriate extensions to the element interiors are needed for
the error estimate of the partially sub-assembled problem in the preconditioner. 
A new
norm, originally defined in \cite{multigrid}, is used to 
replace the $L^2$ norm used in \cite{TuLi:2008:BDDCAD
},  for the non-conforming finite element approximation approach.  
Moreover, an error analysis for 
the numerical trace $\lambda$ in an appropriate norm needs to be established. 
 
The rest of the paper is organized as follows. We first describe the advection-diffusion
problem and the adjoint problem in Section 2. The
HDG discretizations are introduced and a reduced system for $\lambda$
is formed. 
In Section 3, the local bilinear forms and a reduced subdomain
interface problems are introduced.  The BDDC preconditioner and
several useful norms are provided in Sections 4 and 5, respectively. We provide an
analysis of the convergence rate for our BDDC algorithms in Section 6
and all detailed proofs of the lemmas used in Section 6 are provided
in Section 7. 
Finally,
some computational results are presented in Section 8.

\section{Problem setting and HDG discretizations}\label{sec:problem}
Following \cite{ChenC2012,ChenC2014,fu_analysis_2015-1,QS2016},  we consider a second order scalar
advection-diffusion problem  defined in a bounded polyhedral domain $\Omega
\in {\bf R}^n$,   ($n=2$ or $3$ for two and three dimensions, respectively):
\begin{equation}
\label{equation:pde1} \left\{
\begin{array}{rcl}
-\epsilon \Delta u+ \vbeta \cdot \nabla u & =& f ,
\quad \mbox{in } \Omega, \\ [0.5ex] u & = &0 , \quad\mbox{on }
\partial \Omega.
\end{array} \right.
\end{equation}
Here the constant viscosity $\epsilon>0$,
the velocity field $\vbeta(x)\in \left(L^{\infty}(\Omega)\right)^n$,  and  $f(x)\in L^2(\Omega)$.
We also assume that 
\begin{equation}
\label{equation:cond1} \nabla\cdot\vbeta(x)\in L^{\infty}(\Omega),
\quad -\nabla\cdot \vbeta(x)\ge 0
 \quad  \forall ~ x\in \Omega.
\end{equation}
$\vbeta$ has no closed curves and $\vbeta (x) \neq
0$ for any $x\in \Omega$ to ensure the well-posedness of the
continuous problem in $\epsilon=0$ limit, see \cite{DEF1974,RST1996,AyusoM2009,fu_analysis_2015-1}.

The regularity result 
\begin{equation}\label{equation:regularity} \|u\|_{H^2(\Omega)}\le
\frac{C}{\epsilon}\| f \|_{L^2(\Omega)},
\end{equation}
is assumed for the weak solutions of both the original 
problem \EQ{pde1} and adjoint
problem 
\begin{equation}
\label{equation:pdea} \left\{
\begin{array}{rcl}
L^*u=-\epsilon \Delta u- \nabla\cdot(\vbeta u) & =& f ,
\quad \mbox{in } \Omega, \\ [0.5ex] u & = &0 , \quad\mbox{on }
\partial \Omega.
\end{array} \right.
\end{equation}
Here $C$ is a positive constant independent of $\epsilon$. 

\subsection{HDG discretization} Following \cite{nguyen_implicit_2009,fu_analysis_2015-1}, we
 introduce a new variable  
$\q=-\epsilon \nabla u$
and let $\rho=\epsilon^{-1}$. 
We obtain the following system for $\q$ and $u$ as
\begin{equation}\label{equation:upeqn}
\left\{
\begin{array}{ll}
\epsilon^{-1}\q=-\nabla u  & \mbox{in}\quad\Omega, \\
\nabla\cdot \q +\vbeta\cdot \nabla u = f & \mbox{in}\quad\Omega, \\
u = 0 &  \mbox{in}\quad \partial\Omega.
\end{array}
\right.
\end{equation}

Here $\q$ and $u$ will be approximated by some discontinuous finite
element spaces. We first introduce a  shape-regular and quasi-uniform
triangulation $\T_h$ of
$\Omega$.  The  characteristic element size of $\T_h$ is denoted by
$h$,  the element is 
denoted by $K$, and  the union of faces of elements is denoted by $\E$.  The sets of the domain interior
and boundary faces are denoted by $\E_i$ and $\E_\partial$, respectively. 

 Let $P_k(D)$ be the
space of polynomials of order at most $k$ on $D$. We set
$\P_k(D)=[P_k(D)]^n$ and  
define the following finite element spaces:
\begin{eqnarray*}
&&{{\bf V}}_k=\{\v_h\in [L^2(\Omega)]^n: \v_h|_K\in \P_k(K)\quad
\forall K \in \Omega\},\\
&&{W}_k=\{w_h\in L^2(\Omega): w_h|_K\in P_k(K)\quad
\forall K \in \Omega\},\\
&&{M}_k=\{\mu_h\in L^2(\E): \mu_h|_e\in P_k(e)\quad
\forall e \in \E\}.
\end{eqnarray*} 
Let $\Lambda_k=\{\mu\in M_k: \mu|_e=0 ~\forall e\in \partial\Omega\}$. To
make our notations simple, we drop the subscript $k$ from now on for
${{\bf V}}_k$, ${W}_k$, ${M}_k$, and $\Lambda_k$.

For each $K$, let  $(\q_h, u_h)\in (\P_k(K), P_k(K))$ such that for all $K
\in \T_h$
\begin{equation}
\label{equation:updiselement}
\left\{
\begin{array}{lcll}
(\epsilon^{-1}\q_h,\r_h)_K- (u_h, \nabla \cdot \r_h)_K +
\left<\uhat_h,\r_h\cdot {\bf n}\right>_{\partial K}&=&0,&\forall
\r_h\in \P_k(K),\\
(\q_h+\vbeta u_h ,\nabla w_h)_K -\left<(\qhat_h+\widehat{\vbeta
  u_h})\cdot {\bf n},w_h\right>_{\partial K}   +(\nabla\cdot\vbeta u_h , w_h)_K            &=&-(f,w_h)_K,&\forall w_h\in P_k(K),
\end{array}\right.
\end{equation}
where $(\cdot,\cdot)_K$ and
$\left<\cdot,\cdot\right>_{\partial K}$ denote $L^2$-inner product
for functions defined in $K$ and $\partial
K$, respectively. $\uhat_h$ and
$\qhat_h+\widehat{\vbeta
  u_h}$ are the numerical traces which approximate $u_h$ and
$-\epsilon \nabla u_h+\vbeta u_h$
on $\partial K$, respectively. 

The numerical trace $\uhat_h=\lambda_h$ for $\lambda_h\in \Lambda$. 
$(\qhat_h+\widehat{\vbeta u_h}) \cdot {\bf n}$ takes the form:
\begin{equation}\label{equation:nqtrace}
(\qhat_h+\widehat{\vbeta u_h}) \cdot {\bf n} =\q_h\cdot {\bf n}+\vbeta \cdot {\bf n}\lambda_h+\tau_K (u_h-\lambda_h), \quad
\mbox{on }\partial K,
\end{equation}
where the local stabilization parameter $\tau_K$ is piecewise, nonnegative constant defined on $\partial\T_h,$ see \cite{nguyen_implicit_2009} for details.  Following \cite[Lemma 3.1]{nguyen_implicit_2009}, we assume
\begin{assp}\label{assumption:beta0}
  $$\inf_{x\in \E} \left(\tau_K-\frac12\vbeta(x)\cdot{\bf n} \right)\ge 0,\quad
    \forall \E\in \partial K, ~\forall K\in \T_h,$$
where the strict inequality holds at least on one edge $\E$ in each
element $K$.   
\end{assp}

Using  the definitions of numerical traces $\lambda_h$ and 
$(\qhat_h+\widehat{\vbeta u_h}) \cdot {\bf n}$, we can write  the discrete problem resulting
from HDG discretization  as: to find
$(\q_h, u_h, \lambda_h)\in {\bf V}\times W\times \Lambda$ such that
for all $(\r_h,w_h,\mu_h)\in \vvecV\times W\times \Lambda$ 
\begin{equation}\label{equation:updisdual}
\left\{
\begin{array}{lcll}
(\epsilon^{-1}\q_h,\r_h)_{\T_h}  - (u_h,\nabla \cdot
\r_h)_{\T_h}+\left<\lambda_h,\r_h\cdot\vvecn\right>_{\partial \T_h} 
&=&0,\\
(\q_h+\vbeta u_h,\nabla w_h)_{\T_h}+(\nabla\cdot \vbeta u_h,w)_{\T_h}-\left<\q_h\cdot {\bf n}+\vbeta \cdot {\bf n}\lambda_h+\tau_K (u_h-\lambda_h),w_h\right>_{\partial \T_h}
  &=&- (f,w_h)_\Omega,\\
\left<\q_h\cdot {\bf n}+\vbeta \cdot {\bf n}\lambda_h+\tau_K (u_h-\lambda_h), \mu_h\right>_{\partial \T_h}&=&0,
\end{array}\right.
\end{equation}
where $(\cdot,\cdot)_{\T_h}=\sum_{K\in \T_h}(\cdot,\cdot)_K$
and $\left<\cdot,\cdot\right>_{\partial\T_h}=\sum_{K\in
  \T_h}\left<\cdot,\cdot\right>_{\partial K}$. 

As  \cite[Equations (4.1) and (4.2)]{fu_analysis_2015-1}, for any
$({\bf q}, u, {\bf\lambda}), ~({\bf r},w,\mu)\in \vvecV\times
W\times \Lambda$, we can define the following bilinear form 
\begin{align}\label{equation:Bform}
&D(({\bf q},u,\lambda),({\bf r},w,\mu))\\
&=(\epsilon^{-1}{\bf q}, {\bf r})_{\T_h}-(u, \nabla\cdot {\bf r} ) _{\T_h}+\langle\lambda, {\bm r}\cdot {\bf n}\rangle_{\partial \T_h}\nonumber\\
& \quad -({\bf q}+{\bm\beta}u,\nabla w)_{\T_h}+\langle ({\bf q}+{\bm\beta\lambda})\cdot{\bf n}+\tau_K(u-\lambda),w \rangle_{\partial\T_h}-(\nabla\cdot{\bm\beta}u,w)_{\T_h}\nonumber\\
&\quad-\langle({\bf q}+{\bm \beta}\lambda)\cdot n+\tau_K(u-\lambda),\mu\rangle_{\partial\T_h}.\nonumber
\end{align}
The HDG method \EQ{updisdual} can be written as the following compact
form: to find $({\bf q}_h, u_h,{\bf\lambda}_h)\in \vvecV\times W\times \Lambda$ such that 
\begin{equation}\label{equation:Deqn}
D(({\bf q}_h,u_h,\lambda_h),({\bf r}, w,\mu))=(f,w)_{\mathcal{T}_h},
\end{equation}
for all $({\bf r},w,\mu)\in \vvecV\times W\times \Lambda$.


\subsection{The system for $\lambda$}
Next, we reduce \EQ{Deqn} to a system for $\lambda_h$ only by
eliminating ${\bf q}_h$ and $u_h$.   Defining the following operators 
$\mathcal{A}:{\bf V}\rightarrow {\bf V}$, $\mathcal{B}:{\bf
    V}\rightarrow W$, $\mathcal{C}:{\bf V}\rightarrow M$, $\mathcal{R}:
  W\rightarrow W$, $\mathcal{S}_1: M\rightarrow W$, $\mathcal{S}_2:
  W\rightarrow M$, and $\mathcal{T}:  M\rightarrow M$ as follows:
\begin{align}\label{equation:alloperators}
(\mathcal{A}{\bf q},{\bf r})_{\mathcal{T}_h}&=(\epsilon^{-1}{\bf
                                              q},{\bf
                                              r})_{\mathcal{T}_h},
                                              (\mathcal{B}{\bf r},u)_{\mathcal{T}_h}=-(u,\nabla\cdot{\bf r})_{\mathcal{T}_h},
\langle\mathcal{C} {\bf r},\lambda\rangle_{\partial{\mathcal{T}_h}}=\langle \lambda,{\bf r}\cdot {\bf n}\rangle_{\partial{\mathcal{T}_h}}\\
(\mathcal{R}u,w)_{\mathcal{T}_h}&=- \langle \tau_K
                                              u,w \rangle_{\partial\mathcal{T}_h}+(\vbeta
                                              u,\nabla
                                              w)_{\mathcal{T}_h}
                                              +(\nabla \cdot \vbeta u,w) _{\mathcal{T}_h},\nonumber\\
&=- \langle \tau_K
                                              u,w \rangle_{\partial\mathcal{T}_h}+\frac{1}{2}(\vbeta
                                              u,\nabla
                                              w)_{\mathcal{T}_h}+\frac{1}{2}(\vbeta
                                              u,\nabla
                                              w)_{\mathcal{T}_h}
                                              +(\nabla \cdot \vbeta u,w) _{\mathcal{T}_h}\nonumber\\
&=-\langle (\tau_K-\frac12{\bm \beta}\cdot{\bf n})u,w\rangle_{\partial\mathcal{T}_h}+(\frac12 \nabla\cdot{\bm \beta}u,w)_{\mathcal{T}_h}
-\frac12({\bm \beta}\cdot\nabla u,w)_{\mathcal{T}_h}+\frac12(u,{\bm\beta}\cdot\nabla w)_{\mathcal{T}_h},\nonumber\\
(\mathcal{S}_1\lambda,w)_{{\mathcal{T}_h}}&=
                                                                  \langle
                                                                \tau_K
                                                                  \lambda,w\rangle_{\partial\mathcal{T}_h}-\langle {\vbeta}\cdot{\bf n}\lambda,w\rangle_{\partial{\mathcal{T}_h}}\nonumber\\
&=\langle (\tau_K-\frac12{\bm \beta}\cdot{\bf n})\lambda,w\rangle_{\partial\mathcal{T}_h}-\frac12\langle {\bm\beta}\cdot{\bf n}\lambda,w\rangle_{\partial{\mathcal{T}_h}}, \nonumber\\
\langle\mathcal{S}_2
  u,\mu\rangle_{\partial\mathcal{T}_h}&=\langle \tau_K u,\mu\rangle_{\partial\mathcal{T}_h}\nonumber\\
&=\langle(\tau_K-\frac12{\bm \beta}\cdot{\bf n})u,\mu\rangle_{\partial\mathcal{T}_h}+\frac12\langle{\bm \beta}\cdot{\bf n}u,\mu\rangle_{\partial\mathcal{T}_h}, \nonumber\\
\langle\mathcal{T}\lambda,\mu \rangle_{\partial
  \mathcal{T}_h}&=-\langle(\tau_K-{\bm \beta}\cdot{\bf
                  n})\lambda,\mu\rangle_{\partial\mathcal{T}_h}, F_h(w)=-(f,w)_{\mathcal{T}_h}. \nonumber
\end{align}

The HDG method generates operator equations of the following form 
\begin{equation}\label{equation:eq3.1}
\begin{pmatrix}
\mathcal{A}&\mathcal{B}^t&\mathcal{C}^t\\
\mathcal{B}&\mathcal{R}&\mathcal{S}_1\\
\mathcal{C}&\mathcal{S}_2&\mathcal{T}
\end{pmatrix}
\begin{pmatrix}
{\bf q}_h\\
u_h\\
\lambda_h
\end{pmatrix}=
\begin{pmatrix}
0\\
F_h\\
0
\end{pmatrix}.
\end{equation}

In each $K$, given the value of $\lambda_h$ on $\partial K$, under
Assumption \ref{assumption:beta0},
${\bf q}_h$ and $u_h$ can be uniquely determined, see
\cite[Lemma 3.1]{nguyen_implicit_2009}. Namely, given $\lambda_h$, the solution $(\q_h,
u_h)$ of \EQ{eq3.1}  is uniquely determined by the first two
equations. This elimination can be described by introducing the
following operators: ${\mathcal{Q}}_v:{\bf V}\rightarrow
{\bf V}$, ${\mathcal{Q}}_w: W\rightarrow {\bf V}$, $\mathcal{U}_v:{\bf
  V}\rightarrow W$, ${\mathcal{U}_w}:W\rightarrow W$:
\begin{equation}\label{equation:QvQwdef}
\begin{pmatrix}
\mathcal{A} &\mathcal{B}^t\\
\mathcal{B} &\mathcal{R}
\end{pmatrix}
\begin{pmatrix}
{\mathcal{Q}}_v{\bf g}_h\\
{\mathcal{U}_v{\bf g}_h}
\end{pmatrix}=
\begin{pmatrix}
{\bf g}_h\\
0
\end{pmatrix},\quad 
\begin{pmatrix}
\mathcal{A}&\mathcal{B}^t\\
\mathcal{B}&\mathcal{R}
\end{pmatrix}
\begin{pmatrix}
{\mathcal{Q}_w}{f}_h\\
{\mathcal{U}_w}f_h
\end{pmatrix}=\begin{pmatrix}
0\\
f_h
\end{pmatrix},
\end{equation}
for all ${\bf g}_h\in {\bf V},$ and $f_h\in W$. 

Let 
\begin{eqnarray}\label{equation:QmuUmudef}
\mathcal{Q}{\mu}&=&-\mathcal{Q}_v(\mathcal{C}^{t}\mu)-\mathcal{Q}_{w}(\mathcal{S}_1\mu),\\
\mathcal{U}{\mu}&=&-\mathcal{U}_v(\mathcal{C}^t\mu)-\mathcal{U}_{w}(\mathcal{S}_1\mu).\nonumber
\end{eqnarray}
By the definition, we have
\begin{equation}\label{equation:eq3.3}
\begin{aligned}
&\mathcal{A}\mathcal{Q}\mu+\mathcal{B}^t\mathcal{U}\mu=-\mathcal{C}^t\mu,\\
&\mathcal{B}\mathcal{Q}\mu+\mathcal{R}\mathcal{U}\mu=-\mathcal{S}_1\mu.
\end{aligned}
\end{equation}

Given $\mu \in M$, $\mathcal{Q}{\mu}$ and $\mathcal{U}{\mu}$ can be computed on
each element $K$ independently by using the value of $\mu$ on
$\partial K$.  By \EQ{eq3.3}, we have, for all $\mu\in M$ and
${\bf r}\in {\bf V}$, 
\begin{equation}\label{equation:eq3.4}
\begin{aligned}
&(\epsilon^{-1}\mathcal{Q}\mu,{\bf r})_K-(\mathcal{U}\mu,\nabla\cdot{\bf r})_K=-\langle\mu,{\bf r}\cdot{\bf n}\rangle_{\partial K},\\
&(\nabla\cdot{\mathcal{Q}\mu},w)_K+\langle(\tau_K-\frac12{\bm\beta}\cdot{\bf
  n})(\mathcal{U}\mu-\mu),w\rangle_{\partial K}+(-\frac12 \nabla\cdot{\bm\beta}\mathcal{U}\mu,w)_{K}\\
&\qquad-\frac12 ({\bm\beta}\cdot\nabla w,\mathcal{U}\mu)_K+\frac12({\bm\beta}\cdot\nabla\mathcal{U}\mu,w)_{K}
=-\frac12\langle{\bm\beta}\cdot{\bf n}\mu, w \rangle_{\partial K}.
\end{aligned}
\end{equation}

After eliminating ${\bf q}_h$ and $u_h$ in each element, we could
obtain a system for $\lambda_h$. Once the solution $\lambda_h$ of
\EQ{eq3.1} is found, the solution can be completed by computing ${\bf
  q}_h$ and $u_h$  in each element with the given $\lambda_h$.  The
system of  $\lambda_h$ can be characterized by the following theorem,
which is \cite[Theorem 3.1]{fu_analysis_2015-1}.  The proof can be
found in \cite[Lemma  3.2]{nguyen_implicit_2009}.
 \begin{theorem} \label{lemma:lambdaeqn} Under Assumption \ref{assumption:beta0}, 
 $({\bf q}_h,u_h,\lambda_h)\in {\bf V}\times W\times \Lambda$ is the
 solution of
 $\EQ{eq3.1}$ if and only if $\lambda_h\in \Lambda$ is the unique
 solution of  
 \begin{equation}\label{equation:eq3.5}
 a_h(\lambda_h,\mu)=b_h(\mu),
\end{equation}
 for all $\mu\in \Lambda$ and
\begin{equation}\label{eq3.6}
\begin{aligned}
 {\bf q}_h&=\mathcal{Q}\lambda_h+\mathcal{Q}_{w}F_h,\\
 u_h&=\mathcal{U}\lambda_h+\mathcal{U}_wF_h,
\end{aligned}
\end{equation}
 where 
 \begin{align}\label{equation:adef}
 a_h(\lambda,\mu)&=(\epsilon^{-1}\mathcal{Q}\lambda,\mathcal{Q}\mu)_{\mathcal{T}_h}+\langle(\tau_K(\mathcal{U}\lambda-\lambda),\mathcal{U}{\mu}-\mu\rangle_{\partial\mathcal{T}_h}\\
 &\quad+(-\nabla\cdot{\bm
   \beta}\mathcal{U}\lambda,\mathcal{U}\mu)_{\mathcal{T}_h}-({\bm
   \beta}\mathcal{U}\lambda,\nabla\mathcal{U}\mu)_{\mathcal{T}_h}+\langle{\bm\beta}\cdot{\bf n}\lambda,\mathcal{U}\mu\rangle_{\partial\mathcal{T}_h},\nonumber\\
 b_{h}(\mu)&=-\left((
             F_h,\mathcal{U}{\mu})_{\mathcal{T}_h}+\langle{\bm\beta}\cdot{\bf
             n}\mathcal{U}_{w}F_h,\mu\rangle_{\partial\mathcal{T}_h}
 +(\mathcal{U}_wF_h,{\bm\beta}\cdot\nabla\mathcal{U}\mu)_{\mathcal{T}_h}-({\bm\beta}\cdot\nabla\mathcal{U}_wF_h,\mathcal{U}\mu)_{\mathcal{T}_h}\right).\nonumber
 \end{align}
 \end{theorem}

By the definitions of the bilinear forms $D$, defined in \EQ{Bform}, and $a_h$, we have, 
\begin{equation}\label{equation:Daeqn}
a_h(\lambda,\mu)=D (\left(\mathcal{Q}\lambda,
  \mathcal{U}\lambda, \lambda\right),
\left({\bf r}, w, \mu\right)), 
\end{equation}
for all $({\bf r},w,\mu)\in \vvecV\times W\times \Lambda$.

Since $a_h$ is nonsymmetric,  it is useful  to write $a_h$ as the sum
of the symmetric and skew-symmetric parts.  In order to do this, by integration by part, we can rewrite
\begin{align*}
({\bm
   \beta}\mathcal{U}\lambda,\nabla\mathcal{U}\mu)_{\mathcal{T}_h}&=
\frac12({\bm
   \beta}\mathcal{U}\lambda,\nabla\mathcal{U}\mu)_{\mathcal{T}_h}+\frac12 ({\bm
   \beta}\mathcal{U}\lambda,\nabla\mathcal{U}\mu)_{\mathcal{T}_h}\\
&=\frac12({\bm
   \beta}\mathcal{U}\lambda,\nabla\mathcal{U}\mu)_{\mathcal{T}_h}+\frac12\langle{\bm\beta}\cdot{\bf n}\U\lambda,\mathcal{U}\mu\rangle_{\partial\mathcal{T}_h}
-\frac12(\nabla\cdot \vbeta \U\lambda,\U\mu)_{\mathcal{T}_h}-\frac12(
  \vbeta\cdot \nabla \U\lambda,\U\mu)_{\mathcal{T}_h}.
\end{align*}

Plug the above equation to \EQ{adef} and we obtain
\begin{align*}
 a_h(\lambda,\mu)&=(\epsilon^{-1}\mathcal{Q}\lambda,\mathcal{Q}\mu)_{\mathcal{T}_h}+\langle\tau_K(\mathcal{U}\lambda-\lambda),\mathcal{U}{\mu}-\mu\rangle_{\partial\mathcal{T}_h}\\
 &\quad+\left(-\nabla\cdot{\bm
   \beta}\mathcal{U}\lambda,\mathcal{U}\mu\right)_{\mathcal{T}_h}+\langle{\bm\beta}\cdot{\bf
   n}\lambda,\mathcal{U}\mu\rangle_{\partial\mathcal{T}_h}\\
&\quad -\left(\frac12({\bm
   \beta}\mathcal{U}\lambda,\nabla\mathcal{U}\mu)_{\mathcal{T}_h}+\frac12\langle{\bm\beta}\cdot{\bf n}\U\lambda,\mathcal{U}\mu\rangle_{\partial\mathcal{T}_h}
-\frac12(\nabla\cdot \vbeta \U\lambda,\U\mu)_{\mathcal{T}_h}-\frac12(
  \vbeta\cdot \nabla \U\lambda,\U\mu)_{\mathcal{T}_h}\right)\\
&=(\epsilon^{-1}\mathcal{Q}\lambda,\mathcal{Q}\mu)_{\mathcal{T}_h}+\langle(\tau_K-\frac12\vbeta\cdot
  \vvecn)(\mathcal{U}\lambda-\lambda),\mathcal{U}{\mu}-\mu\rangle_{\partial\mathcal{T}_h}\\
 &\quad+(-\frac12\nabla\cdot{\bm \beta}\mathcal{U}\lambda,\mathcal{U}\mu)_{\mathcal{T}_h}+\frac12({\bm \beta}\cdot\nabla\mathcal{U}\lambda,\mathcal{U}\mu)_{\mathcal{T}_h}-\frac12(
 \mathcal{U}\lambda,{\bm\beta}\cdot\nabla\mathcal{U}\mu)_{\mathcal{T}_h}\\
 &\quad-\frac12\langle{\bm\beta}\cdot{\bf n}\mathcal{U}\lambda,\mu\rangle_{\partial\mathcal{T}_h}+\frac12\langle{\bm\beta}\cdot{\bf n}\lambda,\mathcal{U}\mu\rangle_{\partial \mathcal{T}_h}+\frac12\langle{\bm\beta}\cdot{\bf n}\lambda,\mu\rangle_{\partial \mathcal{T}_h}.
\end{align*}


For $\lambda$ which is zero on $\partial\Omega$,
$\frac12\langle{\bm\beta}\cdot{\bf n}\lambda,\mu\rangle_{\partial
  \mathcal{T}_h}=0$ and the symmetric and skew-symmetric parts
of $a_h(\lambda, \mu)$ are denoted  by
\begin{eqnarray}
\label{equation:sbformK} b_h(\lambda,\mu) &=&(\epsilon^{-1}\mathcal{Q}\lambda,\mathcal{Q}\mu)_{\mathcal{T}_h}+\langle(\tau_K-\frac12\vbeta\cdot
  \vvecn)(\mathcal{U}\lambda-\lambda),\mathcal{U}{\mu}-\mu\rangle_{\partial\mathcal{T}_h}\\
&&+(-\frac12\nabla\cdot{\bm \beta}\mathcal{U}\lambda,\mathcal{U}\mu)_{\mathcal{T}_h}, \nonumber\\[0.5ex]
\label{equation:sbformN} z_h(\lambda,\mu) &=&\frac12({\bm \beta}\cdot\nabla\mathcal{U}\lambda,\mathcal{U}\mu)_{\mathcal{T}_h}-\frac12(
 \mathcal{U}\lambda,{\bm\beta}\cdot\nabla\mathcal{U}\mu)_{\mathcal{T}_h}\\
 &&-\frac12\langle{\bm\beta}\cdot{\bf n}\mathcal{U}\lambda,\mu\rangle_{\partial\mathcal{T}_h}+\frac12\langle{\bm\beta}\cdot{\bf n}\lambda,\mathcal{U}\mu\rangle_{\partial \mathcal{T}_h}.\nonumber
\end{eqnarray}

Let $\q$, $u$, and $\lambda$  denote the unknowns
associated with $\q_h$, $u_h$, and $\lambda_h$,
respectively.  The matrix form of \EQ{eq3.1} can be written as 
 \begin{equation}\label{equation:uplmatrix}
 \left[ \begin{array}{ccc}
 A_{\q\q} &  A^T_{u\q }  &  A^T_{ \lambda\q}    \\
 A_{u\q } &  A_{u u}  &  A_{ u\lambda}    \\
 A_{\lambda\q} &  A_{\lambda u} &  A_{\lambda \lambda}  
 \end{array} 
 \right] \left[
 \begin{array}{c}
 \q        \\
 u     \\
 \lambda     
 \end{array}
 \right] = \left[ 
 \begin{array}{l}
 \vvec{0}        \\
 F_h\\
 0
 \end{array}
 \right].
\end{equation} 
Eliminating $\q$ and $u$ in each element independently
from \EQ{uplmatrix}, we obtain
the system for $\lambda$ only
 \begin{equation}\label{equation:lmatrix}
 A\lambda=b,
\end{equation} 
where
$$A=A_{\lambda\lambda}-[A_{\lambda\q}~A_{\lambda u}]\left[ \begin{array}{cc}
 A_{\q\q} &  A^T_{u \q }    \\
 A_{u \q } &  A_{u u}\end{array} 
 \right]^{-1} \left[\begin{array}{c} A_{\lambda\q}^T\\A_{u \lambda} \end{array}\right]$$
and
$$b=-[A_{\lambda\q}~A_{\lambda u}]\left[ \begin{array}{cc}
 A_{\q\q} &  A^T_{ u\q}    \\
 A_{u\q } &  A_{u u}\end{array} 
 \right]^{-1} \left[\begin{array}{c} {\bf
       0}\\F_h\end{array}\right].$$
We also denote $A=B+Z$, where $B$ and $Z$ are symmetric and
skew-symmetric parts of $A$, corresponding to the bilinear forms $b_h$
and $z_h$ defined in \EQ{sbformK} and \EQ{sbformN}, respectively.

In next two sections, we will develop a BDDC algorithm to solve the
system in  \EQ{lmatrix}  for the numerical trace $\lambda$.

\section{Domain decomposition and a reduced subdomain interface problem }

The domain $\Omega$ is decomposed  into $N$
 nonoverlapping subdomains 
$\Omega_i$ with diameters $H_i$, $i=1,\cdots,N$, and set $H=\max_iH_i$. 
Each subdomain is  assumed to be a union of  shape-regular coarse 
triangles and that 
the number of such triangles forming an individual subdomain 
is uniformly bounded. 
The subdomain interface is denoted by  $\Gamma$. $\Gamma_h$ denotes  the 
set of
the interface nodes and can be defined as 
$\Gamma_h =\left(\cup_{i\ne j}\left(\partial \Omega_{i,h}\cap\partial
  \Omega_{j,h}\right)\right) \setminus \partial \Omega_h$,  where $\partial
\Omega_{i,h} $ is the set of nodes on $\partial \Omega_i$ and
$\partial \Omega_{h} $ is the set of nodes on $\partial \Omega$. 
The spaces of finite element functions on $\Omega_i$ are denoted by ${\bf
  V}^{(i)}$, $W^{(i)}$ and $\Lambda^{(i)}$. The local
bilinear forms are defined on ${\bf
  V}^{(i)}\times W^{(i)}\times \Lambda^{(i)}$ and $\Lambda^{(i)}$ as
\begin{align}\label{equation:locBform}
&D^{(i)}(({\bf q}^{(i)},u^{(i)},\lambda^{(i)}),({\bf r}^{(i)},w^{(i)},\mu^{(i)}))\\
&=(\epsilon^{-1}{\bf q}^{(i)}, {\bf r}^{(i)})_{\T_h(\Omega_i)}-(u^{(i)}, \nabla\cdot {\bf r}^{(i)} ) _{\T_h(\Omega_i)}+\langle\lambda^{(i)}, {\bf r}^{(i)}\cdot {\bf n}\rangle_{\partial\T_h(\Omega_i)}\nonumber\\
&\quad -({\bf q}^{(i)}+{\bm\beta}u^{(i)},\nabla w^{(i)}) _{\T_h(\Omega_i)}+\langle ({\bf q}^{(i)}+{\bm\beta\lambda}^{(i)})\cdot{\bf n}+\tau_K(u^{(i)}-\lambda^{(i)}),w^{(i)} \rangle_{\partial\T_h(\Omega_i)}-(\nabla\cdot{\bm\beta}u^{(i)},w^{(i)}) _{\T_h(\Omega_i)}\nonumber\\
&\quad-\langle({\bf q}^{(i)}+{\bm \beta}\lambda^{(i)})\cdot{\bf n}+\tau_K(u^{(i)}-\lambda^{(i)}),\mu^{(i)}\rangle_{\partial\T_h(\Omega_i)}\nonumber
\end{align}
and
\begin{align}\label{equation:locah}
 a^{(i)}_h(\lambda^{(i)},\mu^{(i)})
&=(\epsilon^{-1}\mathcal{Q}\lambda^{(i)},\mathcal{Q}\mu^{(i)})_{\T_h(\Omega_i)}+\langle(\tau_K-\frac12\vbeta\cdot
  \vvecn)(\mathcal{U}\lambda^{(i)}-\lambda^{(i)}),\mathcal{U}{\mu^{(i)}}-\mu^{(i)}\rangle_{\partial\T_h(\Omega_i)}\nonumber\\
 &\quad+(-\frac12\nabla\cdot{\bm \beta}\mathcal{U}\lambda^{(i)},\mathcal{U}\mu^{(i)})_{\T_h(\Omega_i)}+\frac12({\bm \beta}\cdot\nabla\mathcal{U}\lambda^{(i)},\mathcal{U}\mu^{(i)})_{\partial\T_h(\Omega_i)}-\frac12(
 \mathcal{U}\lambda^{(i)},{\bm\beta}\cdot\nabla\mathcal{U}\mu^{(i)})_{\T_h(\Omega_i)}\nonumber\\
 &\quad-\frac12\langle{\bm\beta}\cdot{\bf
   n}\mathcal{U}\lambda^{(i)},\mu^{(i)}\rangle_{\partial
   \T_h(\Omega_i)}+\frac12\langle{\bm\beta}\cdot{\bf
   n}\lambda^{(i)},\mathcal{U}\mu^{(i)}\rangle_{\partial
   \T_h(\Omega_i)}-\frac12\langle{\bm\beta}\cdot{\bf
   n}\lambda^{(i)},\mu^{(i)}\rangle_{\partial\T_h(\Omega_i)},\nonumber
\end{align}
where $(\cdot,\cdot)_{\T_h(\Omega_i)}=\sum_{K\in \T_h, K\subset \Omega_i}(\cdot,\cdot)_K$
and $\left<\cdot,\cdot\right>_{\partial\T_h(\Omega_i)}=\sum_{K\in
  \T_h, K\subset \Omega_i}\left<\cdot,\cdot\right>_{\partial K}$. 

Similarly to \EQ{Daeqn}, by the definitions, we have, 
\begin{equation}\label{equation:locDaeqn}
a^{(i)}_h(\lambda^{(i)},\mu^{(i)})=D^{(i)} \left(\left(\mathcal{Q}\lambda^{(i)},
  \mathcal{U}\lambda^{(i)}, \lambda^{(i)}\right),
\left({\bf r}^{(i)}, w^{(i)}, \mu^{(i)}\right)\right), 
\end{equation}
for all $({\bf r}^{(i)},w^{(i)},\mu^{(i)})\in \vvecV^{(i)}\times W^{(i)}\times \Lambda^{(i)}$.

Due to the sign of the boundary integral
$\frac12\langle{\bm\beta}\cdot{\bf
  n}\lambda^{(i)},\mu^{(i)}\rangle_{\partial\T_h(\Omega_i)} $ depending on $\vbeta$
and $\vvec{n}$, the 
bilinear form $a_h^{(i)}(\cdot, \cdot)$ cannot be ensured to be positive
definite on $\Lambda^{(i)}$.  As in
\cite{Achdou:2000:DDAD,TuLi:2008:BDDCAD}, we use the Robin boundary condition on $\partial\Omega_i$
and introduce
\begin{equation}
\label{equation:lbformA} a^{(i)}(\lambda^{(i)},\mu^{(i)}) =
a_h^{(i)}(\lambda^{(i)},\mu^{(i)})+\frac12\langle{\bm\beta}\cdot{\bf
  n}\lambda^{(i)},\mu^{(i)}\rangle_{\partial\T_h(\Omega_i)}.
\end{equation}
 The modified local bilinear
forms $a^{(i)}(\cdot, \cdot)$ are ensured to be positive definite on $\Lambda^{(i)}$, $i
= 1, 2. \dots, N$ under the assumption \EQ{cond1}. This modification
will not change the global solution since $\lambda^{(i)}$ and
$\mu^{(i)}$ are continuous at element boundaries and  the boundary
integrals on the subdomains will
cancel out between neighboring subdomains. 
We can write the symmetric and skew-symmetric parts of $a^{(i)}(\lambda^{(i)},\mu^{(i)})$
as
\begin{eqnarray}
\label{equation:lbformK} \qquad b^{(i)}(\lambda^{(i)},\mu^{(i)}) & = &(\epsilon^{-1}\mathcal{Q}\lambda^{(i)},\mathcal{Q}\mu^{(i)})_{\T_h(\Omega_i)}+\langle(\tau_K-\frac12\vbeta\cdot
  \vvecn)(\mathcal{U}\lambda^{(i)}-\lambda^{(i)}),\mathcal{U}{\mu^{(i)}}-\mu^{(i)}\rangle_{\partial\T_h(\Omega_i)}\nonumber\\
 &&\quad+(-\frac12\nabla\cdot{\bm \beta}\mathcal{U}\lambda^{(i)},\mathcal{U}\mu^{(i)})_{\T_h(\Omega_i)}
, \\
\label{equation:lbformN} \qquad z^{(i)}(\lambda^{(i)},\mu^{(i)}) &=&
\frac12({\bm \beta}\cdot\nabla\mathcal{U}\lambda^{(i)},\mathcal{U}\mu^{(i)})_{\T_h(\Omega_i)}-\frac12(
 \mathcal{U}\lambda^{(i)},{\bm\beta}\cdot\nabla\mathcal{U}\mu^{(i)})_{\T_h(\Omega_i)}\nonumber\\
 &&\quad-\frac12\langle{\bm\beta}\cdot{\bf n}\mathcal{U}\lambda^{(i)},\mu^{(i)}\rangle_{\partial\T_h(\Omega_i)}+\frac12\langle{\bm\beta}\cdot{\bf n}\lambda^{(i)},\mathcal{U}\mu^{(i)}\rangle_{\partial\T_h(\Omega_i)}.
\end{eqnarray}

In order to  reduce the global problem \EQ{lmatrix} to a subdomain interface
problem, we  decompose $\Lambda$ into the subdomain interior and interface
parts.  Let $\Lhat_{\Gamma}$ denote the degrees of freedom  associated with $\Gamma$
and $\Lambda_I$ be a direct sum of subdomain interior degrees of
freedom $\Lambda_I=\bigoplus_{i=1}^{N}\Lambda_I^{(i)}$. We can write
\[
\Lambda=\Lambda_I\bigoplus\Lhat_{\Gamma}
\]
and
the global problem \EQ{lmatrix} can be written as $\lambda_I\in
\Lambda_I$ and $\Lambda_\Gamma\in \Lhat_\Gamma$,
\begin{equation}\label{equation:ll}
\left[
\begin{array}{cc} 
A_{II} & A_{I\Gamma}\\
A_{I\Gamma}^T &A_{\Gamma\Gamma}
\end{array} \right] \left[
\begin{array}{c} 
\lambda_I\\
\lambda_{\Gamma}
 \end{array}
\right] = 
\left[ \begin{array}{l} 
 b_I \\
 b_\Gamma
\end{array} \right]. \end{equation}

We also denote the subdomain local interface numerical trace  space by
$\Lambda_{\Gamma}^{(i)}$. 
We use the restriction operator $R_\Gamma^{(i)}$ to map functions in $\Lhat_\Gamma $ to their
subdomain components in the space $\Lambda_\Gamma^{(i)}$. The direct
sum of $R^{(i)}_\Gamma$ is denoted by $R_\Gamma$.

The
subdomain problems can be written as 
\begin{equation}
\label{equation:dissubupl} 
A^{(i)}\lambda^{(i)}=b^{(i)},
\end{equation} 
where
$$A^{(i)}=\left[
\begin{array}{cc} 
A^{(i)}_{II} & A^{(i)}_{I\Gamma}\\
A^{{(i)}^T}_{I\Gamma} &A^{(i)}_{\Gamma\Gamma}
\end{array} \right],\quad \lambda^{(i)}=\left[
\begin{array}{c} 
\lambda^{(i)}_I\\
\lambda^{(i)}_{\Gamma}
 \end{array}
\right]\in 
\Lambda^{(i)}_I\times 
\Lambda^{(i)}_{\Gamma}, \quad \mbox{and} \quad b^{(i)} = 
\left[ \begin{array}{l} 
 b^{(i)}_I \\
 b^{(i)}_\Gamma
\end{array} \right].$$ 
 $A^{(i)}$ is corresponding to  the subdomain bilinear form
$a^{(i)}$, defined in \EQ{lbformA}.

We define the subdomain Schur complement $S^{(i)}_\Gamma$ by:
given $\lambda^{(i)}_\Gamma \in \Lambda^{(i)}_\Gamma$, find
$S^{(i)}_\Gamma \lambda^{(i)}_\Gamma$ 
 such that
\begin{equation}
\label{equation:dissubschur} \left[
\begin{array}{cc} 
A^{(i)}_{II} & A^{(i)}_{I\Gamma}\\
A^{{(i)}^T}_{I\Gamma} &A^{(i)}_{\Gamma\Gamma}
\end{array} \right] \left[
\begin{array}{c} 
\lambda^{(i)}_I\\
\lambda^{(i)}_{\Gamma}
 \end{array}
\right] = 
\left[ \begin{array}{l} 
 0 \\
 S_\Gamma^{(i)}\lambda_\Gamma
\end{array} \right]. \end{equation}
We can assemble the subdomain local Schur complement $S^{(i)}_\Gamma$
to obtain the global interface problem: find $\lambda_{\Gamma}\in 
\Lhat_\Gamma$, such that
\begin{equation}
\label{equation:continuousinterface} 
\widehat{S}_{\Gamma}\lambda_{\Gamma} =b_{\Gamma},
\end{equation}
where $b_\Gamma = \sum_{i=1}^N R_\Gamma^{(i)^T}
b_\Gamma^{(i)}$, and
 \begin{equation}
\label{equation:SGammahat} \widehat{S}_\Gamma 
= \sum_{i=1}^N R_{\Gamma}^{(i)^T}
S_\Gamma^{(i)} R_{\Gamma}^{(i)}. \end{equation}
We note that $\widehat{S}_{\Gamma}$ is defined on the interface space
$\Lhat_\Gamma$. It is non-symmetric but 
 positive definite. A
 BDDC preconditioned generalized minimal residual method (GMRES)
 \cite{Saad:1986:GGM} will be proposed in next section to solve \EQ{continuousinterface}.

\section{The BDDC
  preconditioner}

In order to introduce the BDDC preconditioner, we further decompose
the subdomain interface variables to the primal and 
remaining variables, denoted  by $\Lhat_{\Pi}$ and $\Lambda_{\Delta}$,
 respectively. The primal variable $\Lhat_{\Pi}$ is spanned by subdomain  
interface edge/face basis
functions with constant values at the nodes of the edge/face for 
two/three  dimensions. These variables can be explicitly presented  by a change of variables, see  \cite{Li:2004:FBC} and
\cite{Klawonn:2004:FDP}.  The space $\Lambda_{\Delta}$ is the direct sum of the
$\Lambda_{\Delta}^{(i)}$, which have
a zero average over each edge/face.

We define  $\Ltilde_{\Gamma}$, the partially assembled interface
space, as  
$$
\Ltilde_{\Gamma} = \Lhat_{\Pi} \bigoplus
\Lambda_{\Delta} = \Lhat_{\Pi} \bigoplus \left(
\prod_{i=1}^N \Lambda^{(i)}_\Delta \right).
$$

In order to present the BDDC preconditioner, we also need  some restriction, extension, and scaling operators
between  different spaces. 
The restriction operators $\widehat{R}_{\Gamma}^{(i)}$ and $\overline{R}_{\Gamma}^{(i)}$ restricts functions 
in $\Lhat_\Gamma$ and $\Ltilde_\Gamma$ to $\Lambda_{\Gamma}^{(i)}$,
respectively,  while   $R^{(i)}_\Delta$  
 maps functions from $\Lhat_\Gamma$
to $\Lambda_\Delta^{(i)}$ and
  $R_{\Gamma \Pi}$ restricts  $\Lhat_\Gamma$ to its
subspace $\Lhat_\Pi$.
$\overline{R}_{\Gamma}$ is
 the direct sum of the $\overline{R}_{\Gamma}^{(i)}$ and
 $\Rtilde_\Gamma$ is the direct sum of 
$R_{\Gamma \Pi}$ and  
$R_\Delta^{(i)}$.
Recall $\rho=\epsilon^{-1}$ and we assume that the $\rho_i(x)$ is 
 constant in each
subdomain.  Let ${\cal N}_x$ be the set of indices $j$ of the subdomains
 such that $x\in \partial \Omega_j$. A positive scaling factor
 $\delta_i^{\dagger}(x)$ can be defined as follows:
 for $\gamma\in [1/2, \infty)$,
$$\delta^{\dagger}_i(x)=\frac{\rho_i^{\gamma}(x)}{\sum_{j\in {\cal N}_{x}}\rho^{\gamma}_j(x)},
 ~~~~x\in \partial \Omega_{i,h} \cap \Gamma_h.$$
We note that $\delta^{\dagger}_i(x)$ 
is  constant on each edge/face.
 $R^{(i)}_{D,\Delta}$  is given by multiplying each row of  $R^{(i)}_{\Delta}$, with the
scaling factor $\delta_i^{\dagger}(x)$ and the direct sume of
$R_{\Gamma\Pi}$ and $R^{(i)}_{D,\Delta}$ gives $\Rtilde_{D,\Gamma}$. 

We also denote by  $\FGtilde$, the right hand side space
corresponding to  $\Ltilde_{\Gamma}$. The same
 restriction, extension,  and scaled 
restriction operators will be used for 
the  space $\FGtilde$ as  for $\Ltilde_{\Gamma}$.

With the help of $\overline{R}_{\Gamma}$, we can define 
the interface numerical trace  Schur complement $\Stilde_\Gamma$, on the
partially assembled interface numerical trace space
$\Ltilde_\Gamma$, as 
\begin{equation}
\label{equation:SGammatilde1} \Stilde_\Gamma =
\overline{R}_{\Gamma}^T S_\Gamma \overline{R}_{\Gamma},
\end{equation}
where $S_\Gamma$ is a direct sum of $S^{(i)}_\Gamma$, the subdomain
local Schur complement defined in \EQ{dissubschur}.

The BDDC preconditioner is 
\begin{equation}
\label{equation:precond2} M^{-1} =
\Rtilde_{D,\Gamma}^T\Stilde_{\Gamma}^{-1}   \Rtilde_{D,\Gamma}.
 \end{equation}

We have obtained the preconditioned BDDC algorithm for solving the global interface 
 problem \EQ{continuousinterface}: find $\lambda_\Gamma\in \Lhat_\Gamma$, such that
\begin{equation}
\label{equation:BDDC} \Rtilde_{D,\Gamma}^T \Stilde_{\Gamma}^{-1} 
\Rtilde_{D,\Gamma}
\widehat{S}_{\Gamma}\lambda_{\Gamma} = \Rtilde_{D,\Gamma}^T \Stilde_{\Gamma}^{-1} 
\Rtilde_{D,\Gamma}b_{\Gamma}.
\end{equation}
A GMRES method is used to solve~\EQ{BDDC} since the preconditioned problem is
positive definite but nonsymmetric.  In each
iteration, one needs to multiply $\widehat{S}_\Gamma$  and
$\Stilde_\Gamma^{-1}$ by a vector, which require 
solving subdomain Dirichlet
boundary problems, Robin boundary problems, and  a coarse level
problem, see ~\cite{Li:2004:FBC} for details. After obtaining the interface
solution $\lambda_\Gamma$, the solution can be completed by solving subdomain Dirichlet
problems to find $\lambda_I$.

\section{A partial assembled finite element space, bilinear forms, and
norms}
In order to analyze the convergence of the GMRES algorithm for solving
\EQ{BDDC}, in this section, we introduce several bilinear forms and
corresponding norms. 

Similarly to \cite{TuLi:2008:BDDCAD}, we first introduce a partially sub-assembled finite element space
$\Ltilde$, which  is defined as 
$$
\Ltilde = \Lhat_{\Pi} \bigoplus
\Lambda_{r} = \Lhat_{\Pi} \bigoplus \left(
\prod_{i=1}^N \left(\Lambda^{(i)}_I\bigoplus\Lambda^{(i)}_\Delta \right)\right).
$$
 By the definition, we can see that functions in the space $\Ltilde$ have a
continuous primal part but in general a discontinuous
dual part.
$\Lambda\subset\Ltilde$ and the injection operator
from $\Lambda$ to $\Ltilde$ is denoted by $\Rtilde$.

The bilinear forms $\widetilde{a}_h(\lambda, \mu)$,
$\widetilde{a}(\lambda, \mu)$, $\widetilde{b}(\lambda, \mu)$, and $\widetilde{z}(\lambda, \mu)$  are defined on
$\Ltilde$ as: $\forall ~\lambda, \mu\in\Ltilde$, 
\[
\widetilde{a}_h(\lambda, \mu) = \sum_{i = 1}^N a_h^{(i)}(\lambda^{(i)},\mu^{(i)}),
\quad \widetilde{a}(\lambda, \mu)  = \sum_{i = 1}^N a^{(i)}(\lambda^{(i)},\mu^{(i)}),
\]
and
\[
\widetilde{b}(\lambda, \mu)  = \sum_{i = 1}^N b^{(i)}(\lambda^{(i)},\mu^{(i)}), \quad
\widetilde{z}(\lambda, \mu)  = \sum_{i = 1}^N z^{(i)}(\lambda^{(i)},\mu^{(i)}).
\]
where $\lambda^{(i)}$ and $\mu^{(i)}$ represent restrictions of $\lambda$ and $\mu$
to subdomain $\Omega_i$.  $a_h^{(i)}$,  $a^{(i)}$,  $b^{(i)}$, 
$z^{(i)}$ are defined in Section 3. 

$\Atilde$, $\Btilde$,
and $\Ztilde$  denote the partially sub-assembled matrices corresponding to the
bilinear forms $\widetilde{a}(\cdot, \cdot)$, $\widetilde{b}(\cdot,
\cdot)$, and $\widetilde{z}(\cdot, \cdot)$, respectively. We have
\[
A=\Rtilde^T\Atilde\Rtilde,\quad
B=\Rtilde^T\Btilde\Rtilde,\quad\mbox{and }
Z=\Rtilde^T\Ztilde\Rtilde.
\]
Recall $A$, $B$, and $Z$ are defined at the end of Section 2.

We know that  the subdomain bilinear forms
$b^{(i)}(\cdot,\cdot)$, $i = 1, 2, ..., N$, are symmetric positive
semi-definite on $\Lambda^{(i)}$ and define a semi-norm $\|\lambda^{(i)}\|_{B^{(i)}}^2 =
b^{(i)}(\lambda^{(i)},\lambda^{(i)})$.  For any $\lambda \in \Lambda$, we define $\|\lambda\|^2_{B}= \sum_{i=1}^N
\|\lambda^{(i)}\|_{B^{(i)}}^2$ and 
$\|w\|^2_\Btilde= \sum_{i=1}^N \|w^{(i)}\|_{B^{(i)}}^2$,
for
any $w \in \Ltilde$, where $\lambda^{(i)}$ and $w^{(i)}$ are the
restrictions of $\lambda$ and $w$ to Subdomain $\Omega_i$, respectively.
We note that the $\|\cdot\|_B$  and  $\|\cdot\|_{\Btilde}$ are norms
defined for $\Lambda$ and $\Ltilde$, respectively.


Given  $\lambda_\Gamma \in
\Ltilde_\Gamma$,
we define a discrete extension of $\lambda_\Gamma$ to the interior of subdomains by
\begin{equation}
\label{equation:extension} \lambda_{{\A}, \Gamma} = \left[
\begin{array}{c} -
A_{II}^{-1} \Atilde_{I \Gamma} \lambda_\Gamma  \\
\lambda_{\Gamma} \end{array} \right] \in \Ltilde.
\end{equation}
By the definition, we know that $\lambda_{{\A}, \Gamma}$ is obtained
from $\lambda_\Gamma $ by
solving subdomain advection-diffusion problems with Dirichlet
boundary conditions. 
We note that  $\lambda_{{\A}, \Gamma}$ are also
well defined for $\lambda_\Gamma \in \Lhat_\Gamma$, and as a result
$\lambda_{{\A}, \Gamma} \in \Lambda$. However, $\lambda_{{\A}, \Gamma}$ does not have any
energy minimization property as usual harmonic extension from symmetric
elliptic problems. 

Using the extension, we can define two bilinear forms for vectors in $\Lhat_\Gamma$ and
$\Ltilde_\Gamma$ respectively by
\begin{eqnarray}
\label{equation:productS} & & \quad
\left<\lambda_\Gamma,\mu_\Gamma\right>_{B_\Gamma}=\mu_{{\A}, \Gamma}^T B
\lambda_{{\A}, \Gamma}, ~ \mbox{ and }~  \left<\lambda_\Gamma,\mu_\Gamma
\right>_{Z_\Gamma}= \mu_{{\A}, \Gamma}^T Z \lambda_{{\A}, \Gamma}, \quad
\forall ~ \lambda_\Gamma, \mu_\Gamma \in \Lhat_\Gamma, \\
\label{equation:tildeproductS} & & \quad
\left<\lambda_\Gamma,\mu_\Gamma\right>_{\Btilde_\Gamma}=\mu_{{\A}, \Gamma}^T
\Btilde \lambda_{{\A}, \Gamma}, ~ \mbox{ and }~ \left<\lambda_\Gamma,\mu_\Gamma
\right>_{\Ztilde_\Gamma}= \mu_{{\A}, \Gamma}^T \Ztilde \lambda_{{\A},
\Gamma}, \quad \forall ~ \lambda_\Gamma, \mu_\Gamma \in \Ltilde_\Gamma.
\end{eqnarray}
Here $\left< p, q \right>_M$  represents
the product $q^T M p$, for any given matrix $M$ and vectors $p$ and
$q$.

By the definitions ~\EQ{extension}, \EQ{productS}, and
\EQ{tildeproductS}, we have the following lemma, which is similar to
\cite[Lemma 7.2]{TuLi:2008:BDDCAD}.
\begin{lemma}
\label{lemma:null} For any $\mu \in \Ltilde$, denotes its restriction
to $\Gamma$ by $\mu_\Gamma \in \Ltilde_\Gamma$. Then for any
$\lambda_\Gamma\in \Ltilde_\Gamma$ and $\mu \in \Ltilde$, $\left< \lambda_\Gamma,
\mu_\Gamma \right>_{\Stilde_\Gamma} = \left< \lambda_{{\A}, \Gamma}, \mu
\right>_\Atilde$ and $\left< \lambda_\Gamma, \mu_\Gamma
\right>_{\Stilde_\Gamma} = \left< \lambda_\Gamma, \mu_\Gamma
\right>_{\Btilde_\Gamma}+\left< \lambda_\Gamma, \mu_\Gamma
\right>_{\Ztilde_\Gamma}$. For any $\lambda_\Gamma\in \Ltilde_\Gamma$,
$\left< \lambda_\Gamma, \lambda_\Gamma \right>_{\Stilde_\Gamma} = \left<
\lambda_{{\A}, \Gamma}, \lambda_{{\A}, \Gamma} \right>_\Atilde = \left< \lambda_{{\A},
\Gamma}, \lambda_{{\A}, \Gamma} \right>_\Btilde  =
\left<\lambda_\Gamma,\lambda_\Gamma\right>_{\Btilde_\Gamma} \geq 0$, and
$\left<\lambda_\Gamma,\lambda_\Gamma\right>_{\Ztilde_\Gamma}=0.$ The same
results also hold for functions and the corresponding bilinear forms
in the space $\Lhat_\Gamma$.
\end{lemma}

Using Lemma~\LA{null},  for elements in the spaces $\Lhat_\Gamma$ and
$\Ltilde_\Gamma$, we can define $B_\Gamma$- and $\Btilde_\Gamma$-
norms
respectively as: $\|\uG\|^2_{B_\Gamma} = \left< \mu_\Gamma, \mu_\Gamma
\right>_{B_\Gamma}$, for any $\mu_\Gamma \in \Lhat_\Gamma$, and
$\|\wG\|^2_{\Btilde_\Gamma} = \left< w_\Gamma, w_\Gamma
\right>_{\Btilde_\Gamma}$, for any $w_\Gamma \in \Ltilde_\Gamma$.


We also  introduce several useful norms, which are defined in
\cite{schwarz,multigrid,error}. For any domain $D$ and $\lambda\in
\Lambda(D)$,  we denote the $L^2$ norm by $\|\cdot\|_D$ and define 
\begin{equation}\label{equation:L2norm}
\|\lambda\|^2_{h,D}=\sum_{K\in
    \T_h,K\subseteq D}\left\|\lambda\right\|^2_{L^2(\partial
    K)}\frac{|K|}{|\partial K|},
\end{equation} 
where $|K|$ and $|\partial K|$ denote the $n$- and $(n-1)$-dimensional
measures of $K$ and $\partial K$, respectively. 
Define 
\begin{equation}\label{equation:IIInorm}
\interleave\lambda\interleave^2_{D}=\sum_{K\in
  \T_h,K\subseteq D}\left\|\lambda-m_{K}(\lambda)\right\|^2_{L^2(\partial
  K)}\frac{1}{|\partial K|},
\end{equation} 
where 
\begin{equation}\label{equation:IIIm}
m_\kappa=\frac{1}{|\partial K|}\int_{\partial K}\lambda ds.
\end{equation}
When the domain $D$ is $\Omega$, we drop $\Omega$ as a subscript in
notation. Namely we will use $\|\cdot\|_h$ and $\interleave\cdot\interleave$ denote
$\|\cdot\|_{h,\Omega}$ and $\interleave\cdot\interleave_{\Omega}$, respectively. 

\section{Convergence rate of the GMRES iterations}

In our analysis, to make the  notations and proof simpler, we will
focus on two dimensional case. The approach can be extended to three
dimensions as well. 

We need the following assumptions:

\begin{assp}\label{assump:onbeta} Following \cite{AyusoM2009, fu_analysis_2015-1}, we assume    
$\vbeta\in W^{1,\infty}(\Omega)$ satisfies:
\begin{enumerate}
\item $\vbeta$ has no closed curves and $\vbeta\neq 0$, $\forall
  x\in\Omega$;
\item $-\nabla\cdot{\vbeta}\geq 0$, see \EQ{cond1}.
 \item { $$\max_{\E\in \partial K}\inf_{x\in \E}
   \left(\tau_K-\frac12\vbeta(x)\cdot{\bf
       n}\right)=\widetilde{C}_K>0,\quad \forall K\in \T_h,$$  see Assumption \ref{assumption:beta0};}

\item $\inf_{x\in \E} \left(\tau_K-\frac12\vbeta(x)\cdot \vvecn\right) \ge C_1^* \max_{x\in \E} |\vbeta(x)\cdot\vvecn|$, $\forall
\E\in \partial K$ and $\forall K\in \T_h$;
\end{enumerate}  
\end{assp}

\begin{assp}\label{assump:onWtilde} For each face $\E^k$, 
the coarse level primal subspace $\Lhat_{\Pi}$ contains one
face average, one face flux weighted average, and
one face flux weighted first moment  degrees of
freedom such that for any $w \in \Wtilde$,
$$
\int_{\E^k} w^{(i)} ~ds , \quad \int_{\E^k}
\vbeta\cdot\vvec{n}w^{(i)} ~ds ,  \quad \mbox{and} \quad
\int_{\E^k} \vbeta\cdot\vvec{n}w^{(i)}s ~ ds,
$$
respectively, are the same 
(with a difference of factor $-1$ corresponding to opposite
normal directions) for the two subdomains $\Omega_i$ that share
$\E^k$.
\end{assp}

\begin{assp}\label{assump:onsubdomainshape}
Each subdomain $\Omega_i$ is triangular or
quadrilateral in two dimensions.
The
subdomains form a shape regular coarse mesh of $\Omega$.
\end{assp}

Under Assumption \ref{assump:onsubdomainshape}, let $\What_H$ be the
continuous linear 
finite
element space on the coarse subdomain mesh,  and $I_H$ be the
finite element interpolation operator into $\What_H$. 
Bramble-Hilbert lemma,  cf.~\cite[Theorem 2.3]{1xu},  gives the
following result
\begin{lemma}
\label{lemma:WHapproximation} There exists a constant $C$, such that
for all $u \in H^2(\Omega_i)$, $i = 1, 2, ..., N$, $\| u - I_H u
\|_{H^t(\Omega_i)} \leq C H^{2-t} | u |_{H^2(\Omega_i)}$, where $t =
0, 1, 2$.
\end{lemma}

In this paper, we focus on the dependence of the GMRES  convergence
rate on the viscosity $\epsilon$, the subdomain size $H$, and the
mesh size $h$ and always use $c$ and $C$ to
represent  generic positive constants independent of $\epsilon$, $H$,
and $h$.  We consider
$0<\epsilon\le 1$. When $\epsilon$ is close to $1$, the system \EQ{pde1}
is   diffusion dominant,  and when $\epsilon$ is close to $0$, the system \EQ{pde1}
is   advection dominant.

We define two constants which will be used in our analysis:
{
\begin{equation}\label{equation:Cdef}
\begin{aligned}
\nu_{\epsilon,\tau_K,h}&=\max_{K\in \T_h}\left(h\sqrt{\max\limits_{x\in\partial K}(\tau_K-\frac12\vbeta\cdot{\bf n})}+Ch\epsilon^{-1}\|\vbeta\|_{\infty},\right.\\
&\qquad\qquad\left. \frac{C}{\widetilde{C}_K}\left(\widetilde{C}^{\frac12}_K\epsilon\sqrt{\max\limits_{x\in\partial K}(\tau-\frac12\vbeta\cdot{\bf n})}+\|\vbeta\|_{\infty}+h\|\nabla\cdot\vbeta\|_{\infty}\right)\right)
\end{aligned}
\end{equation}
}
and
\begin{equation}\label{equation:gamma}
\gamma_{h,\tau_K}=\max _{K\in \T_h}\left\{1+\tau_Kh_K\right\}
\end{equation} as \cite{multigrid}. Recall $\widetilde{C}_K$ is
defined in Assumption \ref{assump:onbeta}.

We denote the preconditioned operator $\Rtilde_{D,\Gamma}^T
\Stilde_{\Gamma}^{-1} \Rtilde_{D,\Gamma} S_{\Gamma}$
in~\EQ{BDDC} by $T$. The following theorem, proved  in \cite{eielsc}, gives the estimate of the convergence rate of the GMRES
iteration.
\begin{theorem} \label{theorem:EES} Let $c$ and $C^2$
 be two positive parameters such that
\begin{eqnarray}
\label{equation:lowb}
c \left< u , u \right>_\SGS& \leq &  \left< u , T u \right>_\SGS, \\
\label{equation:upb} \left< T  u , T u \right>_\SGS & \leq & C
\left< u , u \right>_\SGS.
\end{eqnarray}
Then
$$
\frac{\| r_m \|_\SGS}{\| r_0 \|_\SGS} \leq \left( 1 -
\frac{c^2}{C^2} \right)^{m/2},
$$
where $r_m$ is the residual  of the GMRES iteration at iteration $m$.
\end{theorem}


The following theorem is our main result for the convergence of the
GMRES iterations for solving \EQ{BDDC}.

\begin{theorem}\label{theorem:maintheorem}
For all
$\mu_{\Gamma}\in\widehat{\Lambda}_{\Gamma}$, we have
\begin{equation}\label{eq3.11}
\begin{aligned}
\langle T\mu_{\Gamma},T\mu_{\Gamma}\rangle_{B_{\Gamma}}\leq CC^2_u\|\mu_{\Gamma}\|^2_{B_{\Gamma}},
\end{aligned}
\end{equation}
\begin{equation}\label{eq3.12}
c_l\langle\mu_{\Gamma},\mu_{\Gamma}\rangle_{B_{\Gamma}}\leq C\frac{\Ctwo^2}{\epsilon}\langle\mu_{\Gamma},T\mu_{\Gamma}\rangle_{B_{\Gamma}},\\
\end{equation}
where $C_u=\frac{\Ctwo^2}{\epsilon}C_{ED}^2$ and
$c_l=1-\frac{\Ctwo^6C_h}{\epsilon^{5/2}}$. $C_{ED}$ and $C_h$ are defined in
Lemmas \LA{averagenorm} and \LA{L2error}, respectively.
\end{theorem}

\begin{remark}\label{remark:onc0}
Compare to the linear finite element discretization, studied in
\cite[Theorem 7.15]{TuLi:2008:BDDCAD}, our upper bound has one more 
$\frac{1}{\epsilon}$ in $C_{ED}$  and two
additional factors $\Ctwo$ and
$\gamma_{h,\tau_K}$, defined in \EQ{Cdef} and \EQ{gamma},
respectively. For the detailed explanation, see Remark
\ref{remark:Ed}.   For the common choices of $\tau_K$,
$\gamma_{h,\tau_K}$ is usually a bounded constant. However, $\Ctwo$
can be quite large if $\epsilon$ is very small unless $h$ is
sufficiently small to balance.  $\Ctwo$ appears in the lower bound as
well. 
For the case $\epsilon = O(1)$, in Theorem~\ref{theorem:maintheorem}, $C_u=\left(1+\log\frac Hh\right)^2\gamma_{h,\tau}$ and
$c_l = 1 - C H \left(1+\log\frac{H}{h}\right)^2$, which will be positive for appropriate choice of $H$ and
$h$. In this case Theorem~\ref{theorem:EES} applies and the convergence rate of our BDDC algorithm is
independent of the number of subdomains and depends on $H/h$
slightly.  Due to a different proof of Lemma \LA{L2error}, our lower bound is better than that in \cite[Theorem
7.15]{TuLi:2008:BDDCAD}, see Remark \ref{remark:error} for details.  The convergence rate deteriorates with a
decrease of $\epsilon$. For the advection-dominated cases,  the
performance of our BDDC algorithms are satisfactory  for lower order
discretization (degree $0$) in our  numerical experiments.  The extra
flux-based constraints used in our BDDC
algorithms improve the performance significantly for degree $1$ and
$2$ discretizations.  However, when $\epsilon$
is very small, in order to make $c_l$ in
Theorem~\ref{theorem:maintheorem} positive,   $H$ and $h$ have to
be tiny which might not be practical. 

\end{remark}

Similarly to the  upper bound estimate for the BDDC algorithm applied
to symmetric positive definite problems, for any function $\mu_\Gamma
\in \Ltilde_\Gamma$, we define the average operator as $E_{D,\Gamma} \mu_\Gamma = \RtildeG
\Rtilde^T_{D,\Gamma} \mu_\Gamma$, which computes an average of $\mu_\Gamma$ across
$\Gamma$. For symmetric positive definite problems, the stability analysis of $E_{D, \Gamma}$ can be found 
in \cite{klaw_wid3, Klawonn:2004:FDP, TW:2016:HDG}. See  \cite{TuLi:2008:BDDCAD} for
nonsymmetric positive definite problems and 
\cite{BDDCH, LiBDDCS,TW:2017:WGS} for symmetric indefinite problems.

We will provide proofs for the following Lemmas \LA{Sgnorm}, \LA{averagenorm}, and
\LA{L2error} in next section.  

\begin{lemma}\label{lemma:Sgnorm}
\begin{eqnarray*}
\langle\lambda_\Gamma,\mu_\Gamma\rangle_{S_\Gamma}&\leq&
                                                         C\frac{\Ctwo^2}{{\epsilon}}
                                                         \|\lambda_\Gamma\|_{B_\Gamma}\|\mu_\Gamma\|_{B_\Gamma},
  \quad \forall \lambda_\Gamma,\mu_\Gamma\in \Lhat_\Gamma,\\
\langle\lambda_\Gamma,\mu_\Gamma\rangle_{\Stilde_\Gamma}&\leq&
                                                         C\frac{\Ctwo^2}{{\epsilon}}
                                                         \|\lambda_\Gamma\|_{\Btilde_\Gamma}\|\mu_\Gamma\|_{\Btilde_\Gamma},
  \quad \forall \lambda_\Gamma,\mu_\Gamma\in \Ltilde_\Gamma,\\
\langle\lambda_\Gamma,\mu_\Gamma\rangle_{Z_\Gamma}&\leq&
                                                         C\frac{\Ctwo^2}{\sqrt{\epsilon}}
                                                         \|\lambda_\Gamma\|_{B_\Gamma}\|\mu_{\A,\Gamma}\|_{h},
  \quad \forall \lambda_\Gamma,\mu_\Gamma\in \Lhat_\Gamma.
\end{eqnarray*}
Recall $\mu_{\A,\Gamma} \in \Lambda$ is the extension of $\mu_\Gamma$,
defined in \EQ{extension}. 
\end{lemma}

\begin{lemma}\label{lemma:averagenorm}
There exists a positive constant $C$ such that for all $\mu_{\Gamma}\in\widetilde{\Lambda}_{\Gamma},$
\begin{align*}
\|E_D\mu_{\Gamma}\|^2_{\widetilde{S}_{\Gamma}}\leq CC_{ED}^2\|\mu_{\Gamma}\|^2_{\widetilde{S}_{\Gamma}},
\end{align*}
where $C^2_{ED}=\frac{\Ctwo^4}{\epsilon^2}\left(1+\log\frac Hh\right)^2\gamma_{h,\tau_K}$.
\end{lemma}

The following lemma is similar to \cite[Lemmas 7.9 and 7.10]{TuLi:2008:BDDCAD}. 

\begin{lemma}\label{lemma:wnorm}
Given $\mu_{\Gamma}\in\Lhat_{\Gamma}$, let
$w_{\Gamma}=\widetilde{S}^{-1}_{\Gamma}\widetilde{R}_{D,\Gamma}S_{\Gamma}\mu_{\Gamma}$, then we have
\begin{equation}\label{equation:weqn}
\| \wG
\|^2_{\BTG} =\left<\uG,T\uG\right>_\SG
\end{equation}
and
\[
\|w_{\Gamma}\|^2_{\widetilde{B}_{\Gamma}}\leq C\frac{\Ctwo^4}{\epsilon^2}C_{ED}^2 \|\mu_{\Gamma}\|^2_{B_{\Gamma}}.
\]
\end{lemma}
\begin{proof}

Since $\Rtilde_{D,\Gamma}^T \wG = \Rtilde_{D,\Gamma}^T
\Stilde^{-1}_\Gamma\Rtilde_{D,\Gamma} S_\Gamma u = T \uG$, we have,
using Lemma~\LA{null},
\begin{eqnarray*}
\| \wG \|^2_{\BTG} &= & \left< \wG, \wG \right>_{\STG} = \wG^T
\Stilde_\Gamma \wG =
\wG^T\Stilde_\Gamma \Stilde^{-1}_\Gamma\Rtilde_{D,\Gamma} \SG \uG \\
& = & \wG^T
\Rtilde_{D,\Gamma}\SG\uG=\big<\uG,\Rtilde_{D,\Gamma}^T\wG \big>_\SG
= \left<\uG, T\uG\right>_\SG. 
\end{eqnarray*}

Moreover,
\begin{align*}
\langle T\mu_{\Gamma},T\mu_{\Gamma}\rangle_{B_{\Gamma}}&=\langle T\mu_{\Gamma},T\mu_{\Gamma}\rangle_{S_{\Gamma}}\\
&=\langle\widetilde{R}^T_{D,\Gamma}\widetilde{S}^{-1}_{\Gamma}\widetilde{R}_{D,\Gamma}S_{\Gamma}\mu_{\Gamma},\widetilde{R}^T_{D,\Gamma}\widetilde{S}^{-1}_{\Gamma}\widetilde{R}_{D,\Gamma}S_{\Gamma}\mu_{\Gamma}\rangle_{S_{\Gamma}}\\
&=\langle\widetilde{R}_{\Gamma}\widetilde{R}^T_{D,\Gamma}w_{\Gamma},\widetilde{R}_{\Gamma}\widetilde{R}^T_{D,\Gamma}w_{\Gamma}\rangle_{\widetilde{S}_{\Gamma}}=\|E_Dw_{\Gamma}\|^2_{\widetilde{B}_{\Gamma}}.
\end{align*}
From Lemmas \LA{averagenorm}, \LA{Sgnorm}, and \EQ{weqn},  we have
\begin{align*}
\langle T\mu_{\Gamma},T\mu_{\Gamma}\rangle_{B_{\Gamma}}&=\|E_Dw_{\Gamma}\|^2_{\widetilde{B}_{\Gamma}}\leq CC^2_{ED}\|w_{\Gamma}\|^2_{\widetilde{B}_{\Gamma}}\\
&=CC_{ED}^2\langle\mu_{\Gamma},T\mu_{\Gamma}\rangle_{S_{\Gamma}}\\
&\leq C\frac{\Ctwo^2}{{\epsilon}} C^2_{ED}\|T\mu_{\Gamma}\|_{B_{\Gamma}}\|\mu_{\Gamma}\|_{B_{\Gamma}}.
\end{align*}
Therefore
\begin{equation}\label{eq3.10}
\begin{aligned}
\|T\mu_{\Gamma}\|_{B_{\Gamma}}\leq C\frac{\Ctwo^2}{{\epsilon}} C_{ED}^2\|\mu_{\Gamma}\|_{B_{\Gamma}},
\end{aligned}
\end{equation}
 and 
\begin{align*}
\|w_{\Gamma}\|^2_{\widetilde{B}_{\Gamma}}&=\langle\mu_{\Gamma},T\mu_{\Gamma}\rangle_{S_{\Gamma}}\leq
                                         C\frac{\Ctwo^2}{{\epsilon}}\|\mu_{\Gamma}\|_{B_{\Gamma}}\|T\mu_{\Gamma}\|_{B_{\Gamma}}\\
&\leq C\frac{\Ctwo^4}{\epsilon^2} C_{ED}^2\|\mu_{\Gamma}\|^2_{B_{\Gamma}}.
\end{align*}
\end{proof}

For the lower bound, we need to estimate  the $L^2$ type error. 
\begin{lemma}\label{lemma:L2error} Given $\mu_{\Gamma}\in\Lhat_{\Gamma}$,
let  $v_{\Gamma}=T\mu_{\Gamma}-\mu_{\Gamma}$, then 
\[\|v_{\mathcal{A},\Gamma}\|_{h}\le
  CC_{h}\|\mu_{\Gamma}\|_{B_{\Gamma}},\]
where $C_h=
\frac{\Ctwo^6\left(1+\log\frac{H}{h}\right)^2\gamma_{h,\tau_K}\mymax}{\epsilon^{4}},$
$C_L(\epsilon,H,h)=\frac{1}{\sqrt\epsilon}\max(H\epsilon, H^2)$ is defined in
Lemma \LA{exact}.
\end{lemma}

Now we are ready to prove our main theorem Theorem
\ref{theorem:maintheorem}. 
The proof is similar to that for
\cite[Theorem 7.15]{TuLi:2008:BDDCAD}.

{\bf Proof of Theorem \ref{theorem:maintheorem}:}
The upper bound Equation $\eqref{eq3.11}$ is given by $\eqref{eq3.10}.$
To prove the lower bound Equation $\eqref{eq3.12}$, for any
$\mu_\Gamma\in \Lhat_\Gamma$, using the identity
$\widetilde{R}^T_{\Gamma}\widetilde{R}_{D,\Gamma}=I$, we have \begin{align*}
\langle \mu_{\Gamma}, \mu_{\Gamma}\rangle_{B_{\Gamma}}
&=\langle \mu_{\Gamma},\mu_{\Gamma} \rangle_{S_{\Gamma}}=\mu^{T}_{\Gamma}\widetilde{R}^T_{\Gamma}\widetilde{R}_{D,\Gamma}S_{\Gamma}\mu_{\Gamma}\\
&=\mu^{T}_{\Gamma}\widetilde{R}^T_{\Gamma}\widetilde{S}_{\Gamma}\widetilde{S}^{-1}_{\Gamma}\widetilde{R}_{D,\Gamma}S_{\Gamma}\mu_{\Gamma}=\langle w_{\Gamma},\widetilde{R}_{\Gamma}\mu_{\Gamma}\rangle_{\widetilde{S}_{\Gamma}}\\
&\leq C\frac{\Ctwo^2}{\epsilon} \|w_{\Gamma}\|_{\widetilde{B}_{\Gamma}}\|\mu_{\Gamma}\|_{B_{\Gamma}}\\
&=C\frac{\Ctwo^2}{\epsilon}  \langle\mu,T\mu_{\Gamma}\rangle^{\frac12}_{S_{\Gamma}}\|\mu_{\Gamma}\|_{B_{\Gamma}},
\end{align*}
where we use Lemmas \LA{Sgnorm} and \LA{wnorm} for the last two steps. 

Canceling the common factor and using Lemma \LA{null},  we have 
\begin{align*}
\|\mu_{\Gamma}\|^2_{B_{\Gamma}}&\leq
                               C\frac{\Ctwo^4}{\epsilon^2} \langle\mu_{\Gamma},T\mu_{\Gamma}\rangle_{S_{\Gamma}}\\
 &=C\frac{\Ctwo^4}{\epsilon^2}  (\langle\mu_{\Gamma},T\mu_{\Gamma}\rangle_{B_{\Gamma}}+
\langle\mu_{\Gamma},T\mu_{\Gamma}-\mu_{\Gamma}\rangle_{Z_{\Gamma}}).
\end{align*}

Let $v_\Gamma=T\mu_{\Gamma}-\mu_{\Gamma}$. Using Lemma \LA{Sgnorm}, we have
\begin{align*}
\langle\mu_{\Gamma},\mu_{\Gamma}\rangle_{B_{\Gamma}}&\leq C\frac{\Ctwo^4}{\epsilon^2}  (\langle\mu_{\Gamma},T\mu_{\Gamma}\rangle_{B_{\Gamma}}+
\langle\mu_{\Gamma},T\mu_{\Gamma}-\mu_{\Gamma}\rangle_{Z_{\Gamma}})\\
&\leq C\frac{\Ctwo^4}{\epsilon^2}  \langle\mu_{\Gamma},T\mu_{\Gamma}\rangle_{B_{\Gamma}}+C\frac{\Ctwo^6}{\epsilon^{5/2}}\|\mu_{\Gamma}\|_{B_{\Gamma}}\|v_{\A,\Gamma}\|_{h}\\
&\leq C\frac{\Ctwo^2}{\epsilon} \langle\mu_{\Gamma},T\mu_{\Gamma}\rangle_{B_{\Gamma}}+C\frac{\Ctwo^6C_h}{\epsilon^{5/2}}\|\mu_{\Gamma}\|^2_{B_{\Gamma}},
\end{align*}
where we use Lemma \LA{L2error} for the last step.

The second term in the right hand side can be combined with the left
hand side and Equation $\eqref{eq3.12}$ is proved. 
\eproof

\section{Some estimates and results}
In this section, we will provide the proofs for Lemmas
\LA{Sgnorm}, \LA{averagenorm}, and \LA{L2error}.

 First, we list two useful results.  For shape regular partition $\mathcal{T}_{h}$
as detailed in \cite{WYWGS2016}, the following trace and inverse inequalities
hold; see\cite{WY2014mixed}.

\begin{lemma}\label{lemma:trace}
(Trace Inequality)  There exists a constant $C$ such that
\begin{equation}
\|g\| _{e}^{2}\leq C\left(h_{K}^{-1}\|g\| _{K}^{2}+h_{K}\|\nabla g\| _{K}^{2}\right),\label{eq:trace}
\end{equation}
where $g\in H^1(K)$, and $K$ is an element of $\mathcal{T}_{h}$ with $e$ as an edge/face.
\end{lemma}

\begin{lemma}
(Inverse Inequality) There exists a constant $C=C(k)$ such that
\begin{equation}
\label{lemma:inverse}
\|{\nabla g}\|_K \leq C(k) h^{-1}_K \|g\|_K, \qquad \forall K \in \mathcal{T}_{h},
\end{equation}
for any piecewise polynomial $g$ of degree $k$ on $\mathcal{T}_{h}$.
\end{lemma}

In order to prove Lemmas \LA{Sgnorm} and
\LA{averagenorm}, we first need to establish the
relation between the bilinear form $a_h(\lambda, \lambda)$ defined in
\EQ{adef} and the $\interleave\lambda \interleave$ defined in
\EQ{IIInorm} for all $\lambda\in \Lambda$.  The difficulty is to estimate $\Q\lambda$ and $\U\lambda$ these two operators.
Similar estimate for the standard finite element discretization
studied in \cite{TuLi:2008:BDDCAD} is much more straightforward.  We
will use some techniques developed in \cite{error,multigrid}.

\subsection{Estimate for the norms}
We introduce two operators  ${\bm J}\mathcal{Q}\mu$ and ${\bm J}\mathcal{U}\mu$
 such that for all ${\bf r}\in {\bf V}$ and
$w\in W$, 
\begin{equation}\label{equation:JQJUdef}
\begin{aligned}
&( \epsilon^{-1}{\bm J}\mathcal{Q}\mu,{\bf r})_K-({\bm J\mathcal{U}\mu},\nabla\cdot{\bf r})_K=-\langle\mu,{\bf r}\cdot{\bf n}\rangle_{\partial K},\\
&(w,\nabla\cdot({\bm J}\mathcal{Q}\mu))_K+\langle(\tau_K-\frac12{\bm\beta}\cdot{\bf n})({\bm J}\mathcal{U}\mu-\mu),w \rangle_{\partial K}=0.
\end{aligned}
\end{equation}
We note that ${\bm J}\mathcal{Q}\mu$ and ${\bm J}\mathcal{U}\mu$
satisfy  the same equations as $\mathcal{Q}\mu$ and $\mathcal{U}\mu$
defined in \cite[Equation (2.9)]{multigrid} with $\tau_K$ in
\cite{multigrid}  replaced by $\tau_K-\frac12{\bm\beta\cdot\bf n}$.

By $\EQ{eq3.4}$, we have  
\begin{equation}\label{equation:IJQIJUdef}
\begin{aligned}
&(\epsilon^{-1}({\bm I}-{\bm J})\mathcal{Q}\mu,{\bf r})_K-( ({\bm I}-{\bm J})\mathcal{U}\mu,\nabla\cdot{\bf r})_K=0,\\
&(\nabla\cdot{({\bm I}-{\bm J})\mathcal{Q}\mu},w)_K+\langle(\tau_K-\frac12{\bm\beta}\cdot{\bf n})({\bm I}-{\bm J})\mathcal{U}\mu,w\rangle_{\partial K}+((-\frac12\nabla\cdot{\bm\beta})\mathcal{U}\mu,w)_{K}\\
&-\frac12({\bm\beta}\mathcal{U}\mu,\nabla w)_K+\frac12( w,{\bm\beta}\cdot\nabla\mathcal{U}\mu)_{K}
=-\frac12\langle{\bm\beta}\cdot{\bf n}\mu, w \rangle_{\partial K}.
\end{aligned}
\end{equation}
We have the   following results:

\begin{lemma}\label{lemma:JQbound}
\begin{eqnarray*}
&&\|{\bm J}\mathcal{Q}\mu\|_{K}\leq C\Cone\epsilon\interleave\mu\interleave_K,\\
&&\|(\tau_K-\frac12{\bm\beta\cdot\bf n})^{\frac12}({\bm J}\mathcal{U}\mu-\mu)\|_{\partial K}\leq  C\epsilon \sqrt{\max_{x\in \partial
                                             K}(\tau_K-\frac12{\bm\beta\cdot\bf
                                             n}) h}\interleave\mu\interleave_K, \\
&&\|{\bm J}\mathcal{U}\mu\|_{K}\leq C\|\mu\|_{h,K},\\
&&\|\nabla {\bm J}{\mathcal U}\mu\|_{K}\leq C\interleave\mu\interleave_K,
\end{eqnarray*}
where $\Cone=1+\epsilon^{-\frac12}\sqrt{\max\limits_{x\in\partial K,K\in\T_h}(\tau_K-\frac12{\bm\beta}\cdot{\bf n})h}$.
\begin{proof}
The first and  second inequalities basically follow  from \cite[Theorem
3.9]{multigrid}.  However, in \cite{multigrid}  the dependence of
$\epsilon$ is not estimated explicitly. We  follow \cite{multigrid}
and give the detailed proof of these inequalities.

By \cite[Lemma 3.2]{multigrid}, we have
\[
\|\U^{RT}\mu-\mu\|_{\partial K}\le
C\epsilon^{-1}h^{\frac12}\|\Q^{RT}\mu\|_K,\]
 where $\U^{RT}$ and $\Q^{RT}$ are defined in \cite[(3.5a)
  and (3.5b)]{multigrid}, which are the corresponding $\U$ and $\Q$
  for the elliptic problem with the RT discretization.

Following the proof of \cite[Equation (3.11)]{multigrid}, we have 
\begin{equation}\label{equation:lllQ}
\|{\bm J}\mathcal{Q}\mu-\Q^{RT}\mu\|_K\le C
\epsilon^{-\frac12}\sqrt{\max_{x\in \partial
    K}(\tau_K-\frac12{\bm\beta\cdot\bf n}) h} \|\Q^{RT}\mu\|_K
\end{equation}
and therefore 
$$\|{\bm J}\mathcal{Q}\mu\|_K\le C\left(1+
\epsilon^{-\frac12}\sqrt{\max_{x\in \partial K}(\tau_K-\frac12{\bm\beta\cdot\bf n}) h} \right)\|\Q^{RT}\mu\|_K.$$
Following the proof of  \cite[Theorem 2.2]{schwarz}, we
have
\begin{equation}\label{equation:lllQ1}\|\Q^{RT}\mu\|_K\le C\epsilon
  \interleave\mu\interleave_K
\end{equation} 
and the first inequality is proved. 

The second inequality is obtained by using
\cite[Equation (3.10)]{multigrid}. We have
\begin{eqnarray*}
\|(\tau_K-\frac12{\bm\beta\cdot\bf n})^{\frac12}({\bm
  J}\mathcal{U}\mu-\mu)\|_{\partial K}&\leq&  C
                                             \sqrt{\max_{x\in \partial
                                             K}(\tau_K-\frac12{\bm\beta\cdot\bf
                                             n}) h} \|\Q^{RT}\mu\|_K\\
&\le& C\epsilon \sqrt{\max_{x\in \partial
                                             K}(\tau_K-\frac12{\bm\beta\cdot\bf
                                             n}) h}
      \interleave\mu\interleave_K.
\end{eqnarray*}

The third inequality can be obtained by following the proof of \cite[Equation
(3.12)]{multigrid}. We have
\begin{eqnarray*}\|{\bm
  J}\mathcal{U}\mu-\U^{RT}\mu\|_K&\le& Ch\left(\|{\bm
  J}\mathcal{Q}\mu-\Q^{RT}\mu\|_K+\|\Q^{RT}\mu\|_K\right)\\
&\le& Ch\left(1+\epsilon^{-\frac12}\sqrt{\max_{x\in \partial
                                             K}(\tau_K-\frac12{\bm\beta\cdot\bf
                                             n}) h}
      \right)\|\Q^{RT}\mu\|_K\\
&\le& C \left(\epsilon+\epsilon^{\frac12}\sqrt{\max_{x\in \partial
                                             K}(\tau_K-\frac12{\bm\beta\cdot\bf
                                             n}) h}
      \right)\|\mu\|_{h,K},
\end{eqnarray*}
where we use \EQ{lllQ} for the second
inequality and \EQ{lllQ1} for the third
inequality. 

By\cite[Equation (3.6)]{multigrid}, we have $\|\U^{RT}\mu\|_K\le C
\|\mu\|_{h,K}$ and 
$$\|{\bm
  J}\mathcal{U}\mu\|_K\le C\left(1+\epsilon +\epsilon^{\frac12} \sqrt{\max_{x\in \partial
                                             K}(\tau_K-\frac12{\bm\beta\cdot\bf
                                             n}) h}
      \right)\|\mu\|_{h,K}\le C\|\mu\|_{h,K}.$$

The fourth inequality can be proved as follows. Recall that $m_K(\mu)$ is the average of $\mu$ on $\partial K$,
defined in \EQ{IIIm}.
We first prove that ${\bm J}\mathcal{U}m_K(\mu)=m_K(\mu)$. If $\mu$ is
constant on $\partial K$, by the
proof of \cite[Lemma 2.1]{schwarz}, we know $\U^{RT}\mu=\mu$ on $K$ and $\Q^{RT}\mu=0$. 
By \cite[Equation (3.12)]{multigrid}, we have
$$\|{\bm J}\mathcal{U}\mu -\U^{RT}\mu\|_K\le 0$$ and 
$${\bm J}\mathcal{U}\mu =\U^{RT}\mu=\mu.$$
Therefore ${\bm J}\mathcal{U}m_K(\mu)=m_K(\mu)$
on $K$  and  $\nabla{\bm
  J}\mathcal{U}(m_K(\mu))=0$. By the inverse inequality, the
third inequality in this lemma, \EQ{L2norm}, and \EQ{IIInorm}, we have 
\begin{align*}
\|\nabla{\bm J}\mathcal{U}\mu\|_{K}&=\|\nabla{\bm J}\mathcal{U}(\mu-m_K(\mu))\|_{K}\\
&\leq Ch^{-1}\|{\bm J}\mathcal{U}(\mu-m_K(\mu)\|_{K}\\
&\leq Ch^{-1}\|\mu-m_K(\mu)\|_{h, K}\\
&\leq C\interleave\mu\interleave_K.
\end{align*}
\end{proof}
\end{lemma}

\begin{lemma}{\label{lemma:estimateu}}
Let $\E$ be any edge of a simplex $K\in \T_h(\Omega)$, then  for all $w\in W_h$,
\begin{align*}
C\|w\|_K\le h\|\B^t w\|_K+h^{\frac12}\|w\|_{\E}. 
\end{align*} 
\end{lemma}
\begin{proof}
See  \cite[Lemmas 3.3]{multigrid}. 
\end{proof}
\begin{lemma}\label{lemma:UmuBound}
\begin{align*}
\|\U\mu\|_K\le C\nu_{\epsilon,\tau_K,h}\|\mu\|_{h,K},
\end{align*}
where $\nu_{\epsilon,\tau_K,h}$ is defined in $\eqref{equation:Cdef}$.
\end{lemma}
\begin{proof}
By triangle inequality we have 
\begin{align*}
\|\mathcal{U}\mu\|_K&\leq \|\mathcal{U}\mu-{\bm J}{\mathcal{U}\mu}\|_K+\|{\bm J}\mathcal{U}\mu\|_K.
\end{align*}
Lemma $\ref{lemma:JQbound}$ gives
$\|{\bm J}\U\mu\|_K\le C\|\mu\|_{h,K}$  and we only need to 
estimate $\|\mathcal{U}\mu-{\bm J}{\mathcal{U}\mu}\|_K$.

Let ${\bf r}=({\bm I}-{\bm J})\Q\mu,$ $w=({\bm I}-{\bm J})\U\mu$ in $\eqref{equation:IJQIJUdef}$, then we have
\begin{align*}
&(\epsilon^{-1}({\bm I}-{\bm J})\mathcal{Q}\mu,\epsilon^{-1}({\bm I}-{\bm J})\mathcal{Q}\mu)_K+\langle(\tau_K-\frac12{\bm\beta}\cdot{\bf n})({\bm I}-{\bm J})\mathcal{U}\mu,({\bm I}-{\bm J})\mathcal{U}\mu\rangle_{\partial K}\\
&+((-\frac12\nabla\cdot{\bm\beta})\mathcal{U}\mu,({\bm I}-{\bm J})\mathcal{U}\mu)_{K}-\frac12({\bm\beta}\mathcal{U}\mu,\nabla({\bm I}-{\bm J})\mathcal{U}\mu)_K+\frac12( ({\bm I}-{\bm J})\mathcal{U}\mu,{\bm\beta}\cdot\nabla\mathcal{U}\mu)_{K}\\
&=-\frac12\langle{\bm\beta}\cdot{\bf n}(\mu-{\bm J}\mathcal{U}\mu), ({\bm I}-{\bm J})\mathcal{U}\mu \rangle_{\partial K}-\frac12\langle{\bm\beta}\cdot{\bf n}{\bm J}\mathcal{U}\mu,  ({\bm I}-{\bm J})\mathcal{U}\mu\rangle_{\partial K}.
\end{align*}
Applying the divergence theorem for the second term in the right hand
side of the last equality, we have
\begin{align*}
&(\epsilon^{-1}({\bm I}-{\bm J})\mathcal{Q}\mu,\epsilon^{-1}({\bm I}-{\bm J})\mathcal{Q}\mu)_K+\langle(\tau_K-\frac12{\bm\beta}\cdot{\bf n})({\bm I}-{\bm J})\mathcal{U}\mu,({\bm I}-{\bm J})\mathcal{U}\mu\rangle_{\partial K}\\
&+((-\frac12\nabla\cdot{\bm\beta})({\bm I}-{\bm J}))\mathcal{U}\mu,({\bm I}-{\bm J})\mathcal{U}\mu)_{K}-\frac12({\bm\beta}({\bm I}-{\bm J})\mathcal{U}\mu,\nabla({\bm I}-{\bm J})\mathcal{U}\mu)_K\\
&+\frac12( ({\bm I}-{\bm J})\mathcal{U}\mu,{\bm\beta}\cdot\nabla({\bm I}-{\bm J})\mathcal{U}\mu)_{K}+( ({\bm I}-{\bm J})\mathcal{U}\mu,{\bm\beta}\cdot\nabla{\bm J}\mathcal{U}\mu)_{K}\\
&=-\frac12\langle{\bm\beta}\cdot{\bf n}(\mu-{\bm J}\mathcal{U}\mu), ({\bm I}-{\bm J})\mathcal{U}\mu \rangle_{\partial K},
\end{align*}
which implies that 
\begin{equation}{\label{equation:normofQU1}}
\begin{aligned}
&(\epsilon^{-1}({\bm I}-{\bm J})\mathcal{Q}\mu,({\bm I}-{\bm J})\mathcal{Q}\mu)_K+\langle(\tau_K-\frac12{\bm\beta}\cdot{\bf n})({\bm I}-{\bm J})\mathcal{U}\mu,({\bm I}-{\bm J})\mathcal{U}\mu\rangle_{\partial K}\\
&+((-\frac12\nabla\cdot{\bm\beta})({\bm I}-{\bm J}))\mathcal{U}\mu,({\bm I}-{\bm J})\mathcal{U}\mu)_{K}\\
&=-\frac12\langle{\bm\beta}\cdot{\bf n}(\mu-{\bm J}\mathcal{U}\mu),
({\bm I}-{\bm J})\mathcal{U}\mu \rangle_{\partial K}-( ({\bm I}-{\bm
  J})\mathcal{U}\mu,{\bm\beta}\cdot\nabla{\bm J}\mathcal{U}\mu)_{K}.
\end{aligned}
\end{equation}

By Lemma \LA{estimateu},  we can obtain that 
\begin{equation}{\label{equation:estimateIJU}}
\|\mathcal{U}\mu-{\bm J}{\mathcal{U}\mu}\|_K\leq h\|\B^t( \mathcal{U}\mu-{\bm J}{\mathcal{U}\mu})\|_K+\min_{\E\subset\partial K}h^{\frac12}\|\mathcal{U}\mu-{\bm J}{\mathcal{U}\mu}\|_{\E}.
\end{equation}
Therefore, we can estimate $\|\U\mu-{\bm J}\U\mu\|_K$ for  two cases
which depend on the values of $h\|\B^t( \mathcal{U}\mu-{\bm
  J}{\mathcal{U}\mu})\|_K$ and $\min_{\E\subset\partial
  K}h^{\frac12}\|\mathcal{U}\mu-{\bm J}{\mathcal{U}\mu}\|_{\E}$.

\begin{enumerate}
\item [Case I]:
$\|\B^t (\mathcal{U}\mu-{\bm J}{\mathcal{U}\mu})\|_K\geq \min\limits_{\E\subset\partial K}h^{-\frac12}\|\mathcal{U}\mu-{\bm J}\mathcal{U}\mu\|_{\E}:$
\begin{equation}{\label{equation:condition1}}
\|{\U}\mu-{\bm J}{\U}\mu\|_K\leq h\|\B^t( {\U}\mu-{\bm J}{{\U}\mu})\|_K
=h\|\epsilon^{-1}({\bm I}-{\bm J}){\Q}\mu\|_K,
\end{equation}
where the second equality follows from the first equation of $\eqref{equation:IJQIJUdef}.$ By $\EQ{normofQU1}$ we have 
\begin{align*}
&\epsilon^{-1}\|({\bm I}-{\bm J})\Q\mu\|^2_K\\
&\le C\||\vbeta\cdot{\bf n}|(\mu-{\bm J}\U\mu)\|_{\partial K}\|({\bm
                                                 I}-{\bm
                                                 J})\U\mu\|_{\partial
                                                 K}+\|({\bm I}-{\bm
                                                 J})\U\mu\|_K\|\vbeta\|_{\infty}\|\nabla
                                                 {\bm J}\U\mu\|_K\\
&\leq C\||\vbeta\cdot{\bf n}|(\mu-{\bm J}\mathcal{U}\mu)\|_{\partial K}h^{-\frac12}\|({\bm I}-{\bm J})\mathcal{U}\mu\|_K+\|({\bm I}-{\bm J})\U\mu\|_K\|\vbeta\|_{\infty}h^{-1}\|{\bm J}\mathcal{U}\mu\|_K\\
&\leq C\epsilon\sqrt{\max\limits_{x\in\partial K}(\tau_K-\frac12\vbeta\cdot{\bf n})h}\interleave\mu\interleave_K
h^{\frac12}\epsilon^{-1}\|({\bm I}-{\bm J})\Q\mu\|_K\\
&\quad +\epsilon^{-1}\|\vbeta\|_{\infty}\|({\bm I}-{\bm J})\Q\mu\|_K\|{\bm J}\U\mu\|_K,
\end{align*}
where the first term of second inequality follows from trace
inequality and the second term of the second inequality follows from
inverse inequality, the third inequality follows from Lemma \LA{JQbound}
and $\EQ{condition1}$.
Cancelling the common factor, we have
\begin{align*}
\|({\bm I}-{\bm J})\Q\mu\|_K\le \epsilon h\sqrt{\max\limits_{x\in\partial K}(\tau_K-\frac12\vbeta\cdot{\bf n})}\interleave\mu\interleave_K+C\|\vbeta\|_{\infty}\|\mu\|_{h,K}.
\end{align*}
Plugging  in the above estimate to \EQ{condition1}, we have
\begin{align*}
\|\U\mu-{\bm J}\U\mu\|_K&\le (h^2\sqrt{\max\limits_{x\in\partial K}(\tau_K-\frac12\vbeta\cdot{\bf n})}\interleave\mu\interleave_K+Ch\epsilon^{-1}\|\vbeta\|_{\infty}\|\mu\|_{h,K})\\
&=(h\sqrt{\max\limits_{x\in\partial K}(\tau_K-\frac12\vbeta\cdot{\bf n})}+Ch\epsilon^{-1}\|\vbeta\|_{\infty})\|\mu\|_{h,K}.
\end{align*}

\item[Case II:]
\begin{equation}\label{equation:condition2}
 \|\B^t (\mathcal{U}\mu-{\bm J}{\mathcal{U}\mu})\|_K\le
 \min\limits_{\E\subset\partial K}h^{-\frac12}\|\mathcal{U}\mu-{\bm
   J}\mathcal{U}\mu\|_{\E}.
 \end{equation}

By Assumption $\ref{assump:onbeta}$ and the second inequality in Lemma $\ref{lemma:JQbound}$, we have
\begin{eqnarray}\label{equation:eee1}
&&\||\vbeta\cdot{\bf n}|^{\frac12}({\bm
                 J}\mathcal{U}\mu-\mu)\|_{\partial K}
                 \le C\|(\tau_K-\frac12\vbeta\cdot{\bf
                 n})^{\frac12}({\bm J}\mathcal{U}\mu-\mu)\|_{\partial
                 K}\nonumber\\
  &&\le           C  \epsilon\sqrt{\max\limits_{\partial K}(\tau_K-\frac12\vbeta\cdot{\bf n})h}\interleave\mu\interleave_K.
\end{eqnarray}

Using the definition of $\B^t$ defined in \EQ{alloperators} and
\EQ{normofQU1},  we can write
\begin{equation}{\label{equation:normofQU2}}
  \begin{aligned}
    &(\epsilon^{-1}({\bm I}-{\bm J})\mathcal{Q}\mu,({\bm I}-{\bm J})\mathcal{Q}\mu)_K+\langle(\tau_K-\frac12{\bm\beta}\cdot{\bf n})({\bm I}-{\bm J})\mathcal{U}\mu,({\bm I}-{\bm J})\mathcal{U}\mu\rangle_{\partial K}\\
&+((-\frac12\nabla\cdot{\bm\beta})({\bm I}-{\bm J}))\mathcal{U}\mu,({\bm I}-{\bm J})\mathcal{U}\mu)_{K}\\
&=-\frac12\langle{\bm\beta}\cdot{\bf n}(\mu-{\bm J}\mathcal{U}\mu),
({\bm I}-{\bm J})\mathcal{U}\mu \rangle_{\partial K}-( ({\bm I}-{\bm
  J})\mathcal{U}\mu,{\bm\beta}\cdot\nabla{\bm J}\mathcal{U}\mu)_{K}\\
&=-\frac12\langle{\bm\beta}\cdot{\bf n}(\mu-{\bm J}\mathcal{U}\mu), ({\bm I}-{\bm J})\mathcal{U}\mu \rangle_{\partial K}-( ({\bm I}-{\bm J})\mathcal{U}\mu,\nabla\cdot(\vbeta{\bm J}\mathcal{U}\mu))_{K}+( ({\bm I}-{\bm J})\mathcal{U}\mu,\nabla\cdot\vbeta({\bm J}\mathcal{U}\mu))_{K}\\
&=-\frac12\langle{\bm\beta}\cdot{\bf n}(\mu-{\bm J}\mathcal{U}\mu), ({\bm I}-{\bm J})\mathcal{U}\mu \rangle_{\partial K}-({\mathcal{B}}^t(({\bm I}-{\bm J})\mathcal{U}\mu,\vbeta\cdot{\bm J}\mathcal{U}\mu)_{K}
+( ({\bm I}-{\bm J})\mathcal{U}\mu,\nabla\cdot\vbeta({\bm J}\mathcal{U}\mu))_{K}.
\end{aligned}
\end{equation}

By $\EQ{normofQU2}$ and applying Cauchy-Schwartz inequality, we have
\begin{eqnarray}\label{equation:eee2}
\|(\tau_K-\frac12\vbeta\cdot{\bf n})^{\frac12}(\mathcal{U}\mu-{\bm J}{\mathcal{U}\mu})\|^2_{\partial K}
&\leq& \frac12\||\vbeta\cdot{\bf n}|^{\frac12}(\mu-{\bm J}\mathcal{U}\mu)\|_{\partial K}\||\vbeta\cdot{\bf n}|^{\frac12}({\bm I}-{\bm J})\mathcal{U}\mu\|_{\partial K}\nonumber\\
&&\qquad+\|\B^t(({\bm I}-{\bm J})\mathcal{U}\mu )\|_{K}\|{\bm J}\mathcal{U}\mu\|_K\|\vbeta\|_{\infty}\nonumber\\
&&\qquad+\|({\bm I}-{\bm J})\mathcal{U}\mu\|_K\|{\bm J}\mathcal{U}\mu\|_K\|\nabla\cdot{\vbeta}\|_{\infty}.
\end{eqnarray}

By Assumption \ref{assump:onbeta}, we have
$\max\limits_{\E\subset\partial K}\inf\limits_{x\in
  \E}(\tau_K-\frac12\vbeta\cdot{\bf n})=\widetilde{C}_K>0$ and let
$\E^s_K$ be the edge of $K$ where $\inf\limits_{x\in \E}(\tau_K-\frac12\vbeta\cdot{\bf n})$ attains its maximum.
By \EQ{eee2}, \EQ{estimateIJU}, \EQ{eee1}, \EQ{condition2}, and  the estimation of $\|{\bm
  J}\U\mu\|_K$ in Lemma $\ref{lemma:JQbound}$, we have 
\begin{align*}
 &\|(\tau_K-\frac12\vbeta\cdot{\bf n})^{\frac12}(\mathcal{U}\mu-{\bm J}{\mathcal{U}\mu})\|^2_{\partial K}\\
&\leq \frac12\||\vbeta\cdot{\bf n}|^{\frac12}(\mu-{\bm J}\mathcal{U}\mu)\|^2_{\partial K}\||\vbeta\cdot{\bf n}|^{\frac12}({\bm I}-{\bm J})\mathcal{U}\mu\|_{\partial K}\\
&\qquad+C\|\B^t(({\bm I}-{\bm J})\mathcal{U}\mu )\|_{K}\|{\bm J}\mathcal{U}\mu\|_K\|\vbeta\|_{\infty}\\
&\qquad+C\|({\bm I}-{\bm J})\mathcal{U}\mu\|_K\|{\bm J}\mathcal{U}\mu\|_K\|\nabla\cdot{\vbeta}\|_{\infty}\\
&\leq C\epsilon h^{\frac12}\sqrt{\max\limits_{x\in\partial K}(\tau_K-\frac12\vbeta\cdot{\bf n})}\interleave\mu\interleave_K\|(\tau_K-\frac12\vbeta\cdot{\bf n})^{\frac12}({\bm I}-{\bm J})\U\mu\|_{\partial K}\\
&\qquad+C\|({\bm I}-{\bm J})\mathcal{U}\mu\|_{\E^s_K}\|\mu\|_{h,K}(h^{-\frac12}\|\vbeta\|_{\infty}+h^{\frac12}\|\nabla\cdot\vbeta\|_{\infty})\\
&\le C\epsilon h^{\frac12}\sqrt{\max\limits_{x\in\partial K}(\tau_K-\frac12\vbeta\cdot{\bf n})}\interleave\mu\interleave_K\|(\tau_K-\frac12\vbeta\cdot{\bf n})^{\frac12}({\bm I}-{\bm J})\U\mu\|_{\partial K}\\
&\qquad+\frac{C}{\widetilde{C}^{\frac12}_K}\|(\tau_K-\frac12\vbeta\cdot{\bf n})^{\frac12}({\bm I}-{\bm J})\mathcal{U}\mu\|_{\partial K}\|\mu\|_{h,K}(h^{-\frac12}\|\vbeta\|_{\infty}+h^{\frac12}\|\nabla\cdot\vbeta\|_{\infty}).
\end{align*}

Canceling the common factor $\|(\tau_K-\frac12\vbeta\cdot{\bf n})^{\frac12}({\bm I}-{\bm J})\U\mu\|_{\partial K}$ we have  
\begin{align*}
&\|(\tau_K-\frac12\vbeta\cdot{\bf n})^{\frac12}({\bm I}-{\bm J})\U\mu\|_{\partial K}\\
&\leq {C}\epsilon h^{\frac12}\sqrt{\max\limits_{x\in\partial K}(\tau_K-\frac12\vbeta\cdot{\bf n})}\interleave\mu\interleave_K+\frac{C}{\widetilde{C}^{\frac12}_K}\|\mu\|_{h,K}(h^{-\frac12}\|\vbeta\|_{\infty}+h^{\frac12}\|\nabla\cdot{\vbeta}\|_{\infty})\\
&\leq {C}(\epsilon\sqrt{\max\limits_{x\in\partial K}(\tau_K-\frac12\vbeta\cdot{\bf n})}+\frac{C}{\widetilde{C}^{\frac12}_K}\|\vbeta\|_{\infty}+C\frac{h}{\widetilde{C}^{\frac12}_K}\|\nabla\cdot\vbeta\|_{\infty})\|\mu\|_{\partial K}.
\end{align*}
Therefore, by Lemma \LA{estimateu} and \EQ{condition2}, 
\begin{align*}
\|\mathcal{U}\mu-{\bm J}\mathcal{U}\mu\|_{K}&\leq C\min\limits_{\E\subset\partial K} h^{\frac12}\|\mathcal{U}\mu-{\bm J}\mathcal{U}\mu\|_{\E}\le Ch^{\frac12}\|\mathcal{U}\mu-{\bm J}\mathcal{U}\mu\|_{\E^s_K}\\
&\leq \frac{Ch^{\frac12}}{\widetilde{C}^{\frac12}_K}\|(\tau_K-\frac12\vbeta\cdot{\bf n})^{\frac12}(\mathcal{U}\mu-{\bm J}\mathcal{U}\mu)\|_{\E^s_K}\leq C\frac{h^{\frac12}}{\widetilde{C}^{\frac12}_K}\|(\tau_K-\frac12\vbeta\cdot{\bf n})^{\frac12}(\mathcal{U}\mu-{\bm J}\mathcal{U}\mu)\|_{\partial K}\\
&\leq \frac{C}{\widetilde{C}_K}\left(\widetilde{C}^{\frac12}_K\epsilon\sqrt{\max\limits_{x\in\partial K}(\tau_K-\frac12\vbeta\cdot{\bf n})}+\|\vbeta\|_{\infty}+h\|\nabla\cdot\vbeta\|_{\infty}\right)\|\mu\|_{h,K}.
\end{align*}
\end{enumerate}
Combining the two cases we can obtain the result.
\end{proof}


\begin{lemma}\label{lemma:IJQzero}
For $\mu$, which is  a constant on $\partial K$, we have 
\begin{align*}
&({\bm I}-{\bm J}){\mathcal{U}\mu}|_{\partial K}=0,\\
&({\bm I}-{\bm J}){\mathcal{U}\mu}|_{K}=0.
\end{align*}
Therefore for any $\mu\in\Lambda,$
\begin{align*}
\|\nabla\mathcal{U}\mu\|_K\leq C\Ctwo \interleave\mu\interleave_K,
\end{align*}
where $\Ctwo$ is defined in \EQ{Cdef}.
\end{lemma}
\begin{proof}
If $\mu$ is a  constant on $\partial K$,  $\interleave \mu\interleave_K=0$. By Lemma
\LA{JQbound}, we have
\begin{equation}\label{equation:zzz}
\|(\tau_K-\frac12{\bm\beta\cdot\bf n})^{\frac12}({\bm
  J}\mathcal{U}\mu-\mu)\|_{\partial K}\leq 0 \quad \mbox{and} \quad \|\nabla{\bm J}\mathcal{U}\mu\|_K\le
0
 \end{equation}
and therefore $\nabla{\bm J}\mathcal{U}\mu=0$ in $K$.

By \EQ{zzz} and Assumption \ref{assump:onbeta}, we have
${\bm
  J}\mathcal{U}\mu-\mu=0$ 
on $\partial K$. 
Therefore, 
for any $w\in W(K)$, by integration by part, we have
\begin{align*}
-\frac12\langle{\bm\beta}\cdot{\bf n}\mu, w\rangle_{\partial K}&=-\frac12\langle{\bm\beta}\cdot{\bf n}{\bm J}\mathcal{U}\mu, w\rangle_{\partial K}\\
&=-\frac12(\nabla\cdot{\bm\beta}{\bm J}\mathcal{U}\mu,w)_K-\frac12\langle{\bm J}\mathcal{U}\mu,{\bm\beta}\cdot\nabla w\rangle_{\partial K}+\frac12( w,{\bm\beta}\cdot\nabla{\bm J}\mathcal{U}\mu)_K.
\end{align*}
Plugging this identity  in the second equation in $\EQ{IJQIJUdef}$ we can obtain that 
\begin{align*}
&(\nabla\cdot{({\bm I}-{\bm
  J})\mathcal{Q}\mu},w)_K+\langle(\tau_K-\frac12{\bm\beta}\cdot{\bf
  n})({\bm I}-{\bm J})\mathcal{U}\mu,w\rangle_{\partial K}+\langle (-\frac12\nabla\cdot{\bm\beta})(\mathcal{U}\mu-{\bm J}\mathcal{U}\mu),w\rangle_{\partial K}\\
&-\frac12({\bm\beta}(\mathcal{U}\mu-{\bm J}\mathcal{U}\mu),\nabla w)_K+\frac12(w,{\bm\beta}\cdot\nabla(\mathcal{U}\mu-{\bm J}\mathcal{U}\mu))_{K}
=0.
\end{align*}
Take $w=\mathcal{U}\mu-{\bm J}\mathcal{U}\mu$ and ${\bf r}=({\bm
  I}-{\bm J})\mathcal{Q}\mu$ in the above equation and first equation of $\EQ{IJQIJUdef}$, we have 
\begin{align*}
&(\epsilon^{-1}({\bm I}-{\bm J})\mathcal{Q}\mu,({\bm I}-{\bm J})\mathcal{Q}\mu)_K+\langle (\tau_K-\frac12{\bm\beta}\cdot{\bf n})({\bm I}-{\bm J})\mathcal{U}\mu,({\bm I}-{\bm J})\mathcal{U}\mu\rangle_{\partial K}\\
&\qquad+((-\frac12\nabla\cdot{\bm\beta})({\bm I}-{\bm J}){\mathcal{U}\mu},({\bm I}-{\bm J}){\mathcal{U}\mu})_K=0.
\end{align*}
By Assumption \ref{assump:onbeta}, we have
\begin{align*}
&({\bm I}-{\bm J}){\mathcal{Q}\mu}|_{K}=0, \\
& ({\bm I}-{\bm J}){\mathcal{U}\mu}|_{\partial K}=0.
\end{align*} 
Plugging in the first formula of $\EQ{IJQIJUdef}$, for any ${\bf r}\in{\bf V}$ we have
\begin{align*}
0&=(({\bm I}-{\bm J})\mathcal{U}\mu,\nabla\cdot{\bf r})_{K}\\
&=\langle({\bm I}-{\bm J})\mathcal{U}\mu,{\bf r}\cdot{\bf n}\rangle_{\partial K}-(\nabla({\bm I}-{\bm J})\mathcal{U}\mu,{\bf r})_K\\
&=(\nabla({\bm I}-{\bm J})\mathcal{U}\mu,{\bf r})_K,
\end{align*}
which implies that $\nabla({\bm I}-{\bm J})\mathcal{U}\mu|_{K}=0$. 
Therefore, $({\bm I}-{\bm J})\mathcal{U}\mu$ is a constant
  on $K$. Due to the fact that $({\bm I}-{\bm J})\mathcal{U}\mu=0$ on
  $\partial K$,  we obtain $({\bm I}-{\bm
    J}){\mathcal{U}(\mu)}|_{K}=0$.

  Therefore, for any $\mu\in \Lambda$, 
\begin{equation}\label{equation:JUUequal}
({\bm I}-{\bm J}){\mathcal{U}m_K(\mu)}|_{K}=0,
\end{equation}
where $m_K(\mu) $ is the average of $\mu$ on $\partial K$ defined in \EQ{IIIm}.

For any  $\mu\in \Lambda$,  by Lemmas \LA{UmuBound} and \LA{JQbound}, we have
\begin{align*}
\|({\bm I}-{\bm J})\mathcal{U}\mu\|_K&\leq \|\mathcal{U}\mu\|_K+\|{\bm J}\mathcal{U}\mu\|_{K}\\
&\leq (C\Ctwo+C)\|\mu\|_{h, K}\\
&\leq C\Ctwo \|\mu\|_{h,K},
\end{align*}
where
$\Ctwo$ is defined \EQ{Cdef}.

Then, by \EQ{JUUequal} and the inverse inequality, we have
\begin{align*}
\|\nabla({\bm I}-{\bm J})\mathcal{U}\mu\|_K&=\|\nabla({\bm I}-{\bm J})\mathcal{U}(\mu-m_K(\mu))\|_K\\
&\leq Ch^{-1}\|({\bm I}-{\bm J})\mathcal{U}(\mu-m_K(\mu))\|_K\\
&\leq C\Ctwo h^{-1}\|\mu-m_K(\mu)\|_{h,K}\\
&=
  C\Ctwo\interleave\mu\interleave_K.
\end{align*}
Therefore, {by Lemma \LA{JQbound}}, we have
$$\|\nabla\mathcal{U}\mu\|_K\leq \|\nabla{\bm J}\mathcal{U}\mu\|_K+ \|\nabla({\bm I}-{\bm J})\mathcal{U}\mu\|_K\leq C\Ctwo\interleave\mu\interleave_K.
$$
\end{proof}

\begin{lemma}\label{lemma:Ubound}
\[\|\mathcal{U}\mu-\mu\|_{L^2(\partial K)}\leq Ch^{\frac12}\Ctwo
  \interleave\mu\interleave_{K},
\]
where $\Ctwo$ is defined in \EQ{Cdef}.
\begin{proof}

By Lemma \LA{IJQzero} and \cite[Lemma 3.5 (i)]{multigrid}, we have, on
$\partial K$,
\begin{equation}\label{equation:UJU}
\mathcal{U}m_K(\mu)= {\bm J}\mathcal{U}m_K(\mu)=m_K(\mu).
\end{equation} 
Therefore,
\begin{align*}
&\|\mathcal{U}\mu-\mu\|_{\partial K}\\
&\leq \|\mathcal{U}\mu-\mathcal{U}(m_K(\mu))\|_{\partial K}+\|(\mathcal{U}(m_K(\mu))-\mu)\|_{\partial K}\\
&=\|\mathcal{U}(\mu-m_K(\mu))\|_{\partial K}+
\|(m_K(\mu)-\mu)\|_{\partial K}\\
&\leq h^{-\frac12}\|\mathcal{U}(\mu-m_{K}(\mu))\|_{K}+\|m_K(\mu)-\mu\|_{\partial K}\\
&\leq C\Ctwo h^{-\frac12}\|\mu-m_{K}(\mu)\|_{h, K}+\|m_K(\mu)-\mu\|_{\partial K}\\
&\leq C\Ctwo h^{\frac12}\interleave\mu\interleave_K,
\end{align*}
where we use \EQ{UJU} for the second equality, the trace theorem for
the third inequality, Lemma \LA{UmuBound} for the fourth inequality, and
\EQ{IIInorm} and \EQ{L2norm} for the last inequality. 
\end{proof}
\end{lemma}
\begin{lemma}\label{lemma:lastterm}
$$\epsilon^{-1}\|({\bm I}-{\bm
  J})\mathcal{Q}\mu\|^2_K+\|(\tau_K-\frac12{\bm\beta}\cdot{\bf
  n})^{\frac12}({\bm I}-{\bm J})\mathcal{U}\mu\|^2_{\partial K}\leq
C\Ctwo^2(\|\mu\|_{h,K}^2+\interleave\mu\interleave^2_K),$$
where $\Ctwo$ is defined in \EQ{Cdef}.
\end{lemma}
\begin{proof}
Let ${\bf r}=({\bm I}-{\bm J})\mathcal{Q}\mu$ and $w=({\bm I}-{\bm
  J})\mathcal{U}\mu$ in \EQ{IJQIJUdef}, and we have
\begin{align*}
&\epsilon^{-1}\|({\bm I}-{\bm J})\mathcal{Q}\mu\|^2_K+\|(\tau_K-\frac12{\bm\beta}\cdot{\bf n})^{\frac12}({\bm I}-{\bm J})\mathcal{U}\mu\|^2_{\partial K}\\
&=( \frac12\nabla\cdot{\bm\beta}\mathcal{U}\mu,({\bm I}-{\bm J})\mathcal{U}\mu)_{K}+\frac12(\mathcal{U}\mu, {\bm\beta}\cdot\nabla({\bm I}-{\bm J})\mathcal{U}\mu)_K-\frac12(
({\bm I}-{\bm J})\mathcal{U}\mu,{\bm\beta}\cdot{\nabla\mathcal{U}\mu})_{K}\\
&\qquad-\frac12\langle{{\bm\beta}\cdot{\bf n}\mu, ({\bm I}-{\bm J})\mathcal{U}\mu}\rangle_{\partial  K}\\
&=(\frac12\nabla\cdot{\bm\beta}\mathcal{U}\mu,({\bm I}-{\bm J})\mathcal{U}\mu)_{K}-(\mathcal{U}\mu,{\bm\beta}\cdot\nabla{\bm J}\mathcal{U}\mu)_K-\frac12(\nabla\cdot{\bm\beta}{\bm J}\mathcal{U}\mu,\mathcal{U}\mu)_K+\frac12\langle{\bm J}\mathcal{U}\mu,{\bm\beta}\cdot{\bf n}\mathcal{U}\mu\rangle_{\partial K}\\
&\qquad-\frac12\langle{{\bm\beta}\cdot{\bf n}\mu, ({\bm I}-{\bm J})\mathcal{U}\mu}\rangle_{\partial  K}\\
&=(\frac12\nabla\cdot{\bm\beta}\mathcal{U}\mu, ({\bm I}-{\bm J})\mathcal{U}\mu)_K-(\mathcal{U}\mu,{\bm\beta}\cdot\nabla{\bm J}\mathcal{U}\mu)_K-\frac12(\nabla\cdot{\bm\beta}{\bm J}\mathcal{U}\mu,\mathcal{U}\mu)_K\\
&\qquad+\frac12\langle{\bm J}\mathcal{U}\mu,{\bm\beta}\cdot{\bf n}{\bm J}\mathcal{U}\mu\rangle_{\partial K}+\frac12\langle{\bm J}\mathcal{U}\mu-\mu,{\bm\beta}\cdot{\bf n}({\bm I}-{\bm J})\mathcal{U}\mu\rangle_{\partial K}\\
&:=I+II+III+IV+V.
\end{align*}

We will estimate these five terms separately as follows:

For the first term, we have 
\begin{align*}
I&=(\frac12\nabla\cdot{\bm\beta}\mathcal{U}\mu, ({\bm I}-{\bm
  J})\mathcal{U}\mu)_K\leq\|\frac12\nabla\cdot{\bm\beta}\|_{{\infty}}\|\mathcal{U}\mu\|_{K}\|({\bm
  I}-{\bm J})\mathcal{U}\mu\|_K,\\
&\leq C\|\frac12\nabla\cdot{\bm\beta}\|_{\infty}\left(\Ctwo\right)^2\|\mu\|^2_{h, K},
\end{align*}
where we use Lemmas \LA{UmuBound} and  \LA{JQbound} for the last inequality.

For the second term, by Lemmas \LA{UmuBound} and  \LA{JQbound} we have
\begin{align*}
II&=-(\mathcal{U}\mu,{\bm\beta}\cdot\nabla{\bm J}\mathcal{U}\mu)_K\leq\|{\bm\beta}\|_{{\infty}}\|\mathcal{U}\mu\|_K\|\nabla{\bm J}\mathcal{U}\mu\|_K\\
&\leq C\|{\bm\beta}\|_{\infty}\Ctwo\|\mu\|_{
  h,K}\interleave\mu\interleave_K\leq
  C\|{\bm\beta}\|_{\infty}\Ctwo(\|\mu\|^2_{h,
  K}+\interleave\mu\interleave^2_K).
\end{align*}
Similarly, we have
\begin{align*}
III&=-\frac12(\nabla\cdot{\bm\beta}{\bm J}\mathcal{U}\mu,\mathcal{U}\mu)_K\leq\frac12\|\nabla\cdot{\bm\beta}\|_{\infty}\|{\bm J}\mathcal{U}\mu\|_K\|\mathcal{U}\mu\|_K\\
&\leq C\|\nabla\cdot{\bm\beta}\|_{\infty}\Ctwo\|\mu\|^2_{h,K}.
\end{align*}

For the fourth term, we first use integration by part for the second
equality and Lemma \LA{JQbound} 
\begin{align*}
IV&=\frac12\langle{\bm J}\mathcal{U}\mu,{\bm\beta}\cdot{\bf n}{\bm
    J}\mathcal{U}\mu\rangle_{\partial
    K}=\frac12({\nabla}\cdot{\bm\beta}{\bm J}\mathcal{U}\mu,{\bm
    J}\mathcal{U}\mu)_K+({\bm\beta}\cdot\nabla{{\bm
    J}\mathcal{U}\mu,{\bm J}\mathcal{U}\mu})_K\\
&\leq C\|\nabla\cdot{\bm\beta}\|_{\infty}\|\mu\|^2_{ h,K}+C\|{\bm\beta}\|_{\infty}\interleave\mu\interleave_K\|\mu\|_{h,K}\\
&\leq
  C\|\nabla\cdot{\bm\beta}\|_{\infty}\|\mu\|^2_{h,K}+C\|{\bm\beta}\|_{\infty}(\interleave\mu\interleave_K^2+\|\mu\|^2_{h,K}).
\end{align*}
For the last term, by Assumption \ref{assump:onbeta} and Lemma \LA{JQbound},  we have
\begin{align*}
V&=\frac12\langle{\bm J}\mathcal{U}\mu-\mu,{\bm\beta}\cdot{\bf n}({\bm I}-{\bm J})\mathcal{U}\mu\rangle_{\partial K}\\
&\leq\frac12\|(\tau_K-\frac12{\bm\beta}\cdot{\bf n})^{\frac12}({\bm J}\mathcal{U}\mu-\mu)\|_{\partial K}\|(\tau_K-\frac12{\bm\beta}\cdot{\bf n})^{\frac12}({\bm I}-{\bm J})\mathcal{U}\mu\|_{\partial K}\\
&\leq \frac14  \epsilon^2 \max_{x\in \partial
                                             K}(\tau_K-\frac12{\bm\beta\cdot\bf
                                             n}) h\interleave\mu\interleave^2_K+\frac14\|(\tau_K-\frac12{\bm\beta}\cdot{\bf n})^{\frac12}({\bm I}-{\bm J})\mathcal{U}\mu\|^2_{\partial K}.
\end{align*}

Combing the estimation of these five terms and we obtain 

\begin{align*}
&\epsilon^{-1}\|({\bm I}-{\bm J})\mathcal{Q}\mu\|^2_K+\|(\tau_K-\frac12{\bm\beta}\cdot{\bf n})^{\frac12}({\bm I}-{\bm J})\mathcal{U}\mu\|^2_{\partial K}\\
&\qquad\leq
  C(\Ctwo)^2(\|{\bm\beta}\|_{\infty}+\|\frac12\nabla\cdot{\bm\beta}\|_{\infty}+\|\nabla\cdot{\bm\beta}\|_{\infty}+1)(\|\mu\|^2_{h,
  K}+\interleave\mu\interleave^2_K)\\
&\qquad +\frac14\|(\tau_K-\frac12{\bm\beta}\cdot{\bf n})^{\frac12}({\bm I}-{\bm J})\mathcal{U}\mu\|^2_{\partial K}.
\end{align*}
Cancelling the common terms in both sides, we have
\[
\epsilon^{-1}\|({\bm I}-{\bm J})\mathcal{Q}\mu\|^2_K+\|(\tau_K-\frac12{\bm\beta}\cdot{\bf n})^{\frac12}({\bm I}-{\bm J})\mathcal{U}\mu\|^2_{\partial K}\\
\leq C(\Ctwo)^2 (\|\mu\|^2_{h,
  K}+\interleave\mu\interleave^2_K).\]
\end{proof}
\begin{lemma}\label{lemma:muk|||}
For all $\mu$ with zero boundary on quasiuniform meshes, we have
$$\|\mu\|^2_{h, \Omega}\leq C\interleave\mu\interleave^2_{\Omega}.$$ 
\end{lemma}
\begin{proof}
See  \cite[Lemma 3.7]{multigrid}.
\end{proof}
\begin{lemma}\label{lemma:equivalent} 
\begin{align*}
c{\epsilon}\interleave\mu\interleave^2 \leq a_h(\mu,\mu)\leq C\Ctwo^2\interleave\mu\interleave^2.
\end{align*}
\end{lemma}
\begin{proof}
By \EQ{sbformK} and \EQ{sbformN}, we have
\begin{align*}
a_h(\mu,\mu)&=b_h(\mu,\mu)+z_h(\mu,\mu)\\
&=\epsilon^{-1}\|\mathcal{Q}\mu\|^2_{\mathcal{T}_h}
+\|(\tau_K-\frac12{\bm\beta}\cdot{\bf n})^{\frac12}(\mathcal{U}\mu-\mu)\|^2_{\partial\mathcal{T}_h}+\|(-\frac12\nabla\cdot{\bm\beta})^{\frac12}\mathcal{U}\mu\|^2_{\mathcal{T}_h}\\
&=\epsilon^{-1}\|{\bm J}\mathcal{Q}\mu+({\bm I}-{\bm J})\mathcal{Q}\mu\|^2_{\mathcal{T}_h}\\
&+\|(\tau_K-\frac12{\bm\beta}\cdot{\bf n})^{\frac12}\left(({\bm I}-{\bm J}) \mathcal{U}\mu+({\bm J}\mathcal{U}\mu-\mu)\right)\|^2_{\partial\mathcal{T}_h}+\|(-\frac12\nabla\cdot{\bm\beta})^{\frac12}\mathcal{U}\mu\|^2_{\mathcal{T}_h}\\
&\leq C\left(\epsilon^{-1}\|{\bm J}\mathcal{Q}\mu\|^2_{\mathcal{T}_h}+\epsilon^{-1}\|({\bm I}-{\bm J})\mathcal{Q}\mu\|^2_{\mathcal{T}_h}+\|(\tau_K-\frac12{\bm\beta}\cdot{\bf n})^{\frac12}(({\bm I}-{\bm J})\mathcal{U}\mu \|^2_{\partial\mathcal{T}_h}\right.\\
& \left.\qquad+\|(\tau_K-\frac12{\bm\beta}\cdot{\bf n})^{\frac12}({\bm J}\mathcal{U}\mu-\mu)\|^2_{\partial\mathcal{T}_h}+\|(-\frac12\nabla\cdot{\bm\beta})^{\frac12}\mathcal{U}\mu\|^2_{\mathcal{T}_h}\right)\\
&=C\sum\limits_{K\in\mathcal{T}_h}\left(\epsilon^{-1}\|{\bm J}\mathcal{Q}\mu\|^2_{K}+\epsilon^{-1}\|({\bm I}-{\bm J})\mathcal{Q}\mu\|^2_{K}+\|(\tau_K-\frac12{\bm\beta}\cdot{\bf n})^{\frac12}({\bm I}-{\bm J})\mathcal{U}\mu \|^2_{\partial K}\right.\\
& \qquad+\left.\|(\tau_K-\frac12{\bm\beta}\cdot{\bf n})^{\frac12}({\bm J}\mathcal{U}\mu-\mu)\|^2_{\partial K}+\|(-\frac12\nabla\cdot{\bm\beta})^{\frac12}\mathcal{U}\mu\|^2_{K}\right)\\
&\leq C\left(\Cone^2\epsilon +\Ctwo^2+\Cone^2\epsilon^2+\Ctwo^2\right)(\sum\limits_{K\in\mathcal{T}_h}\|\mu\|^2_{h,K}+\sum\limits_{K\in\mathcal{T}_h}\interleave\mu\interleave^2_{K}).
\end{align*}
Here we use Lemma \LA{JQbound} for the first term, Lemma \LA{lastterm} for
the second and third terms,  Lemma \LA{JQbound} for the fourth term, and
Lemma \LA{UmuBound} for the last term.  

Combining the above estimate with Lemma \LA{muk|||}, we obtain the
upper bound estimate.

On the other hand, by the similar proof in \cite[Lemma 3.8]{multigrid}, we can obtain that $\epsilon^{-1}\|\mathcal{Q}\mu\|_K\geq C\interleave\mu\interleave_K$, therefore
\begin{align*}
a_h(\mu,\mu)=\sum\limits_{K\in\T_h}a_K(\mu,\mu)\geq \sum\limits_{K\in\T_h}\epsilon^{-1}\|\mathcal{Q}\mu\|^2_K\geq C\sum\limits_{K\in\T_h}\epsilon\interleave \mu\interleave^2_K=C\epsilon \interleave\mu\interleave^2_{\Omega}.
\end{align*}

\end{proof}

In the following, we establish the relation between $\|\cdot\|_{h,\Omega_i}$
and $\interleave\cdot\interleave_{\Omega_i}$  for the functions in
$\Lambda^{(i)}$, which  have zero edge average
  on 
$\partial \Omega_i$.  We also prove similar result as
Lemma \LA{muk|||} for the functions in $\Ltilde$. 
 
Let $\Lambda^{0,(i)}$ be the zero-order numerical trace space in
$\Omega_i$ and $Q_0$ be the $L^2$ orthogonal projection from
$\Lambda^{(i)}$ into $\Lambda^{0,(i)}$.

\begin{lemma}{\label{lemma:L2projection}}
For any $\mu^{(i)}\in\Lambda^{(i)}$,
\begin{equation}
\begin{aligned}
\interleave Q_0\mu^{(i)}\interleave_{\Omega_i}&\leq C\interleave \mu^{(i)}\interleave_{\Omega_i},\\
\|\mu^{(i)}-Q_0\mu^{(i)}\|^2_{\partial K}&\leq Ch\interleave\mu\interleave^2_{K},\\
\interleave\mu\interleave^2_K &\leq  Ch^{-1}\|\mu\|^2_{\partial K}.
\end{aligned}
\end{equation}
\end{lemma}
\begin{proof}
See \cite[(4.9), (4.10), and (2.9)]{schwarz}.
\end{proof}

Let $V^0_h(\Omega_i)$ denote the $P_1$ non-conforming finite element spaces on
the mesh ${\mathcal T}_h(\Omega_i)$. As in \cite[Section 3]{schwarz},
we can establish an isomorphism $X_h: \Lambda^{0,(i)}\rightarrow
V_h^0(\Omega_i)$, by 
\begin{equation}\label{equation:Xhdef}
(X_h\lambda) (x_e)=\lambda|_e,
\end{equation}
where $\lambda\in \Lambda^{0,(i)}$ and $x_e$ denotes the midpoint of edge
$e$.  

 For $w\in V^0_h(\Omega_i)$, we define $|w|^2_{H^1(\Omega_i)}=\sum_{K\in \Omega_i}\|\nabla
w\|^2_{L^2(K)}$.  Similar to \cite[Lemma 3.1]{schwarz}, by the definition and
scaling arguments, we have, for any $\lambda\in \Lambda^{0,(i)}$,
\begin{eqnarray}\label{equation:p10}
c|X_h\lambda|_{H^1(\Omega_i)}&\le& \interleave\lambda\interleave_{\Omega_i}\le
  C|X_h\lambda|_{H^1(\Omega_i)},\\
c\|X_h\lambda\|_{L^2(\Omega_i)}&\le& \|\lambda\|_{h,\Omega_i}\le
  C\|X_h\lambda\|_{L^2(\Omega_i)}.\nonumber
\end{eqnarray}

Using $X_h$, we have the following lemma

\begin{lemma}\label{lemma:averagezeroscaling}
If $ \mu\in \Lambda^{(i)}$ and $\mu$ is edge average zero on $\partial
\Omega_i$, then 
$$\|\mu\|_{h,\Omega_i}\leq CH\interleave\mu\interleave_{\Omega_i}.$$
\end{lemma}
\begin{proof}
Letting $\mu_0$=$Q_0\mu$, by Lemma \LA{L2projection}, we have
\begin{eqnarray}\label{equation:u0}
\|\mu-\mu_0\|^2_{h,\Omega_i}&=&h\sum_{K\in
    \T_h,K\subseteq\Omega_i}\left\|\mu-\mu_0\right\|^2_{L^2(\partial
    K)}\\
&\le&C h\sum_{K\in
    \T_h,K\subseteq\Omega_i}h\interleave\mu\interleave^2_{K}=Ch^2\interleave\mu\interleave^2_{\Omega_i}.\nonumber
\end{eqnarray}

Given any edge $E$ of $\partial\Omega_i$, we have $\int_E\mu=0$.
Since $\mu_0$=$Q_0\mu$, we have  $\int_E\mu_0=0$. By the definition of
$X_h$ and the linearity of $P_1$ non-conforming elements, we have
 $\int_E(X_h\mu_0)=0$, and therefore
 $\int_{\partial\Omega_i}(X_h\mu_0)=0$.

By the Poincare inequality (see \cite{brenner}) and the scaling
argument for $P_1$ non-conforming element,  we have
\[\|X_h\mu\|_{L^2(\Omega_i)}\le CH|X_h\mu|_{H^1(\Omega_i)}.\]
Combing this with \EQ{p10}, we have
\begin{equation}\label{equation:p10x}
\|\mu_0\|_{h,\Omega_i}\le C\|X_h\mu_0\|_{L^2(\Omega_i)}\le
CH|X_h\mu_0|_{H^1(\Omega_i)}\le C H\interleave\mu_0\interleave_{\Omega_i}.
\end{equation}

Therefore, combing \EQ{u0} and \EQ{p10x}, we have
\begin{eqnarray*}
\|\mu\|_{h,\Omega_i}&=&\|\mu-\mu_0+\mu_0\|_{h,\Omega_i}\leq
                        \|\mu-\mu_0\|_{h,\Omega_i}+\|\mu_0\|_{h,\Omega_i}\\
&\le& Ch\interleave\mu\interleave_{\Omega_i}+C H\interleave\mu_0\interleave_{\Omega_i}\\
&\le &CH\interleave\mu\interleave_{\Omega_i}.
\end{eqnarray*}

\end{proof}

Using Lemma \LA{averagezeroscaling} and the proof of Lemma
\LA{equivalent}, we can prove the following
equivalence for subdomain functions.
\begin{lemma}\label{lemma:subequivalent} 
{If $ \mu^{(i)}\in \Lambda^{(i)}$ and $\mu^{(i)}$ is edge average zero on $\partial
\Omega_i$, then 
\begin{align*}
c{\epsilon}\interleave\mu^{(i)}\interleave^2_{\Omega_i} \leq a^{(i)}_h(\mu^{(i)},\mu^{(i)})\leq C\Ctwo^2\interleave\mu^{(i)}\interleave_{\Omega_i}^2.
\end{align*}}
We also note that the lower bound is always hold for any $ \mu^{(i)}\in \Lambda^{(i)}$.
\end{lemma}

 Assumption \ref{assump:onsubdomainshape} assume that the 
subdomains form a shape regular coarse mesh of $\Omega$, denoted by
$\T_H$. We also denote by $V_H^0$, the $P_1$
non-conforming finite element space on this coarse mesh formed by the
subdomains.  Let $\Lambda_H^0$ be the zero-order numerical trace space
in $\Omega$ on this mesh.  
Similarly to $X_h$ defined in \EQ{Xhdef}, we
 can establish an isomorphism $X_H:\Lambda^0_H\rightarrow V_H^0$ by
$$(X_H\lambda)(x_E)=\lambda|_E,$$
where $\lambda\in\Lambda^0_H$ and $x_E$ denotes the midpoint of
element edge
$E$ of the coarse mesh.

Let $Q_H^V$ denotes the $L^2$ orthogonal
projection from $V^0_h$ to $V^0_H$ and $I_h^V$ is an intergrid
transfer operator from $V^0_H$ to $V^0_h$, defined in
\cite{schwarz}. A similar operator can be found in
\cite{BrennerA1989}.  By the definition/construction, we have that
$Q_H^V$ and $I_h^V$ both preserve  the edge average over the edges in $\T_H$.

By \cite[Lemma 4.5 and Lemma 4.6]{schwarz}, we have 
  \begin{equation}\label{equation:htH}
   |Q_H^Vw|_{H^1(\T_H)}\le C|w|_{H^1(\T_h)}, \quad \forall w\in V_h^0,
  \end{equation}
  and
   \begin{equation}\label{equation:Hth}
    |I_h^Vw|_{H^1(\T_h)}\le C|w|_{H^1(\T_H)}, \quad \forall w\in V_H^0.
  \end{equation}
  \EQ{Hth} is also given in \cite[Lemma 1]{BrennerA1989}.


We have the similar result as
Lemma \LA{muk|||} for the functions in $\Ltilde$.
\begin{lemma}\label{lemma:mu|||t}
If $\mu\in\Ltilde,$ then 
\begin{align*}
\|\mu\|_{h}\leq C\interleave\mu\interleave.
\end{align*}
\end{lemma}
\begin{proof}
Given $\mu\in \Ltilde$,  let $\mu^{(i)}$ be the restriction of $\mu$
to subdomain $\Omega_i$ and $\mu^{(i)}_0=Q_0\mu^{(i)}$.   by Lemma \LA{L2projection}, we have
\begin{align*}
\|\mu\|^2_h&=\sum_{i=1}^N \|\mu^{(i)}\|^2_{h,\Omega_i} 
\le \sum_{i=1}^N \left(\|\mu^{(i)}-\mu^{(i)}_0\|^2_{h,\Omega_i}+\|\mu^{(i)}_0\|^2_{h,\Omega_i}\right)\\
&\leq Ch^2\interleave\mu\interleave^2+C \sum_{i=1}^N \|X_h\mu^{(i)}_0\|^2_{L^2(\Omega_i)}.
\end{align*}
Next we give an estimation of $\sum_{i=1}^N
\|X_h\mu^{(i)}_0\|^2_{L^2(\Omega_i)}$.

Let $u_H^{(i)}=Q_H^VX_h \mu^{(i)}_0$.
By the definition of $u_H^{(i)}$ and \EQ{htH}, we
can obtain that
\begin{equation}\label{equation:hh}
|u_H^{(i)}|_{H^1(\T_H(\Omega_i))}=|Q_H^VX_h\mu^{(i)}_0|_{H^1(\T_H(\Omega_i))}\le
C|X_h\mu^{(i)}_0|_{H^1(\T_h(\Omega_i))}.
\end{equation}

Since $\mu \in \Ltilde$,  the average
of $\mu^{(i)}$ are the same across the subdomain interfaces.  Therefore,
$\mu^{(i)}_0$ and  $X_h\mu^{(i)}_0$ have the same average across the subdomain interfaces,
see the proof in Lemma \LA{averagezeroscaling}. Given an edge $E_{ij}$
shared
by 
Subdomains $\Omega_i$ and $\Omega_i$, we have
\begin{eqnarray*}
\int_{E_{ij}}u_H^{(i)}&=&\int_{E_{ij}} Q_H^VX_h\mu^{(i)}_0 = \int_{E_{ij}}
  X_h\mu^{(i)}_0=\int_{E_{ij}}\mu^{(i)}_0\\
&=&\int_{E_{ij}}\mu^{(j)}_0=\int_{E_{ij}}X_h\mu^{(j)}_0=\int_{E_{ij}} Q_H^VX_h\mu^{(j)}_0
 =\int_{E_{ij}}u_H^{(j)}.
\end{eqnarray*}

Due to the linearity of $u_H^{(i)}$ and $u_H^{(j)}$, $u_H^{(i)}$ and
$u_H^{(j)}$ have the same values at the midpoint of $E_{ij}$.  Let $u_H$
be  the function with $u_H^{(i)}$ in Subdomain $\Omega_i$, we have $u_H\in V_H^0$.

Since $\mu$ has zero boundary condition, $u_H$ is zero on
$\partial\Omega$.   By the Poincare inequality (see \cite{brenner})
for $P_1$ non-conforming finite element functions, we have
\begin{equation}\label{equation:pH}
\|I_h^Vu_H\|^2_{L^2(\Omega)}\le |I_h^Vu_H|^2_{H^1(\T_h(\Omega))}.
\end{equation}
Moreover,  $X_h\mu^{(i)}_0-I_h^Vu_H^{(i)}$ have zero edge averages
on $\partial \Omega_i$. We also have, by the Poincare inequality in \cite{brenner},
\begin{equation}\label{equation:ph}
\|X_h\mu^{(i)}_0-I_h^Vu_H^{(i)}\|^2_{L^2(\Omega_i)}\le |X_h\mu^{(i)}_0-I_h^Vu_H^{(i)}|^2_{H^1(\T_h(\Omega_i))}.
\end{equation}

By \EQ{Hth}, \EQ{hh}, \EQ{pH}, \EQ{ph}, and \EQ{p10}, 
we have
\begin{eqnarray*}
&&\sum_{i=1}^N \|X_h\mu^{(i)}_0\|^2_{L^2(\Omega_i)}\\
 &\le& \sum_{i=1}^N
                                                        \left(\|X_h\mu^{(i)}_0-I_h^Vu_H^{(i)}\|^2_{L^2(\Omega_i)}+\|I_h^Vu_H^{(i)}\|^2_{L^2(\Omega_i)}\right)
\le  \sum_{i=1}^N
       |X_h\mu^{(i)}_0-I_h^Vu_H^{(i)}|^2_{H^1(\T_h(\Omega_i))}+\|I_h^Vu_H\|^2_{L^2(\Omega)}\\
&\le &C \sum_{i=1}^N
      \left( |X_h\mu^{(i)}_0|^2_{H^1(\T_h(\Omega_i))}+|I_h^Vu_H^{(i)}|^2 _{H^1(\T_h(\Omega_i))}\right)+|I_h^Vu_H|^2_{H^1(\T_h(\Omega))}\\
&\le &C \sum_{i=1}^N
      |X_h\mu^{(i)}_0|^2_{H^1(\T_h(\Omega_i))}+|I_h^Vu_H|^2_{H^1(\T_h(\Omega))}
\le C \sum_{i=1}^N
       |X_h\mu^{(i)}_0|^2_{H^1(\T_h(\Omega_i))}+|u_H|^2_{H^1(\T_H(\Omega))}\\
&=&\sum_{i=1}^N
       \left(|X_h\mu^{(i)}_0|^2_{H^1(\T_h(\Omega_i))}+|u_H^{(i)}|^2_{H^1(\T_H(\Omega_i))}\right)\\
&\le &\sum_{i=1}^N |X_h\mu^{(i)}_0|^2_{H^1(\T_h(\Omega_i))} 
\le  \sum_{i=1}^N\interleave\mu^{(i)}\interleave^2_{\Omega_i}=\interleave\mu\interleave^2.
 \end{eqnarray*}

\end{proof}

\subsection{Estimates for the bilinear forms}

Recall that the symmetric and skew-symmetric parts of the subdomain local
bilinear form $a^{(i)}(\eta^{(i)},\mu^{(i)})$, defined \EQ{lbformA},  are denoted by
$b^{(i)}(\eta^{(i)},\mu^{(i)})$ and $z^{(i)}(\eta^{(i)},\mu^{(i)})$, defined in
\EQ{lbformK} and  \EQ{lbformN}, respectively.  We will give the estimates of these
bilinear forms which are useful for our lower and upper bound estimates.

\begin{lemma}\label{lemma:lbnorm}
There exists a positive $C$ such that for all $\eta^{(i)},\mu^{(i)}\in \Lambda^{(i)}$
\begin{eqnarray*}
b^{(i)}(\eta^{(i)},\mu^{(i)})&\le &C\|\eta^{(i)}\|_{B^{(i)}}\|\mu^{(i)}\|_{B^{(i)}},\\
z^{(i)}(\eta^{(i)},\mu^{(i)})&\leq & C\frac{\nu^2_{\epsilon,\tau_K,h}}{\sqrt{\epsilon}}\|\eta^{(i)}\|_{B^{(i)}}\|\mu\|_{h,\Omega_i}+C\frac{\nu^2_{\epsilon,\tau_K,h}}{\sqrt{\epsilon}}
\|\eta^{(i)}\|_{h,\Omega_i}\|\mu^{(i)}\|_{B^{(i)}}.
\end{eqnarray*}
\end{lemma}
\begin{proof} 
\begin{align*}
&b^{(i)}(\eta^{(i)},\mu^{(i)})\\
&=(\epsilon^{-1}\mathcal{Q}\eta^{(i)},\mathcal{Q}\mu^{(i)})_{\T_h(\Omega_i)}+((\tau_K-\frac12{\bm\beta}\cdot{\bf n})(\mathcal{U}\eta^{(i)}-\eta^{(i)}),\mathcal{U}\mu^{(i)}-\mu^{(i)})_{\partial\T_h(\Omega_i)}+((-\frac12\nabla\cdot{\bm\beta})\mathcal{U}\eta^{(i)},\mathcal{U}\mu^{(i)}\rangle_{\T_h(\Omega_i)}\\
&\leq \epsilon^{-1}\|\mathcal{Q}\eta^{(i)}\|_{\T_h(\Omega_i)}\|\mathcal{Q}\mu^{(i)}\|_{\T_h(\Omega_i)}+
\|(\tau_K-\frac12{\bm\beta}\cdot{\bf n})^{\frac12}(\mathcal{U}\eta-\eta)\|_{\partial\T_h(\Omega_i)}\|(\tau_K-\frac12{\bm\beta}\cdot{\bf n})^{\frac12}(\mathcal{U}\mu-\mu)\|_{\partial\T_h(\Omega_i)}\\
&\qquad+\frac12\|(-\nabla\cdot{\bm\beta})^{\frac12}\mathcal{U}\eta^{(i)}\|_{\T_h(\Omega_i)}
\|(-\nabla\cdot{\bm\beta})^{\frac12}\mathcal{U}\mu^{(i)}\|_{\T_h(\Omega_i)}\\
&\leq C\|\eta^{(i)}\|_{B^{(i)}}\|\mu^{(i)}\|_{B^{(i)}},
\end{align*}
where we use the definition of $\|\cdot\|_{B^{(i)}}$ and the Cauchy-Schwarz
inequality.

We can prove the bound for $z^{(i)}(\eta^{(i)},\mu^{(i)})$
  as follows:  
\begin{align}\label{equation:localzzz}
&|z^{(i)}(\eta^{(i)},\mu^{(i)})|\nonumber\\
&=\left|\frac12({\bm\beta}\cdot\nabla\mathcal{U}\eta^{(i)},\mathcal{U}\mu^{(i)})_{\T_h(\Omega_i)}-\frac12(\mathcal{U}\eta^{(i)},{\bm\beta}\cdot\nabla\mathcal{U}\mu^{(i)})_{\T_h(\Omega_i)}\right.\\
&\qquad\left.-\frac12\langle{\bm\beta}\cdot{\bf n}\mathcal{U}\eta^{(i)},\mu^{(i)}\rangle_{\partial\T_h(\Omega_i)}+\frac12\langle{\bm\beta}\cdot{\bf{n}}\eta^{(i)},\mathcal{U}\mu^{(i)}\rangle_{\partial\T_h(\Omega_i)}\right|\nonumber\\
&=\left|\frac12({\bm\beta}\cdot\nabla\mathcal{U}\eta^{(i)},\mathcal{U}\mu^{(i)})_{\T_h(\Omega_i)}-\frac12(\mathcal{U}\eta^{(i)},{\bm\beta}\cdot\nabla\mathcal{U}\mu^{(i)})_{\T_h(\Omega_i)}\right.\nonumber\\
&\qquad\left.+\frac12\langle{\bm\beta}\cdot{\bf n}\mathcal{U}\eta^{(i)}, \mathcal{U}\mu^{(i)}-\mu^{(i)}\rangle_{\partial\T_h(\Omega_i)}+\frac12\langle{\bm\beta}\cdot{\bf{n}}\left(\eta^{(i)}-\mathcal{U}\eta^{(i)}\right),\mathcal{U}\mu^{(i)}\rangle_{\partial\T_h(\Omega_i)}\right|\nonumber\\
&\le
  \left|\frac12({\bm\beta}\cdot\nabla\mathcal{U}\eta^{(i)},\mathcal{U}\mu^{(i)})_{\T_h(\Omega_i)}\right|+\left|\frac12(\mathcal{U}\eta^{(i)},{\bm\beta}\cdot\nabla\mathcal{U}\mu^{(i)})_{\T_h(\Omega_i)}\right|\nonumber\\
&\qquad
+\frac12\left\|\left|{\bm\beta}\cdot{\bf
  n}\right|^{\frac12}\right\|_{\infty}\|\mathcal{U}\eta^{(i)}\|_{\partial\T_h(\Omega_i)}\left\|\left|{\bm\beta}\cdot{\bf{n}}\right|^{\frac12}\left(\mathcal{U}\mu^{(i)}-\mu^{(i)}\right)\right\|_{\partial\T_h(\Omega_i)}\nonumber\\
&\qquad +\frac12 \|\left|{\bm\beta}\cdot{\bf
  n}\right|^{\frac12}\|_{\infty}\|\mathcal{U}\mu^{(i)}\|_{\partial\T_h(\Omega_i)}\left\|\left|{\bm\beta}\cdot{\bf{n}}\right|^{\frac12}\left(\mathcal{U}\eta^{(i)}-\eta^{(i)}\right)\right\|_{\partial\T_h(\Omega_i)}\nonumber\\
&\leq
  \frac12\|\bm\beta\|_{\infty}\|\nabla\mathcal{U}\eta^{(i)}\|_{\T_h(\Omega_i)}\|\mathcal{U}\mu^{(i)}\|_{\T_h(\Omega_i)}
+\frac12\|\bm\beta\|_{\infty}\|\mathcal{U}\eta^{(i)}\|_{\T_h(\Omega_i)}\|\nabla\mathcal{U}\mu^{(i)}\|_{\T_h(\Omega_i)}\nonumber\\
&\qquad +\frac12\left\|\left|{\bm\beta}\cdot{\bf
  n}\right|\right\|_{\infty}\|\mathcal{U}\eta^{(i)}\|_{\partial\T_h(\Omega_i)}\left\|\left(\mathcal{U}\mu^{(i)}-\mu^{(i)}\right)\right\|_{\partial\T_h(\Omega_i)}\nonumber\\
&\qquad +\frac12\|\left|{\bm\beta}\cdot{\bf
  n}\right|\|_{\infty}\|\mathcal{U}\mu^{(i)}\|_{\partial\T_h(\Omega_i)}\left\|\left(\mathcal{U}\eta^{(i)}-\eta^{(i)}\right)\right\|_{\partial\T_h(\Omega_i)},
\end{align}
where Assumption \ref{assump:onbeta} is used in the last step.

By Lemmas \LA{IJQzero} and \LA{UmuBound}, we have
$$\|\nabla\mathcal{U}\eta^{(i)}\|_{\T_h(\Omega_i)}\le
C\Ctwo\interleave\eta^{(i)}\interleave_{\Omega_i},\quad
\|\mathcal{U}\mu^{(i)}\|_{\T_h(\Omega_i)}\le C\Ctwo\|\mu^{(i)}\|_{h,\Omega_i}.$$


By Lemma \LA{Ubound}, we have
$$\left\|\mathcal{U}\mu^{(i)}-\mu^{(i)}\right\|_{\partial\T_h(\Omega_i)}\le
C\Ctwo
  h^{\frac12}\interleave\mu^{(i)}\interleave_{\Omega_i}.$$

Plugging in all these estimates in \EQ{localzzz}  and use Lemma
\LA{equivalent}, we have
\begin{align*}
&|z^{(i)}(\eta^{(i)},\mu^{(i)})|\\
&\leq
  C\nu^2_{\epsilon,\tau_K,h}\interleave\eta^{(i)}\interleave_{\Omega_i}\|\mu\|_{h,\Omega_i}+C\nu^2_{\epsilon,\tau_K.h}\|\eta^{(i)}\|_{h,\Omega_i}\interleave\mu^{(i)}\interleave_{\Omega_i}\nonumber\\
&+C\|\mathcal{U}\eta^{(i)}\|_{\partial\Omega_i}\Ctwo h^{\frac12}\interleave\mu^{(i)}\interleave_{\Omega_i}
+C\|\mathcal{U}\mu^{(i)}\|_{\partial\Omega_i}\Ctwo h^{\frac12}\interleave\eta^{(i)}\interleave_{\Omega_i}\\
&\leq C\nu^2_{\epsilon,\tau_K,h}\frac1{\sqrt{\epsilon}}\|\eta^{(i)}\|_{B^{(i)}}\|\mu\|_{h,\Omega_i}+C\nu^2_{\epsilon,\tau_K,h}
\|\eta^{(i)}\|_{h,\Omega_i}\frac1{\sqrt{\epsilon}}\|\mu^{(i)}\|_{B^{(i)}}.
\end{align*}

\end{proof}
\begin{lemma}\label{lemma:lbnormzero}
For all $\eta^{(i)},\mu^{(i)}\in \Lambda^{(i)}$. If  $\mu^{(i)}$
 is zero on $\partial \Omega_i$, we have
\begin{eqnarray*}
z^{(i)}(\eta^{(i)},\mu^{(i)})&\leq&
                            \frac{\Ctwo^2}{\sqrt{\epsilon}}(\|\eta^{(i)}\|_{h,\Omega_i}+\interleave\eta^{(i)}\interleave_{\Omega_i})\|\mu^{(i)}\|_{B^{(i)}}
\end{eqnarray*}
\end{lemma}
\begin{proof}
By integration by part, we have 
\begin{align*}
&z^{(i)}(\eta^{(i)},\mu^{(i)})\\
&=\frac12
  ({\bm\beta}\cdot\nabla\mathcal{U}\eta^{(i)},\mathcal{U}\mu^{(i)})_{\T_h(\Omega_i)}-\frac12
  (\mathcal{U}\eta^{(i)},{\bm\beta}\cdot\nabla\mathcal{U}\mu^{(i)})_{\T_h(\Omega_i)}\\
&\quad-\frac12\langle{\bm\beta}\cdot{\bf n}\mathcal{U}\eta^{(i)},\mu^{(i)}\rangle_{\partial\T_h(\Omega_i)}+\frac12\langle{\bm\beta}\cdot{\bf{n}}\eta^{(i)},\mathcal{U}\mu^{(i)}\rangle_{\partial\T_h(\Omega_i)}\\
&=\frac12
  \langle{\bm\beta}\cdot{\bf n}\mathcal{U}\eta^{(i)},\mathcal{U}\mu^{(i)}\rangle_{\partial\T_h(\Omega_i)}
-\frac12
  (\mathcal{U}\eta^{(i)},{\bm\beta}\cdot\nabla\mathcal{U}\mu^{(i)})_{\T_h(\Omega_i)}
-\frac12
  (\mathcal{U}\eta^{(i)},\nabla\cdot{\bm\beta}\mathcal{U}\mu^{(i)})_{\T_h(\Omega_i)}\\
&\quad-\frac12
  (\mathcal{U}\eta^{(i)},{\bm \beta}\cdot\nabla\mathcal{U}\mu^{(i)})_{\T_h(\Omega_i)}-\frac12\langle{\bm\beta}\cdot{\bf n}\mathcal{U}\eta^{(i)},\mu^{(i)}\rangle_{\partial\T_h(\Omega_i)}+\frac12\langle{\bm\beta}\cdot{\bf{n}}\eta^{(i)},\mathcal{U}\mu^{(i)}\rangle_{\partial\T_h(\Omega_i)}\\
&= -(\mathcal{U}\eta^{(i)},{\bm\beta}\cdot\nabla\mathcal{U}\mu^{(i)})_{\T_h(\Omega_i)}
-\frac12
  (\mathcal{U}\eta^{(i)},\nabla\cdot{\bm\beta}\mathcal{U}\mu^{(i)})_{\T_h(\Omega_i)}
-\frac12
  \langle\mathcal{U}\eta^{(i)}-\eta^{(i)},{\bm\beta}\cdot {\bf
  n}\left(\mathcal{U}\mu^{(i)}-\mu^{(i)}\right)\rangle_{\partial
 \T_h(\Omega_i)}\\
&\qquad+\langle{\bm\beta}\cdot{\bf n}(\mathcal{U}\mu^{(i)}-\mu^{(i)}), \mathcal{U}\eta^{(i)}\rangle_{\partial\T_h(\Omega_i)}+\frac12\langle{\bm\beta}\cdot{\bf{n}}\eta^{(i)},\mu^{(i)}\rangle_{\partial\T_h(\Omega_i)}.
\end{align*}
Since $\mu^{(i)}=0$ on $\partial \Omega_i$, we have
$\langle{\bm\beta}\cdot{\bf{n}}\eta^{(i)},\mu^{(i)}\rangle_{\partial\T_h(\Omega_i)}=0$
and therefore
\begin{align*}
&z^{(i)}(\eta^{(i)},\mu^{(i)})\\
&= -(\mathcal{U}\eta^{(i)},{\bm\beta}\cdot\nabla\mathcal{U}\mu^{(i)})_{\T_h(\Omega_i)}
-\frac12
  (\mathcal{U}\eta^{(i)},\nabla\cdot{\bm\beta}\mathcal{U}\mu^{(i)})_{\T_h(\Omega_i)}
-\frac12
  \langle\mathcal{U}\eta^{(i)}-\eta^{(i)},{\bm\beta}\cdot {\bf
  n}\left(\mathcal{U}\mu^{(i)}-\mu^{(i)}\right)\rangle_{\partial
  \T_h(\Omega_i)}\\
&\qquad+\langle\mathcal{U}\eta^{(i)},{\bm\beta}\cdot{\bf
  n}(\mathcal{U}\mu^{(i)}-\mu^{(i)})\rangle_{\partial\T_h(\Omega_i)}\\
\le &
     \|\mathcal{U}\eta^{(i)}\|_{\T_h(\Omega_i)}\|{\bm\beta}\cdot\nabla\mathcal{U}\mu^{(i)}\|_{\T_h(\Omega_i)}
+\frac12\|(-\nabla\cdot{\bm\beta})^{\frac12}\mathcal{U}\eta^{(i)}\|_{\T_h(\Omega_i)}\|(-\nabla\cdot{\bm\beta})^{\frac12}\mathcal{U}\mu^{(i)}\|_{\T_h(\Omega_i)}\\
&+\frac12\| |{\bm\beta}\cdot {\bf
  n}|^{\frac12}\left(\mathcal{U}\eta^{(i)}-\eta^{(i)}\right)\|_{\partial\T_h(\Omega_i)}\||{\bm\beta}\cdot {\bf
  n}|^{\frac12}\left(\mathcal{U}\mu^{(i)}-\mu^{(i)}\right)\|_{\partial\T_h(\Omega_i)}\\
&+\|{\bm\beta}\cdot{\bf n}\|^{\frac12}_{\infty}\|\mathcal{U}\eta^{(i)}\|_{\partial\T_h(\Omega_i)
}\||{\bm\beta}\cdot {\bf
  n}|^{\frac12}\left(\mathcal{U}\mu^{(i)}-\mu^{(i)}\right)\|_{\partial\T_h(\Omega_i)}\\
\le&
C\Ctwo^2\|
     \eta^{(i)}\|
     _{h,\Omega_i}\interleave\mu^{(i)}\interleave_{\Omega_i}
+C\Ctwo\|\eta^{(i)}\|_{h,\Omega_i}\|\mu^{(i)}\|_{B^{(i)}}\\
&+C\Ctwo h^{\frac12}\interleave\eta^{(i)}\interleave_{\Omega_i}\|\mu^{(i)}\|_{B^{(i)}}+C\Ctwo
  h^{-\frac12}\|
     \eta^{(i)}\|
     _{h,\Omega_i}\Ctwo h^{\frac12}\interleave\mu^{(i)}\interleave_{\Omega_i},
\end{align*}
where we use Lemmas \LA{UmuBound} and  \LA{IJQzero} for the first
term, Lemma \LA{UmuBound} and \EQ{lbformK}
for the second term, Assumption \ref{assump:onbeta}, Lemma \LA{Ubound} and \EQ{lbformK}
for the third term, and Lemmas \LA{trace}, \LA{inverse},
\LA{UmuBound}, Assumption
\ref{assump:onbeta}, and Lemma \LA{Ubound} for the last term.

We conclude the proof by using the lower bound in Lemma
\LA{equivalent} for Subdomain $\Omega_i$.

\end{proof} 
\begin{remark}\label{remark:moree}
  Compare to \cite{TuLi:2008:BDDCAD}, we cannot control $L_2$-norm of
    $\mathcal{U}\mu$ when the velocity field $\bm\beta$ is divergence
    free. Therefore, we have additional $\frac{1}{\sqrt{\epsilon}}$ in our estimate. 
  \end{remark}

\begin{lemma}\label{lemma:bnorm}
There exists a positive $C$ such that for all $\eta,\mu\in
\Lambda$
\begin{eqnarray*}
z_h(\eta,\mu)&\leq&C \frac{\Ctwo^2}{{\epsilon}}\|\mu\|_{B}\|\eta\|_B,\\
a_h(\eta,\mu)&\leq &C\frac{\Ctwo^2}{{\epsilon}}\|\mu\|_{B}\|\eta\|_B,\\
z_h(\eta,\mu)&\leq& C\Ctwo^2\interleave\eta\interleave \|\mu\|_{h},
\end{eqnarray*}
and 
for all $\eta,\mu\in
\Ltilde$
\begin{eqnarray*}
\widetilde{z}(\eta,\mu)&\leq& C\frac{\Ctwo^2}{{\epsilon}}\|\mu\|_{\widetilde{B}}\|\eta\|_{\widetilde{B}},\\
\widetilde{a} (\eta,\mu)&\leq &C\frac{\Ctwo^2}{{\epsilon}}\|\mu\|_{\widetilde{B}}\|\eta\|_{\widetilde{B}}.
\end{eqnarray*}
\end{lemma}
\begin{proof}

Using Lemma \LA{lbnorm}, we have 
\begin{align*}
z_h(\eta,\mu)&=\sum\limits_{i=1}^Nz^{(i)}(\eta^{(i)},\mu^{(i)})\\
&\leq C\frac{\Ctwo^2}{\sqrt{\epsilon}}\sum\limits_{i=1}^N\left(\|\eta^{(i)}\|_{B^{(i)}}\|\mu^{(i)}\|_{h,\Omega_i}
+\|\eta^{(i)}\|_{h,\Omega_i}\|\mu^{(i)}\|_{B^{(i)}}\right)\\
&\leq C\frac{\Ctwo^2}{\sqrt{\epsilon}}\left(\|\eta\|_B\sum\limits_{i=1}^N\|\mu^{(i)}\|_{h,\Omega_i}+\sum\limits_{i=1}^N\|\eta^{(i)}\|_{h,\Omega_i}\|\mu\|_B\right)\\
&\leq C\frac{\Ctwo^2}{\sqrt{\epsilon}}\left(\|\eta\|_B\interleave\mu\interleave+\interleave\eta\interleave\|\mu\|_B\right)\\
&\leq C\frac{\Ctwo^2}{\epsilon}\|\eta\|_B\|\mu\|_B,
\end{align*}
where we use Lemmas \LA{muk|||} and \LA{equivalent} for the last two inequalities, respectively.
By Lemma \LA{lbnorm}, we have \begin{align*}
a_h(\eta,\mu)=b_h(\eta,\mu)+z_h(\eta,\mu)\leq C\frac{\Ctwo^2}{{\epsilon}}\|\mu\|_{B}\|\eta\|_{B}.
\end{align*}

Similarly, by using Lemma \LA{mu|||t}, we can prove 
\begin{eqnarray*}
\widetilde{z}(\eta,\mu)&\leq& C\frac{\Ctwo^2}{{\epsilon}}\|\mu\|_{\widetilde{B}}\|\eta\|_{\widetilde{B}},\\
\widetilde{a} (\eta,\mu)&\leq &C\frac{\Ctwo^2}{{\epsilon}}\|\mu\|_{\widetilde{B}}\|\eta\|_{\widetilde{B}}.
\end{eqnarray*}
for $\eta,\mu\in
\Ltilde$.

Similarly to the proof of Lemma \LA{lbnormzero}, by integration by
part, we have 
\begin{align*}
z_h(\eta,\mu)&=\frac12({\bm\beta}\cdot\nabla\mathcal{U}\eta,\mathcal{U}\mu)_{\mathcal{T}_h}-\frac12(\mathcal{U}\eta,{\bm\beta}\cdot\nabla{\mathcal{U}}\mu)_{\mathcal{T}_h}+\frac12\langle{\bm\beta}\cdot{\bf n}\eta,\mathcal{U}\mu\rangle_{\partial\mathcal{T}_h}-\frac12\langle{\bm\beta}\cdot{\bf n}\mathcal{U}\eta,\mu\rangle_{\partial\mathcal{T}_h}\\
&=\frac12
  ({\bm\beta}\cdot\nabla\mathcal{U}\eta,\mathcal{U}\mu) _{\mathcal{T}_h}
-\frac12
  \langle\mathcal{U}\eta,{\bm\beta}\cdot {\bf
  n}\mathcal{U}\mu\rangle _{\partial\mathcal{T}_h}
+\frac12
  (\mathcal{U}\eta,\nabla\cdot{\bm\beta}\mathcal{U}\mu)_{\mathcal{T}_h}
+\frac12
  ({\bm \beta}\cdot\nabla\mathcal{U}\eta,\mathcal{U}\mu)_{\mathcal{T}_h}\\
&-\frac12\langle{\bm\beta}\cdot{\bf
  n}\mathcal{U}\eta,\mu\rangle _{\partial\mathcal{T}_h}+\frac12\langle{\bm\beta}\cdot{\bf{n}}\eta,\mathcal{U}\mu\rangle _{\partial\mathcal{T}_h}\\
&=({\bm\beta}\cdot\nabla\mathcal{U}\eta,\mathcal{U}\mu)_{\mathcal{T}_h}+\frac12
  (\mathcal{U}\eta,\nabla\cdot{\bm\beta}\mathcal{U}\mu)_{\mathcal{T}_h}-\frac12\langle{\bm\beta}\cdot{\bf n}(\mathcal{U}\eta-\eta),\mathcal{U}\mu\rangle_{\partial\mathcal{T}_h}
-\frac12\langle{\bm\beta}\cdot{\bf n}(\mathcal{U}\eta-\eta),\mu\rangle_{\partial\mathcal{T}_h}\\
&\leq C \Ctwo^2\interleave\eta\interleave_{\Omega}\|\mu\|_{h},
\end{align*}
where we use $\langle{\bm\beta}\cdot{\bf
  n}\eta,\mu\rangle_{\partial\mathcal{T}_h}=0$ for $\mu=0$ on
$\partial \Omega$ for the third equality. For the last inequality, we
use Lemma \LA{IJQzero} and \LA{UmuBound} for the first term,
Lemmas \LA{UmuBound} and \LA{muk|||} for the second term,  Assumption
\ref{assump:onbeta}, Lemmas \LA{Ubound} and \LA{trace},
\LA{inverse}  for the last two terms.
\end{proof} 

{\bf Proof of Lemma \LA{Sgnorm}:} 

By Lemma \LA{null}, we have
\begin{eqnarray*}
\langle\lambda_\Gamma,\mu_\Gamma\rangle_{S_\Gamma}
  &=&\langle\lambda_{\A,\Gamma},\mu_{\A,\Gamma}\rangle_{A}\le C\frac{\Ctwo^2}{{\epsilon}}\|\lambda_{\A,\Gamma}\|_B\|\mu_{\A,\Gamma}\|_B\\
&=&C\frac{\Ctwo^2}{{\epsilon}}\|\lambda_{\Gamma}\|_{B_\Gamma}\|\mu_{\Gamma}\|_{B_\Gamma}.
\end{eqnarray*}
where we use Lemma \LA{bnorm} for the second inequality and Lemma
\LA{null} again for the last equality.
The result for $\lambda_\Gamma,\mu_\Gamma\in \Ltilde_\Gamma$ with the
corresponding norm can be proved similarly using Lemma \LA{bnorm}.

Similarly, by the definition \EQ{productS}, we have 
\begin{eqnarray*}
\langle\lambda_\Gamma,\mu_\Gamma\rangle_{Z_\Gamma}&=&\langle
                                                      \lambda_{\A,\Gamma},\mu_{\A,\Gamma}\rangle_Z
\le C\Ctwo^2\interleave \lambda_{\A,\Gamma}\interleave \|\mu_{\A,\Gamma}\|_h\\
&\le &\frac{\Ctwo^2}{\sqrt{\epsilon}}\|\lambda_{\A,\Gamma}\|_B\|\mu_{\A,\Gamma}\|_h\\
&=&\frac{\Ctwo^2}{\sqrt{\epsilon}}\|\lambda_{\Gamma}\|_{B_\Gamma}\|\mu_{\A,\Gamma}\|_h,
\end{eqnarray*}
where we use Lemma \LA{bnorm} for the second inequality, Lemma
\LA{equivalent} for the third inequality,  and Lemma
\LA{null}  for the last equality.

\eproof

\subsection{Estimate of the average operator}


We define the following elliptic problem:
\begin{equation}
\label{equation:pde2} \left\{
\begin{array}{rcl}
-\epsilon \Delta u & =& f ,
\quad \mbox{in } \Omega, \\ [0.5ex] u & = &0 , \quad\mbox{on }
\partial \Omega.
\end{array} \right.
\end{equation}
Let $a^e_h(\eta,\mu)$ be the bilinear form defined in
\cite[(2.7)]{TW:2016:HDG} and define the norm 
$$\|\mu\|^2_{A^{(e)}}:=a^{e}_h(\mu,\mu),\quad\forall \mu\in\Lambda.$$
Given $\mu^{(i)}_{\Gamma}\in\Lambda^{(i)}_{\Gamma}$, the harmonic extension 
$\mu^{(i)}_{\mathcal{H},\Gamma}$ is defined as  
\begin{equation}\label{defh1}
\begin{aligned}
\|\mu^{(i)}_{\H,\Gamma}\|^2_{A^{e(i)}}:=\min_{\mu^{(i)}\in\Lambda^{(i)},\mu^{(i)}=\mu^{(i)}_{\Gamma
  } \mbox{ on }\partial\Omega_i}\|\mu^{(i)}\|^2_{A^{e(i)}}
\end{aligned}
\end{equation}
and define $S^{e(i)}_{\Gamma}$ norm as 
\begin{equation}\label{defh2}
\begin{aligned}
\|\mu^{(i)}_{\Gamma}\|^2_{S^{e(i)}_{\Gamma}}:=(\mu^{(i)}_{\Gamma})^TS^{e(i)}_{\Gamma}\mu^{(i)}_{\Gamma}=\|\mu^{(i)}_{\H,\Gamma}\|^2_{A^{e(i)}}.
\end{aligned}
\end{equation}
We note that the harmonic extension has the energy minimization
property, which the extension defined in \EQ{extension} does not
have. 

We also have the following results:
\begin{lemma}\label{lemma:equivo}
For any $\mu^{(i)}\in\Lambda^{(i)},$
\begin{align*}
c\epsilon\interleave\mu^{(i)}\interleave^2_{\Omega_i}\leq |\mu|^2_{A^{e(i)}}\leq C\epsilon\gamma_{h,\tau_K}\interleave\mu\interleave^2_{\Omega_i}.
\end{align*}
\end{lemma}
\begin{proof}
See \cite[Theorem 3.9]{multigrid}. 
\end{proof}
\begin{lemma}\label{lemma:EDO}
For any $\mu_{\Gamma}\in\widetilde{\Lambda}_{\Gamma},$
\begin{align*}
|E_D\mu_{\Gamma}|^2_{\widetilde{S}^e_{\Gamma}}\leq C\gamma_{h,\tau_K}\left(1+\log\frac Hh\right)^2|\mu_{\Gamma}|^2_{\widetilde{S}^e_{\Gamma}}.
\end{align*}
\begin{proof}
See  \cite[Lemma 5.5]{TW:2016:HDG}.
\end{proof}
\end{lemma}

Now we are ready to estimate our average operator.

{\bf Proof of Lemma \LA{averagenorm}}
Define $v_{\Gamma}=E_D\mu_{\Gamma}-\mu_{\Gamma},$ and
$v^{(i)}_\Gamma=\barR_\Gamma^{(i)} v_\Gamma \in \Lambda^{(i)}$.
$v^{(i)}_{\Gamma}$ has  zero edge average on $\partial\Omega_i$. 
Moreover, by the fact that
$a^{(i)}(v^{(i)}_{\A,\Gamma},q^{(i)})=0,$ for
any $q^{(i)}\in \Lambda^{(i)},$ taking
$q^{(i)}=v^{(i)}_{\A,\Gamma}-v^{(i)}_{\H,\Gamma}$, 
we have
\begin{align*}
&\|v^{(i)}_{\A,\Gamma}-v^{(i)}_{\H,\Gamma}\|^2_{B^{(i)}}\\
&=|a^{(i)}(v^{(i)}_{\A,\Gamma}-v^{(i)}_{\H,\Gamma},v^{(i)}_{\A,\Gamma}-v^{(i)}_{\H,\Gamma})|\\
&=|a^{(i)}(v^{(i)}_{\H,\Gamma},v^{(i)}_{\A,\Gamma}-v^{(i)}_{\H,\Gamma})|\\
&\leq
  |b^{(i)}(v^{(i)}_{\H,\Gamma},v^{(i)}_{\A,\Gamma}-v^{(i)}_{\H,\Gamma})|+|z^{(i)}(v^{(i)}_{\H,\Gamma},v^{(i)}_{\A,\Gamma}-v^{(i)}_{\H,\Gamma})|.
\end{align*}
Since $v^{(i)}_{\H,\Gamma}$ has zero edge average 
on $\partial \Omega_i$ and 
$v^{(i)}_{\A,\Gamma}-v^{(i)}_{\H,\Gamma}$
is zero on $\partial \Omega_i$, by Lemmas \LA{lbnorm},
\LA{lbnormzero}, \LA{subequivalent}, and \LA{averagezeroscaling}, we have
\begin{eqnarray*}
&&\|v^{(i)}_{\A,\Gamma}-v^{(i)}_{\H,\Gamma}\|^2_{B^{(i)}}\\
&\leq&
  |b^{(i)}(v^{(i)}_{\H,\Gamma},v^{(i)}_{\A,\Gamma}-v^{(i)}_{\H,\Gamma})|+|z^{(i)}(v^{(i)}_{\H,\Gamma},v^{(i)}_{\A,\Gamma}-v^{(i)}_{\H,\Gamma})|\\
&\leq&\|v^{(i)}_{\H,\Gamma}\|_{B^{(i)}}\|v^{(i)}_{\A,\Gamma}-v^{(i)}_{\H,\Gamma}\|_{B^{(i)}}\\
&&+C\frac{\Ctwo^2}{\sqrt{\epsilon}}\left(\interleave v^{(i)}_{\H,\Gamma}\interleave_{\Omega_i}
+\|v^{(i)}_{\H,\Gamma}\|_{h,\Omega_i}\right)
\|v^{(i)}_{\A,\Gamma}-v^{(i)}_{\H,\Gamma}\|_{B^{(i)}}\\
&\le&C \frac{\Ctwo^2}{\sqrt{\epsilon}}\interleave  v^{(i)}_{\H,\Gamma}\interleave_{\Omega_i}\|v^{(i)}_{\A,\Gamma}-v^{(i)}_{\H,\Gamma}\|_{B^{(i)}}.
\end{eqnarray*}
We have
\[
\|v^{(i)}_{\A,\Gamma}-v^{(i)}_{\H,\Gamma}\|_{B^{(i)}}\leq
\frac{\Ctwo^2}{\sqrt{\epsilon}}\interleave v^{(i)}_{\H,\Gamma}\interleave_{\Omega_i},
\]
and by Lemma \LA{subequivalent}, we have
\begin{equation}\label{equation:eq5.18}
\|v^{(i)}_{\A,\Gamma}\|_{B^{(i)}}\leq \frac{\Ctwo^2}{\sqrt{\epsilon}}\interleave v^{(i)}_{\H,\Gamma}\interleave_{\Omega_i}.
\end{equation}

Therefore, 
\begin{align*}
&|E_D\mu_{\Gamma}-\mu_{\Gamma}|^2_{\widetilde{S}_{\Gamma}}=\sum\limits_{i=1}^N|v^{(i)}_{\Gamma}|^2_{{S}^{(i)}_{\Gamma}}\\
&=\sum\limits_{i=1}^N|v^{(i)}_{\A,\Gamma}|^2_{A^{(i)}}=\sum\limits_{i=1}^N|v^{(i)}_{\A,\Gamma}|^2_{B^{(i)}}\leq\frac{\Ctwo^4}{\epsilon}\sum\limits_{i=1}^N
  \interleave v^{(i)}_{\H,\Gamma}
  \interleave^2_{\Omega_i},
\end{align*}
where we use \EQ{eq5.18} for the last step. Let
$\mu^{(i)}_\Gamma=\barR^{(i)}_{\Gamma}\mu_{\Gamma}$.  {By Lemma
\LA{equivo}, (\ref{defh2}),  and Lemma \LA{EDO},} we have
\begin{align*}
&|E_D\mu_{\Gamma}-\mu_{\Gamma}|^2_{\widetilde{S}_{\Gamma}}\leq C\frac{\Ctwo^4}{\epsilon}\sum\limits_{i=1}^N
  \interleave v^{(i)}_{\H,\Gamma}
  \interleave^2_{\Omega_i}\\
&\leq C\frac{\Ctwo^4}{\epsilon^2}\sum\limits_{i=1}^N|v^{(i)}_{\H,\Gamma}|^2_{A^{e(i)}}
=C\frac{\Ctwo^4}{\epsilon^2}
\sum\limits_{i=1}^N|v^{(i)}_{\Gamma}|^2_{S^{e(i)}_{\Gamma}}\le C\frac{\Ctwo^4}{\epsilon^2}\left(1+\log\frac Hh\right)^2|\mu_{\Gamma}|^2_{\widetilde{S}^e_{\Gamma}}\\
&= C\frac{\Ctwo^4}{\epsilon^2}\left(1+\log\frac
  Hh\right)^2\sum\limits_{i=1}^N|\mu^{(i)}_{\H,\Gamma}|^2_{A^{e(i)}}
\le C\frac{\Ctwo^4}{\epsilon^2}\left(1+\log\frac
  Hh\right)^2\sum\limits_{i=1}^N|\mu^{(i)}_{\A,\Gamma}|^2_{A^{e(i)}}\\
&\le
  C\frac{\Ctwo^4}{\epsilon^2}\left(1+\log\frac
  Hh\right)^2\epsilon\gamma_{h,\tau_K}\sum\limits_{i=1}^N\interleave \mu^{(i)}_{\A,\Gamma}\interleave_{\Omega_i}^2,
\end{align*}
where we use Lemmas \LA{equivo}  again. By Lemma \LA{subequivalent}, we have
\begin{align*}
&|E_D\mu_{\Gamma}-\mu_{\Gamma}|^2_{\widetilde{S}_{\Gamma}}\leq C\frac{\Ctwo^4}{\epsilon} \left(1+\log\frac
  Hh\right)^2\gamma_{h,\tau_K}\sum\limits_{i=1}^N\interleave \mu^{(i)}_{\A,\Gamma}\interleave_{\Omega_i}^2\\
&\leq C\frac{\Ctwo^4}{\epsilon^2}\left(1+\log\frac Hh\right)^2\gamma_{h,\tau_K}\sum\limits_{i=1}^N
|\mu^{(i)}_{\A,\Gamma}|^2_{A^{(i)}}\\
&=C\frac{\Ctwo^4}{\epsilon^2}\left(1+\log\frac Hh\right)^2\gamma_{h,\tau_K}\sum\limits_{i=1}^N
|\mu^{(i)}_{\Gamma}|^2_{S^{(i)}_{\Gamma}}=C\frac{\Ctwo^4}{\epsilon^2}\left(1+\log\frac Hh\right)^2\gamma_{h,\tau_K}|\mu_{\Gamma}|^2_{\widetilde{S}_{\Gamma}}.
\end{align*}
\eproof

\begin{remark}\label{remark:Ed}
Due to the HDG discretization, we have the same 
bounds for both two and three dimensions. Compare to the result for the linear
conforming  finite element discretization in
\cite[Lemma 7.8]{TuLi:2008:BDDCAD}, our bound has one more
$\frac{1}{\epsilon}$ and two
additional factors $\Ctwo$ and
$\gamma_{h,\tau_K}$, defined in \EQ{Cdef} and \EQ{gamma},
respectively.  The additional  $\frac{1}{\epsilon}$  is due to the
estimate in Lemma  \LA{lbnormzero}, see Remark \ref{remark:moree} for
more details.  The first factor $\Ctwo$  is from Lemma \LA{IJQzero},
where $\|\nabla\mathcal{U}\mu\|_K\leq \Ctwo
\interleave\mu\interleave_K$. However, for the linear conforming finite element,
$\|\nabla \mu\|_K=|\mu|_{H^1(K)}$ is hold simply by the definition.
The second factor $\gamma_{h,\tau_K}$ is from Lemma \LA{EDO}, which is
the upper bound for the average operator of the corresponding elliptic
problem with the HDG discretization. 
For the common choices of $\tau_K$,
$\gamma_{h,\tau_K}$ is usually a bounded constant. However, $\Ctwo$
can be quite large if $\epsilon$ is very small unless $h$ is
sufficiently small to balance. 
\end{remark}

\subsection{Estimate of the $L^2$ type error}

This subsection is devoting to proving Lemma \LA{L2error}. 
The usual techniques to estimate the $L^2$ error is a duality
argument \cite{Braess:1997:FET}.   \cite[Lemma
7.12]{TuLi:2008:BDDCAD}, which is a standard finite element version of
Lemma \LA{L2error}, was  proved using these techniques. However, there are
several difficulties for the HDG discretization: (1) 
the preconditioned system
\EQ{BDDC} is a reduced system for the variable $\lambda\in \Lambda$.
We need to handle $\Q\lambda$ and $\U\lambda$ these two
operators appropriately; (2) $\lambda$ is defined only on mesh
element boundaries, we need appropriate extensions to the element
interiors. 

In order to overcome these difficulties, we first introduce a local
extension operator $S^K_i$ as  that defined in \cite[Section 5.2]{multigrid}.

Given any $\lambda\in \Ltilde$, let $\lambda^K$ be the restriction of
$\lambda$ on $\partial K$, where $K$ is an element in $
\T_h$. Let $E_i$ denote the edge of $K$ opposite to the $i$th vertex
of $K$. We define $S_i^K\lambda^K$ in $P_{k+1}(K)$ by
\begin{eqnarray}
\langle S_i^K\lambda^K,\eta\rangle_{E_i}&=&\langle
                                            \lambda^K,\eta\rangle_{E_i},
                                            \quad \forall ~\eta \in
                                            P_{k+1}(E_i),\label{equation:SKdef1}\\
( S_i^K\lambda^K,v)_K&=&(\U\lambda^K,v)_K,
                                            \quad \forall~ v \in
                                            P_{k}(K),\label{equation:SKdef2}
\end{eqnarray}
for $i=1,\cdots,n+1$ all $n+1$ edges of $K$.  Define
\begin{equation}\label{equation:SKdef}
(\lambda,\mu)_S=\sum_{K\in
  \T_h}\frac{1}{n+1}\sum_{i=1}^{n+1}(S^K_i\lambda,S_i^K\mu)_K\quad\mbox{and}\quad \|\lambda\|^2_S=(\lambda,\lambda)_S.
\end{equation}

\cite[Lemma 5.7]{multigrid} gives the following results
\begin{lemma}\label{lemma:SK}
Equations \EQ{SKdef1} and \EQ{SKdef2} uniquely define an $S_i^K\lambda^K$
in $P_{k+1}(K)$. For all $\lambda\in \Lambda$, we have
\begin{eqnarray}
&&\U\lambda^K=Q_{k}(S_i^K\lambda^K),\label{equation:SKL1}\\
&&c\|\lambda\|_h\le \|\lambda\|_S\le C{\nu_{\epsilon,\tau_K,h}}\|\lambda\|_h,\label{equation:SKL2}\\
&&{|S_i^K\lambda|_{H^1(K)}\le C{\nu_{\epsilon,\tau_K,h}}\interleave\lambda\interleave_K},\label{equation:SKH1}
\end{eqnarray}
where $Q_{k}$ is a $L^2$ projection to $P_{k}(K)$ and {$\nu_{\epsilon,\tau_K,h}$ is defined in \EQ{Cdef}}.
\end{lemma}
{
\begin{proof}
This can be proved by a modification on the proof of \cite[Lemma 5.7]{multigrid} with  Lemma \LA{UmuBound} and Lemma \LA{IJQzero}.
\end{proof}
}

In our proof, we will also need to use another extension $X_h$, defined
in \EQ{Xhdef}, to extend $\lambda$ defined on the 
element boundary to the function defined in the element interior by
connecting with the $P_1$ non-conforming finite element  functions.

Now we are ready to use the duality argument to prove Lemma
\LA{L2error}.

Given $\lambda\in \Ltilde$, 
let $\varphi_\lambda \in \Lambda$ and $\widetilde{\varphi}_\lambda \in \Ltilde$ as the solutions
to the following
problems, respectively, 
\begin{eqnarray}
\label{equation:phig}
a_h(\mu, \varphi_\lambda ) & = & (\U\mu, g)_{\T_h}, \quad\forall \mu \in \Lambda, \\
\label{equation:phigh} \qquad \widetilde{a}(\widetilde{\mu}, \phitilde_\lambda) & = & (\U\widetilde{\mu}, g)_{\T_h}, \quad \forall \widetilde{\mu} \in \Ltilde,
\end{eqnarray}
where $g$ is a function defined element-wise as 
\begin{equation}\label{equation:gdef}
g=\frac{1}{n+1}\sum_{i=1}^{n+1} {Q_k}(S_i^K\lambda) 
\end{equation}
 in an element
$K$. $g$ is a $L^2$ function  and let $u^*_g$ be the weak solution
to the adjoint problem  \EQ{pdea}  $L^*u^*_g=g$. 
We know that $u^*_g \in H^2(\Omega)$
under the regularity assumption \EQ{regularity}. 
Let ${\bf Q}^*_g=\epsilon \nabla u^*_g$ and $\varphi^*_g$ be the trace of
$u^*_g$ on the element boundary.  We have
\begin{equation}\label{equation:ad}
-\nabla\cdot {\bf Q}^*_g-\vbeta\cdot\nabla u^*_g-\nabla\cdot\vbeta u^*_g=g.
\end{equation}

Denote
$({\bf Q}_g, u_g,\varphi_g)$ as the HDG solution of the
adjoint problem \EQ{ad}.  
We can see that $\varphi_\lambda=\varphi_g$, since for arbitrary $\mu\in\Lambda$,
\begin{align*}
a_h(\mu,\varphi_{\lambda})&=D((\Q\mu,\U\mu,\mu);(\Q\varphi_\lambda,\U\varphi_\lambda,\varphi_\lambda))
=D((\Q\mu,\U\mu,\mu);({\bf Q}_g,u_g,\varphi_\lambda))\\
&=D((\Q\mu,\U\mu,\mu);({\bf Q}_g,u_g,\varphi_\lambda-\varphi_g))+
D((\Q\mu,\U\mu,\mu);({\bf Q}_g,u_g,\varphi_g))\\
&=D((\Q\mu,\U\mu,\mu);({\bf Q}_g,u_g,\varphi_\lambda-\varphi_g))+(\U\mu,g)_{\T_h},
\end{align*}
where we use $({\bf Q}_g, u_g,\varphi_g)$ is  the HDG solution of the
adjoint problem in the last step.

Therefore, $\forall \mu\in\Lambda$,
\begin{align*}
D((\Q\mu,\U\mu,\mu);({\bf Q}_g,u_g,\varphi_\lambda-\varphi_g))=0
\end{align*}
and we have $\varphi_\lambda=\varphi_g.$

The following lemma gives an error estimate for the HDG solution of
the dual problem. Since we need to obtain the error estimate in
$B$-norm, the results in \cite{fu_analysis_2015-1} cannot be used directly.

\begin{lemma} \label{lemma:HDGerror}
  Let $P_M$ be the $L_2$ projection onto $\Lambda$. We have
  $$
\|\varphi_g-P_M\varphi_g^*\|_B\le C_E(\epsilon,h)|u^*|_{H^2},
$$
where
$C_E(\epsilon,h)=\epsilon^{\frac12}h+\Ctwo\epsilon^{\frac12}h+\Ctwo
\epsilon^{-\frac12} h^2+h^{\frac32}$.
\end{lemma}

\begin{proof} 

  Let ${\bf \Pi} _hq$ and $\Pi_h
u$ are projections defined in
\cite[(4.5a)-(4.5d)]{fu_analysis_2015-1}
and
\begin{align*}
 \epsilon^{\varphi_g}&:=\varphi_g-P_M\varphi^*_g,\epsilon^{u_{\varphi_g}}:=u_g-\Pi_h
                       u^*_g,\epsilon^{{\bf Q}_{\varphi_g}}:={\bf
                       Q}_g-{\bf \Pi}_h {\bf Q}^*_g,\\
  \delta^{\varphi_g}&:=\varphi^*_g-P_M\varphi^*_g,\delta^{u_{\varphi_g}}:=u^*_g-\Pi_h u^*_g,\delta^{{\bf Q}_{\varphi_g}}:={\bf Q}^*_g-{\bf \Pi}_h{\bf Q}^*_g.
 \end{align*}

 We have
\begin{align*}
\|\epsilon^{\varphi_g}\|^2_{B}&=D((\Q\epsilon^{\varphi_g},\U\epsilon^{\varphi_g},\epsilon^{\varphi_g});
(\Q\epsilon^{\varphi_g},\U\epsilon^{\varphi_g},\epsilon^{\varphi_g}))=D((\Q\epsilon^{\varphi_g},\U\epsilon^{\varphi_g},\epsilon^{\varphi_g});
(\epsilon^{{\bf Q}_{\varphi_g}},\epsilon^{{u}_{\varphi_g}},\epsilon^{\varphi_g}))\\
&=D((\Q\epsilon^{\varphi_g},\U\epsilon^{\varphi_g},\epsilon^{\varphi_g});(\delta^{{\bf Q}_{\varphi_g}},\delta^{{u}_{\varphi_g}},\delta^{\varphi_g})),
\end{align*}
where the third equality follows from Galerkin orthogonality of the dual problem that 
$$D((Q\epsilon^{u_{\varphi_g}},\U\epsilon^{u_{\varphi_g}}, \epsilon^{u_{\varphi_g}});({\bf Q}_g-{\bf Q}^*_g,{u}_g-{u}^*_g,\varphi_g-\varphi^*_g))=0.$$

Next we estimate
$D((\Q\epsilon^{\varphi_g},\U\epsilon^{\varphi_g},\epsilon^{\varphi_g});(\delta^{{\bf
    Q}_{\varphi_g}},\delta^{{u}_{\varphi_g}},\delta^{\varphi_g}))$. Following
the proof of \cite[Lemma 4.4]{fu_analysis_2015-1} and using the
definition of the projections ${\bf \Pi}_h$ and $\Pi_h$, we have
\begin{align*}
&D((\Q\epsilon^{\varphi_g},\U\epsilon^{\varphi_g},\epsilon^{\varphi_g});(\delta^{{\bf Q}_{\varphi_g}},\delta^{{u}_{\varphi_g}},\delta^{\varphi_g}))\\
&=(\epsilon^{-1}{\Q}\epsilon^{\varphi_g}, \delta^{{\bf Q}_{\varphi_g}})_{\T_h}-(\U\epsilon^{\varphi_g}, \nabla\cdot\delta^{{\bf Q}_{\varphi_g}}) _{\T_h}+\langle\epsilon^{\varphi_g}, \delta^{{\bf Q}_{\varphi_g}}\cdot {\bf n}\rangle_{\partial\T_h}\nonumber\\
&\quad -({\Q}\epsilon^{\varphi_g}+{\bm\beta}\U\epsilon^{\varphi_g},\nabla\delta^{{\bf u}_{\varphi_g}} ) _{\T_h}+\langle ({\Q}\epsilon^{\varphi_g}+{\bm\beta}\epsilon^{\varphi_g})\cdot{\bf n}+\tau_K(\U\epsilon^{\varphi_g}-\epsilon^{\varphi_g}),\delta^{u_{\varphi_g}}\rangle_{\partial\T_h}\\
&\quad-(\nabla\cdot{\bm\beta}\U\epsilon^{\varphi_g},\delta^{u_{\varphi_g}}) _{\T_h}-\langle(\Q\epsilon^{\varphi_g}+{\bm \beta}\epsilon^{\varphi_g})\cdot{\bf n}+\tau_K(\U\epsilon^{\varphi_g}-\epsilon^{\varphi_g}),\delta^{\varphi_g}\rangle_{\partial\T_h}\\
&=\epsilon^{-1}(\Q\epsilon^{\varphi_g},\delta^{{\bf Q}_{\varphi_g}})_{\T_h}+(\nabla\U\epsilon^{\varphi_g},\delta^{{\bf Q}_{\varphi_g}})_{\T_h}-
\langle\U\epsilon^{\varphi_g}-\epsilon^{\varphi_g},\delta^{{\bf Q}_{\varphi_g}}\cdot{\bf n}\rangle_{\partial\T_h}
\nonumber\\
&\quad+(\nabla\cdot\Q\epsilon^{\varphi_g},\delta^{u_{\varphi_g}})_{\T_h}+
(\vbeta\cdot\nabla\U\epsilon^{\varphi_g},\delta^{u_{\varphi_g}})_{\T_h}
+\langle (\tau_K-\vbeta\cdot{\bf n}) (\U\epsilon^{\varphi_g}-\epsilon^{\varphi_g}),\delta^{u_{\varphi_g}}\rangle_{\partial\T_h}
\nonumber\\
&\quad-\langle(\Q\epsilon^{\varphi_g}+{\bm \beta}\epsilon^{\varphi_g})\cdot{\bf n}+\tau_K(\U\epsilon^{\varphi_g}-\epsilon^{\varphi_g}),\delta^{\varphi_g}\rangle_{\partial\T_h}\\
&=\epsilon^{-1}(\Q\epsilon^{\varphi_g},\delta^{{\bf Q}_{\varphi_g}})_{\T_h}-
\langle\U\epsilon^{\varphi_g}-\epsilon^{\varphi_g},\delta^{{\bf Q}_{\varphi_g}}\cdot{\bf n}\rangle_{\partial\T_h}+
(\vbeta\cdot\nabla\U\epsilon^{\varphi_g},\delta^{u_{\varphi_g}})_{\T_h}\nonumber\\
&\quad+\langle(\tau_K-\vbeta\cdot{\bf n})(\U\epsilon^{\varphi_g}-\epsilon^{\varphi_g}),\delta^{u_{\varphi_g}}\rangle_{\partial\T_h}
\nonumber\\
&\leq C\epsilon^{-\frac12}\|\epsilon^{\varphi_g}\|_B\|\delta^{{\bf Q}_{\varphi_g}}\|_{\T_h}+C\|\U\epsilon^{\varphi_g}-\epsilon^{\varphi_g}\|_{\partial\T_h}\|\delta^{{\bf Q}_g}\|_{\partial\T_h}
+C\|\nabla\U\epsilon^{\varphi_g}\|_{\T_h}\|\delta^{u_{\varphi_g}}\|_{\T_h}+C\|\epsilon^{\varphi_g}\|_B\|\delta^{u_{\varphi_g}}\|_{\partial\T_h}\\
&\leq C\epsilon^{-\frac12}\|\epsilon^{\varphi_g}\|_B\|\delta^{{\bf Q}_{\varphi_g}}\|_{\T_h}+Ch^{\frac12}\nu_{\epsilon,\tau,h}\interleave\epsilon^{\varphi_g}\interleave_{\T_h}\|\delta^{{\bf Q}_g}\|_{\partial\T_h}
+C\Ctwo \interleave\epsilon^{\varphi_g}\interleave_{\T_h}\|\delta^{u_{\varphi_g}}\|_{\T_h}+C\|\epsilon^{\varphi_g}\|_B\|\delta^{u_{\varphi_g}}\|_{\partial\T_h}\\
&\leq C(\epsilon^{-\frac12}\|\epsilon^{\varphi_g}\|_B+\epsilon^{-\frac12}\nu_{\epsilon,\tau,h}
\|\epsilon^{\varphi_g}\|_B)\|\delta^{{\bf Q}_{\varphi_g}}\|_{\T_h}
+C\epsilon^{-\frac12}\nu_{\epsilon,\tau_K,h}\|\epsilon^{\varphi_g}\|_{B}\|\delta^{u_{\varphi_g}}\|_{\T_h}+C\|\epsilon^{\varphi_g}\|_B\|\delta^{u_{\varphi_g}}\|_{\partial\T_h}
\end{align*}
where the second equality follows from the facts 
\begin{align*}
&(\delta^{u_{\varphi_g}},\nabla\cdot{\Q}\epsilon^{\varphi_g})_{\T_h}=0,\quad
                 (\delta^{{\bf Q}_{\varphi_g}},\nabla\cdot{\U}\epsilon^{\varphi_g})_{\T_h}=0,
\end{align*}
the third equality follows from the definition of ${\bf \Pi}_h, P_M$ and the fact that
$\langle\vbeta\cdot{\bf
  n}\epsilon^{\varphi_g},\delta^{\varphi_g}\rangle_{\partial\T_h}=0$. The
second inequality follows from Lemmas $\ref{lemma:Ubound}.$
and \LA{IJQzero}.
The last inequality follows from Lemma
\LA{equivalent}.

Cancelling the common factor both sides, we obtain
$$\|\epsilon^{\varphi_g}\|_B\leq C(\epsilon^{-\frac12}+\epsilon^{-\frac12}\nu_{\epsilon,\tau,h})
\|\delta^{{\bf Q}_{\varphi_g}}\|_{\T_h}
+C\epsilon^{-\frac12}\nu_{\epsilon,\tau_K,h}\|\delta^{u_{\varphi_g}}\|_{\T_h}+Ch^{-\frac12}\|\delta^{u_{\varphi_g}}\|_{\T_h}.$$

Using the properties of the projections, see \cite{fu_analysis_2015-1}, we have 
$$\|\delta^{{\bf Q}_{\varphi_g} }\|_{\T_h}\leq h|{\bf Q}^*_{g}|_{H^1}=\epsilon h|u^*_g|_{H^2},
\|\delta^{{u}_{\varphi_g} }\|_{\T_h}\leq Ch^2|u^*_g|_{H^2},$$
which implies that 
\begin{align*}
\|\epsilon^{\varphi_g}\|_B&\leq C(\epsilon^{-\frac12}+\Ctwo\epsilon^{-\frac12})
\|\delta^{{\bf Q}_{\varphi_g}}\|_{\T_h}+(\epsilon^{-\frac12}\nu_{\epsilon,\tau_K,h}+h^{-\frac12})\|\delta^{u_{\varphi_g}}\|_{\T_h}\\
&\leq C(\epsilon^{-\frac12}+\Ctwo\epsilon^{-\frac12})\epsilon h|u^*_g|_{H^2}+
(\epsilon^{-\frac12}\nu_{\epsilon,\tau_K,h}+h^{-\frac12})h^2|u^*_g|_{H^2}\\
&=C(\epsilon^{\frac12}h+\Ctwo\epsilon^{\frac12} h+\Ctwo \epsilon^{-\frac12}h^2+h^{\frac32})|u^*_g|_{H^2}.
\end{align*}

\end{proof}

We will prove the following lemma which is similar to \cite[Lemma
5.1]{TuLi:2008:BDDCAD}.
However, there are several technical differences. Since for any $q \in
\Ltilde$, $q$  is only defined on the element boundary, we need to  use  $\U
q$ which is the extension of $q$ into interior element to define the
$L^2$ inner product. Due to the HDG discretization, we use $\varphi_g$,
the solution of HDG, in the definition of $L_h$, which requires an
error bound provided in Lemma \LA{HDGerror}. This new definition of
$L_h$ simplifies the proof of Lemma \LA{ErrorBound}.
Moreover, the estimation of the term related to $q$ requires more tools.

\begin{lemma}\label{lemma:exact}
Let Assumption~\ref{assump:onWtilde} hold and
$L_h(q,\varphi_g) = \widetilde{a}(q,\varphi_g)- (\U q,g)_{\T_h}$, for $q \in
\Ltilde$. There exists a constant $C$, which is independent of
$\epsilon$, 
$H$, and $h$, such that for all $q \in \Ltilde$, 
$|L_h(q,\varphi_g)| \leq C C_L(\epsilon,H,h)\| u^*_g \|_{H^2}\| q \|_\Btilde$,
where
$C_L(\epsilon,H,h)=\frac{1}{\sqrt{\epsilon}}\max(H\epsilon, H^2)$.
\end{lemma}

\beginproof For any
$q \in \Ltilde$, we have
\begin{eqnarray*}
\widetilde{a}(q,\varphi_g)&= &\sum_{i=1}^N
                               a^{(i)}(q^{(i)},\varphi_g)=\sum_{i=1}^N
                               a_h^{(i)}(q^{(i)},\varphi_g)+\frac12\langle{\bm\beta}\cdot{\bf
  n}q^{(i)},\varphi_g\rangle_{\partial \Omega_i}\\
&=&\sum_{i=1}^N  D^{(i)}((\mathcal{Q} q^{(i)},
                               \mathcal{U}q^{(i)}, q^{(i)});  {({\bf Q}_g, u_g,\varphi_g)}) +\frac12\langle{\bm\beta}\cdot{\bf
  n}q^{(i)},\varphi_g\rangle_{\partial \Omega_i}\\\
  &=&\sum_{i=1}^N  D^{(i)}((\mathcal{Q} q^{(i)},
                               \mathcal{U}q^{(i)}, q^{(i)}); ( {{\bf Q}_g}- {{\bf Q}^*_g}, u_g-u^*_g,\varphi_g-\varphi^*_g)) \\
   &&\qquad +D^{(i)}((\mathcal{Q} q^{(i)},
                               \mathcal{U}q^{(i)}, q^{(i)});(  {{\bf Q}^*_g}, u^*_g,\varphi^*_g))+\frac12\langle{\bm\beta}\cdot{\bf
  n}q^{(i)},\varphi_g\rangle_{\partial \Omega_i}\\\
&=&\sum_{i=1}^N \left\{ (\epsilon^{-1} \MQ q^{(i)},
     {{\bf Q}^*_g})_{\T_h(\Omega_i)}-(\MU q^{(i)},\nabla\cdot  {{\bf Q}^*_g})_{\T_h(\Omega_i)}+\langle q^{(i)},  {{\bf Q}^*_g}\cdot{\bf n}\rangle_{\partial\T_h(\Omega_i)}\right. \\
&&\qquad  -(\MQ q^{(i)}+\vbeta \MU q^{(i)}, \nabla u^*_g)_{\T_h(\Omega_i)}+\langle
   (\MQ q^{(i)}+\vbeta q^{(i)})\cdot {\bf n}+\tau_K (\MU q^{(i)}-q^{(i)}), u^*_g\rangle_{\partial\T_h(\Omega_i)}\\
&&\qquad \left.-((\nabla \cdot \vbeta)\MU q^{(i)},
   u^*_{g^{(i)}})_{\T_h(\Omega_i)}-\langle   (\MQ q^{(i)}+\vbeta q^{(i)})\cdot {\bf n}+\tau_K (\MU
   q^{(i)}-q^{(i)}), \varphi^*_g\rangle_{\partial\T_h(\Omega_i)}\right\} \\
&&\qquad +\frac12\langle{\bm\beta}\cdot{\bf
  n}q^{(i)},\varphi_g\rangle_{\partial\T_h(\Omega_i)}\\
&=&\sum_{i=1}^N \left\{ \langle q^{(i)},  {{\bf Q}^*_g}\cdot{\bf n}\rangle_{\partial\T_h(\Omega_i)} +(\MQ q^{(i)}, \epsilon^{-1} {{\bf Q}^*_g}-\nabla
    u^*_g)_{\T_h(\Omega_i)}\right.\\
&&\qquad + (\MU q^{(i)},-\nabla\cdot {{\bf Q}^*_g}-\vbeta\cdot\nabla u^*_g-(\nabla
    \cdot \vbeta) u^*_g)_{\T_h(\Omega_i)}\\
&&\qquad \left. +\langle
   (\MQ q^{(i)}+\vbeta q^{(i)})\cdot {\bf n}+\tau_K (\MU
   q^{(i)}-q^{(i)}), u^*_g\rangle_{\partial \T_h(\Omega_i)} \right.\\
&&\qquad \left.-\langle   (\MQ q^{(i)}+\vbeta q^{(i)})\cdot {\bf n}+\tau_K (\MU q^{(i)}-q^{(i)}), \varphi^*_g\rangle_{\partial \T_h(\Omega_i)}\right\} 
+\frac12\langle{\bm\beta}\cdot{\bf
  n}q^{(i)},\varphi_g\rangle_{\partial\T_h(\Omega_i)}\\
&=&\sum_{i=1}^N \epsilon\langle q^{(i)}, \partial_n u^*_g\rangle_{\partial\Omega_i} +\frac12\langle{\bm\beta}\cdot{\bf
  n}q^{(i)},\varphi_g\rangle_{\partial \Omega_i}+(\MU q^{(i)}, g)_{\Omega_i},
\end{eqnarray*}
where we use the definition of $a^{(i)}$ defined in \EQ{lbformA} for the second equality and \EQ{locDaeqn} for the third equality. The Galerkin-orthogonality condition for the dual discretization
$$D^{(i)}((\mathcal{Q} q^{(i)},
                               \mathcal{U}q^{(i)}, q^{(i)});(  {{\bf Q}_g-{\bf Q}^*_g}, u_g-u^*_g,\varphi_g-\varphi^*_g))=0$$ and  the
definition of $D^{(i)}$ \EQ{locBform} are used for the fifth equality. The
definitions of $ {{\bf Q}^*_g}$, $u^*_g$, $\varphi^*_g$ and \EQ{ad}
for the second, third and
fourth terms in the last equality. 

Therefore we have
\begin{eqnarray}\label{equation:Lh1}
&&L_h(q,\varphi_g) = \widetilde{a}(q,\varphi_g)-
                      (\U q,g)_{\Omega_i}\\
&=&\sum_{i=1}^N \epsilon\langle
                      q^{(i)}, \partial_n u^*_g\rangle_{\partial\Omega_i}+\frac12\langle{\bm\beta}\cdot{\bf
  n}q^{(i)},\varphi^*_g\rangle_{\partial \Omega_i}
 -\frac12\langle{\bm\beta}\cdot{\bf n}q^{(i)},\varphi^*_g\rangle_{\partial \Omega_i}+\frac12\langle{\bm\beta}\cdot{\bf n}q^{(i)},\varphi_g\rangle_{\partial \Omega_i}\nonumber\\
&=&\sum_{i=1}^N
                      \left(\epsilon\langle q^{(i)}, \partial_n
                      u^*_g\rangle_{\partial\Omega_i}+\frac12\langle{\bm\beta}\cdot{\bf
  n}q^{(i)},u^*_g\rangle_{\partial \Omega_i}\right)+
 \left(-\frac12\langle{\bm\beta}\cdot{\bf n}q^{(i)},\varphi^*_g\rangle_{\partial \Omega_i}+\frac12\langle{\bm\beta}\cdot{\bf n}q^{(i)},\varphi_g\rangle_{\partial \Omega_i}\right),\nonumber
\end{eqnarray}
where we use $\varphi^*_g=u^*_g$ on the element boundary for the second equality.

Denote the common average of $q$ on the edge $\E^{ij}$ of $\partial\Omega_i$
by $\overline{q}_{\E^{ij}}$.
\begin{eqnarray*}
&&{\sum_{i=1}^N
                      \left(\epsilon\langle q^{(i)}, \partial_n
                      u^*_g\rangle_{\partial\Omega_i}{+}\frac12\langle{\bm\beta}\cdot{\bf
  n}q^{(i)},u^*_g\rangle_{\partial \Omega_i}\right)}\\
  &&= \sum_{i=1}^N\sum_{\E^{ij}\subset
\partial\Omega_i} \int_{\E^{ij}} \epsilon
\partial_n  u^*_g (q^{(i)} - \overline{q}_{\E^{ij}} ){+}\frac12
\vbeta\cdot\vvec{n} u^*_g (q^{(i)} -  \overline{q}_{\E^{ij}} )ds,
\end{eqnarray*}
where we have also subtracted the constant average values
$\overline{q}_{\E^{ij}}$ from $q^{(i)}$,
which does not change the sum. Then, from
Assumption~\ref{assump:onWtilde}, we know that
\begin{eqnarray}
 &&{\sum_{i=1}^N
                      \left(\epsilon\langle q^{(i)}, \partial_n
                      u^*_g\rangle_{\partial\Omega_i}{+}\frac12\langle{\bm\beta}\cdot{\bf
  n}q^{(i)},u^*_g\rangle_{\partial \Omega_i}\right)}\nonumber\\
 &&= \sum_{i=1}^N\sum_{\E^{ij}\subset
\partial\Omega_i} \left(\int_{\E^{ij}} \epsilon
\partial_n \left( u^*_g -I^{(i)}_H u^*_g\right)\left(q^{(i)} -
              \overline{q}_{\E^{ij}} \right)~ds\right.\nonumber\\
&&\left.\quad{+}\frac12 \int_{\E^{ij}} 
\vbeta\cdot\vvec{n} \left( u^*_g -I^{(i)}_H u^*_g\right) \left(q^{(i)} -  \overline{q}_{\E^{ij}} \right)~ds \right),
\label{equation:terms} 
\end{eqnarray}
where $I^{(i)}_H u_g$ represents the interpolation of $u_g$ into
the finite element space on the coarse subdomain mesh.

Plug \EQ{terms} in \EQ{Lh1}, we have
\begin{eqnarray}\label{equation:Lh2}
L_h(q,\varphi_g) &=& \sum_{i=1}^N\sum_{\E^{ij}\subset
\partial\Omega_i} \int_{\E^{ij}} \epsilon
\partial_n \left( u^*_g -I^{(i)}_H u^*_g\right)\left(q^{(i)} -
              \overline{q}_{\E^{ij}} \right)~ds\nonumber\\
&+&\sum_{i=1}^N\sum_{\E^{ij}\subset
\partial\Omega_i} \frac12 \int_{\E^{ij}} 
\vbeta\cdot\vvec{n} \left( u^*_g -I^{(i)}_H u^*_g\right) \left(q^{(i)}
   -  \overline{q}_{\E^{ij}} \right)~ds \nonumber\\
  &+&
 \left(-\frac12\langle{\bm\beta}\cdot{\bf
     n}q^{(i)},\varphi^*_g\rangle_{\partial
     \Omega_i}+\frac12\langle{\bm\beta}\cdot{\bf
     n}q^{(i)},\varphi_g\rangle_{\partial \Omega_i}\right)\nonumber\\
  &:=&I_1+I_2+I_3.
     \end{eqnarray}
   
We will show in the following that each of the three terms
in~\EQ{Lh2} can be bounded by $CC_L\| u^*_g
\|_{H^2} \| q \|_\Btilde$, where $C$ is a positive constant
independent of $H$ and $h$.

For the first term $I_1$, from the Cauchy-Schwarz inequality, we
have
\begin{equation}\label{equation:0term}
\qquad |I_1|\le \sum_{i=1}^N\sum_{\E^{ij}\subset
\partial\Omega_i} \epsilon\left(\int_{\E^{ij}} |\nabla (u^*_g-I^{(i)}_H
u^*_g)|^2~ds \int_{\E^{ij}}  |q^{(i)}
-\overline{q}_{\E^{(ij)}}|^2~ds\right)^{1/2}.
\end{equation}
Using Lemmas ~\LA{trace}, \LA{inverse}, and \LA{WHapproximation}, we have for
the first factor
\begin{eqnarray}\label{equation:1term}
&&\int_{\E^{ij}} |\nabla (u^*_g - I^{(i)}_H u^*_g)|^2~ds\le
CH\|\nabla(u^*_g-I^{(i)}_H
u^*_g)\|_{H^1(\Omega_i)}^2\nonumber \\
&\le & CH\|u^*_g - I^{(i)}_Hu^*_g\|_{H^2(\Omega_i)}^2\le
CH|u^*_g|_{H^2(\Omega_i)}^2.
\end{eqnarray}

For the second factor, we can estimate as follows.
 Let $q^{(i)}_0=Q_0 q^{(i)} $ and we have $\overline{q}_{\E^{ij}} = \overline{q_0}_{\E^{ij}} $.
\begin{equation}\label{equation:qterm1}
\int_{\E^{ij}} | q^{(i)} -
\overline{q}_{\E^{ij}}|^2~ds
\le  \int_{\E^{ij}} | q^{(i)} -q^{(i)} _0 |^2~ds+\int_{\E^{ij}}
       | q_0^{(i)} -
\overline{q}_{\E^{ij}}|^2~ds.
\end{equation}

The first term in \EQ{qterm1} can be estimated by Lemma \LA{L2projection}
as
\begin{eqnarray*}
\int_{\E^{ij}} | q^{(i)} - q_0^{(i)}|^2~ds&\le& \sum_{K\in
    \T_h,K\subseteq \Omega_i}\left\|q^{(i)} - q_0^{(i)}\right\|^2_{L^2(\partial
    K)}\\ 
&\le& C \sum_{K\in
    \T_h,K\subseteq\Omega_i} h  \interleave q^{(i)}\interleave^2_{K} =  Ch \interleave q^{(i)}\interleave^2_{\Omega_i}.
\end{eqnarray*}

The second term in \EQ{qterm1} can be estimated by the isomorphism
$X_h$ defined by \EQ{Xhdef}. We follow
\cite[Equation (23)]{TuLi:2008:BDDCAD} and have
\begin{eqnarray*}
\int_{\E^{ij}}
       |  q_0^{(i)} -
\overline{q}_{\E^{(ij)}}|^2~ds&\le& \int_{\E^{ij}}
       | X_h(q_0^{(i)} - \overline{q_0}^{(i)}_{\E^{ij}})|^2~ds\\
&\le &  CH| X_h(q_0^{(i)}-
\overline{q}_{\E^{(ij)}} )|_{H^1(\Omega_i)}^2 
\le   CH \interleave q_0^{(i)}-
\overline{q}_{\E^{(ij)}} \interleave_{\Omega_i}^2 \\
&\le&   CH \interleave q_0^{(i)} \interleave_{\Omega_i}^2 
\le   CH \interleave q^{(i)} \interleave_{\Omega_i}^2 ,
\end{eqnarray*}
where we have used  the definition of $X_h$  for the first inequality, a  trace
theorem and a Poincar\'{e} 
inequality  of $P_1$ non-conforming finite element functions for the
second inequality, \EQ{p10} for the third
inequality, and Lemma \LA{L2projection} for the last inequality.

Therefore, we have
\begin{equation}\label{equation:2term}
\int_{\E^{ij}} | q^{(i)} -
\overline{q}_{\E^{ij}}|^2~ds
\le CH \interleave q^{(i)} \interleave_{\Omega_i}^2 .
\end{equation}

Combining  \EQ{0term}, \EQ{1term}, and
\EQ{2term}, we have the following bound for $I_1$,
\[
|I_1| \le  C \epsilon H \sum_{i=1}^N
|u^*_g|_{H^2(\Omega_i)}\interleave q^{(i)} \interleave_{\Omega_i}
\le C \sqrt{\epsilon} H|u^*_g|_{H^2(\Omega)}
\|q\|_{\Btilde},
\]
where we use the Cauchy-Schwarz inequality and
Lemma~\LA{equivalent} in the last step.

To derive a bound for $I_2$, we find
from the Cauchy-Schwarz inequality that
\begin{equation}\label{equation:20term}
\qquad |I_2|\le \sum_{i=1}^N\sum_{\E^{ij}\subset
\partial\Omega} \|\vbeta\|_{\infty}\Big(\int_{\E^{ij}} |u^*_g-I^{(i)}_H
u^*_g|^2~ds\int_{\E^{ij}} | q^{(i)}
-\overline{q}_{\E^{ij}}|^2~ds\Big)^{1/2}.
\end{equation}
Using a trace theorem and  Lemma~\LA{WHapproximation}, we have, for
the first factor on the right hand side of~\EQ{20term},
\begin{equation}\label{equation:21term}
\int_{\E^{ij}}|u^*_g-I^{(i)}_H u^*_g|^2~ds\le
CH\|u^*_g-I^{(i)}_H u^*_g\|_{H^1(\Omega_i)}^2 \le
CH^3|u^*_g|_{H^2(\Omega_i)}^2.
\end{equation}
Combining \EQ{20term}, \EQ{21term}, and \EQ{2term}, and using
Lemma~\LA{equivalent}, we have
\[
|I_2| \le C \|\vbeta\|_{\infty} H^2 \sum_{i=1}^N
|u^*_g|_{H^2(\Omega_i)}\interleave q^{(i)} \interleave_{\Omega_i}\le C \|\vbeta\|_{\infty}
\frac{H}{\sqrt{\epsilon}} H
|u^*_g|_{H^2(\Omega)} \|q\|_{\Btilde}.
\]

For $I_3$, we have 

\begin{eqnarray}\label{equation:I1}
I_3&=&\sum\limits_{i=1}^N\left(\frac12\langle{\bm\beta}\cdot{\bf n}q^{(i)},\varphi_g\rangle_{\partial \Omega_i}-\frac12\langle{\bm\beta}\cdot{\bf n}q^{(i)},\varphi^*_g\rangle_{\partial \Omega_i}\right)\nonumber\\
&=&\sum\limits_{i=1}^N\sum\limits_{\E_{ij}\subset\partial\Omega_i}\frac12\int_{\E_{ij}}{\bm\beta}\cdot{\bf n}(q^{(i)}-\overline{q}_{\E^{ij}})(\varphi_g-\varphi^*_g) ds \nonumber\\
&\leq &C\|\vbeta\cdot{\bf n}\|^{\frac12}_{\infty}\sum\limits_{i=1}^N\sum\limits_{\E_{ij}\subset\partial\Omega_i}\left(\int_{\E_{ij}}(q^{(i)}-\overline{q}_{\E^{ij}})^2ds\int_{\E_{ij}}(\tau-\frac12\vbeta\cdot{\bf n})(\varphi^*_g-\varphi_g)^2ds\right)^{\frac12}\nonumber\\
&\leq &C\|\vbeta\cdot{\bf n}\|^{\frac12}_{\infty}\sqrt{H}\interleave
                                                                                                                                                                                                                                                                      q\interleave\|(\tau-\frac12\vbeta\cdot{\bf n})^{\frac12}(\varphi^*_g-\varphi_g)\|_{\partial\T_h},
\end{eqnarray}
where the second inequality follows from $\eqref{equation:2term}$.

We will estimate $\|(\tau-\frac12\vbeta\cdot{\bf
  n})^{\frac12}(\varphi^*_g-\varphi_g)\|_{\partial\T_h}$ using the
error estimate in Lemma \LA{HDGerror}.
Recall $P_M$ is the $L_2$ projection onto $\Lambda$ and we have
\begin{equation}\label{equation:I1_1}
  \|\varphi^*_g-P_M\varphi^*_g\|_{\partial\T_h}\leq
  h^{\frac32}|u^*_g|_{H^2}.
\end{equation}
\begin{equation}\label{equation:I1_2}
\|(\tau-\frac12\vbeta\cdot{\bf n})^{\frac12}(\varphi^*_g-\varphi_g)\|_{\partial\T_h}\leq C\|\varphi^*_g-P_M\varphi^*_g\|_{\partial\T_h}+C\|(\varphi_g-P_M\varphi^*_g)\|_{\partial\T_h}.
\end{equation}

Let  $\epsilon^{\varphi_g}:=\varphi_g-P_M\varphi^*_g$. 
By Lemmas \LA{muk|||}, \LA{equivalent}, and \LA{HDGerror}, we have,
\begin{eqnarray}\label{equation:I1_3}
\|\epsilon^{\varphi_g}\|_{\partial\T_h}&=&h^{-\frac12} C\|\epsilon^{\varphi_g}\|_{h,\Omega}\leq Ch^{-\frac12}\interleave\epsilon^{\varphi_g}\interleave\leq Ch^{-\frac12}\epsilon^{-\frac12}\|\epsilon^{\varphi_g}\|_B\nonumber\\
&\leq &C(h^{\frac12}+\Ctwo h^{\frac12}+\Ctwo\epsilon^{-1}h^{\frac32}+h\epsilon^{-\frac12})|u^*_g|_{H^2}
\end{eqnarray}
Combining \EQ{I1_1}, \EQ{I1_2}, and \EQ{I1_3}, we have
\begin{align*}
\|(\tau-\frac12\vbeta\cdot{\bf
  n})^{\frac12}(\varphi^*_g-\varphi_g)\|_{\T_h}&\leq C(h^{\frac32}
                                                 +h^{\frac12}+\Ctwo h^{\frac12}+\Ctwo\epsilon^{-1}h^{\frac32}+\epsilon^{-\frac12}h)|u^*_g|_{H^2}
\end{align*}
and
$$|I_3|\le C\sqrt{H}(h^{\frac32}
                                                 +h^{\frac12}+\Ctwo
                                                  h^{\frac12}+\Ctwo\epsilon^{-1}h^{\frac32}+\epsilon^{-\frac12}h)\interleave
                                                 q\interleave |u^*_g|_{H^2}.$$

\endproof




\begin{lemma} \label{lemma:ErrorBound}
Let Assumption~\ref{assump:onWtilde} hold. $\varphi_\lambda$ and
$\phitilde_\lambda$ are solutions of \EQ{phig} and \EQ{phigh},
respectively, for $\lambda \in \Ltilde$. We have
$$
\| \varphi_\lambda- \phitilde_\lambda \|_{\Btilde}\leq C\frac{C_L(\epsilon,H,h)}{\epsilon}
\|g\|_{L^2(\Omega)},
$$
where $C_L(\epsilon,H,h)$ is given
in Lemma~\LA{exact}.
\end{lemma}

\beginproof
Recall $\varphi_\lambda$ is the HDG solution of \EQ{phig}. 
\begin{eqnarray*}
& & \| \phitilde_\lambda - \varphi_\lambda \|_{\Btilde}^2 =
\widetilde{a}(\phitilde_\lambda - \varphi_\lambda,
\phitilde_\lambda - \varphi_\lambda) \\
& = &  \widetilde{a}(\phitilde_\lambda-
\varphi_\lambda,\phitilde_\lambda) - \widetilde{a}( \phitilde_\lambda -
\varphi_\lambda,\varphi_{\lambda}) \\
&=&(\U(\phitilde_\lambda -
\varphi_\lambda),g)_{\T_h}- \widetilde{a}( \phitilde_\lambda -
\varphi_\lambda, \varphi_{\lambda}),
\end{eqnarray*}
where we have used \EQ{phigh} for the third equality.

Dividing by $\| \phitilde_\lambda -
\varphi_\lambda \|_{\Btilde}$ on both sides and denoting
$\phitilde_\lambda - \varphi_\lambda$ by $q$, we have, by the fact that
\begin{eqnarray}\label{equation:eee}
\| \phitilde_\lambda-\varphi_\lambda \|_{\Btilde} & \leq &  \frac{\big| (\U q,g)_{\T_h} -
\widetilde{a}(q, \varphi_{\lambda}) \big|}{\| q
\|_{\Btilde}} \nonumber\\
& \leq & C C_L(\epsilon,H,h)\|u^{*}_g\|_{H^2}\leq C\frac{C_L(\epsilon,H,h)}{\epsilon}\|g\|_{L^2(\Omega)},
\end{eqnarray}
where we use  Lemma
\LA{exact} for the second step and the regularity assumption \EQ{regularity} for
the last step. 
\endproof

The following lemma is \cite[Lemma 7.11]{TuLi:2008:BDDCAD}.
\begin{lemma}
\label{lemma:propw}  Let $w_\Gamma = \Stilde^{-1}_\Gamma
\Rtilde_{D,\Gamma} S_\Gamma \mu_\Gamma$, for $\mu_\Gamma \in \Lhat_\Gamma$.
Then for all $v \in \Rtilde(\Lhat)$, $\left<w_{{\A},\Gamma}, v
\right>_\Atilde=\big<\Rtilde \mu_{{\A},\Gamma}, v \big>_\Atilde$, i.e.,
$\big<\wE-\Rtilde \uE, v \big>_\Atilde=0$.
\end{lemma}

\beginproof
For any $v \in \Rtilde(\Lhat)$, denote its continuous interface part
by $v_\Gamma \in \Rtilde_\Gamma(\Lhat_\Gamma)$.  Given $\uG \in
\Lhat_\Gamma$, from Lemma~\LA{null} and the fact that $\RtildeG
\Rtilde_{D,\Gamma}^T v_\Gamma = v_\Gamma$, we have
\begin{eqnarray*}
\left<\wE,v \right>_\Atilde = \left<w_\Gamma,v_\Gamma
\right>_{\Stilde_\Gamma} =  v_\Gamma^T \STG \wG = v_\Gamma^T \STG
\Stilde^{-1}_\Gamma \Rtilde_{D,\Gamma} \SG \uG = v_\Gamma^T
\Rtilde_{D,\Gamma} \RtildeG^T \STG
\RtildeG \uG & & \\
= \big< \RtildeG \uG , \RtildeG \Rtilde_{D,\Gamma}^T v_\Gamma
\big>_\STG =  \big<\RtildeG \uG,v_\Gamma\big>_\STG = v_\Gamma^T \STG
\RtildeG \uG = \big<\Rtilde \uE, v \big>_\Atilde . \mbox{\endproof}
\end{eqnarray*}

The following lemma is similar to  \cite[Lemma
7.12]{TuLi:2008:BDDCAD}. However, we have to use $S$-norm defined in
\EQ{SKdef} instead of $L^2$ norm, which requires more technical tools.

\begin{lemma}
\label{lemma:approximation}  Let Assumption~\ref{assump:onWtilde}
hold. Given any $\mu_\Gamma \in
\Lhat_\Gamma$, let $w_\Gamma = \Stilde^{-1}_\Gamma \Rtilde_{D,\Gamma}
S_\Gamma \mu_\Gamma$. There then exists a
positive constant $C$ such that 
\[
\| \wE - \uE \|_{S}
 \leq C \frac{\Ctwo^4C_{ED}\mymax}{\epsilon^3 }
\|\uG\|_{B_\Gamma},
\]
where $C_L(\epsilon, H,h)$ is given in Lemma~\LA{exact}.
\end{lemma}

\beginproof
Given any $q^{(i)}, \lambda^{(i)} \in \Lambda^{(i)}$, by \EQ{SKdef}, \EQ{SKdef2}, and \EQ{gdef},  we have
\begin{eqnarray}
&&(q^{(i)},\lambda^{(i)})_{S(\Omega_i)}=\sum_{K\in
                 \T_h(\Omega_i)}\frac{1}{n+1}\sum_{j=1}^{n+1}(S_j^K
                 q^{(i)},S_j^K\lambda^{(i)})_K\nonumber\\
&=&\sum_{K\in
                 \T_h(\Omega_i)}\frac{1}{n+1}\sum_{j=1}^{n+1}\left((S_j^K
                 q^{(i)},Q_k S_j^K\lambda^{(i)})_K+(S_j^K
                 q^{(i)}, (1-Q_k)S_j^K\lambda^{(i)})_K\right)\nonumber\\
&=&\sum_{K\in
                 \T_h(\Omega_i)}\frac{1}{n+1}\sum_{j=1}^{n+1}(\MU
    q^{(i)},Q_k S_j^K\lambda^{(i)})_K\nonumber\\
&&+\sum_{K\in
                 \T_h(\Omega_i)}\frac{1}{n+1}\sum_{j=1}^{n+1}
\left((1-Q_k)S_j^Kq^{(i)}, (1-Q_k)S_j^K\lambda^{(i)})_K+(Q_kS_j^Kq^{(i)}, (1-Q_k)S_j^K\lambda^{(i)})_K\right)\nonumber\\
&=&\sum_{K\in
                 \T_h(\Omega_i)}(\MU q^{(i)}, \frac{1}{n+1}\sum_{j=1}^{n+1} Q_kS_j^K\lambda^{(i)})_K+\sum_{K\in
                 \T_h(\Omega_i)}\frac{1}{n+1}\sum_{j=1}^{n+1}
((1-Q_k)S_j^Kq^{(i)}, (1-Q_k)S_j^K\lambda^{(i)})_K\nonumber\\
&=& \sum_{K\in
                 \T_h(\Omega_i)}(\MU q^{(i)}, g)_K+\sum_{K\in
                 \T_h(\Omega_i)}\frac{1}{n+1}\sum_{j=1}^{n+1}
((1-Q_k)S_j^Kq^{(i)}, (1-Q_k)S_j^K\lambda^{(i)})_K\nonumber\\
&=&(\MU q^{(i)},g)_{\Omega_i}+\sum_{K\in
                 \T_h(\Omega_i)}\frac{1}{n+1}\sum_{j=1}^{n+1}
((1-Q_k)S_j^Kq^{(i)}, (1-Q_k)S_j^K\lambda^{(i)})_K.\label{equation:SnormCalculation}
\end{eqnarray}

Using the Friedrichs estimate and  \EQ{SKH1},  we have
$$\|(1-Q_k)S_j^Kq^{(i)}\|_{L^2(K)}\le Ch |S_j^Kq^{(i)}|_{H^1(K)}\le
C{\nu_{\epsilon,\tau_K,h}}h\interleave q^{(i)} \interleave_K$$
and
$$\|(1-Q_k)S_j^K\lambda^{(i)}\|_{L^2(K)}\le C \|S_j^K\lambda^{(i)}\|_{L^2(K)}.
$$
Combining these two estimates, we have
\begin{eqnarray}\label{equation:residualofSnorm}
&&\sum_{K\in
                 \T_h(\Omega_i)}\frac{1}{n+1}\sum_{j=1}^{n+1}
((1-Q_k)S_j^Kq^{(i)}, (1-Q_k)S_j^K\lambda^{(i)})_K\nonumber\\
&\le&
\sum_{K\in
                 \T_h(\Omega_i)}\frac{1}{n+1}\sum_{j=1}^{n+1}\|(1-Q_k)S_j^Kq^{(i)}\|_{L^2(K)}\|(1-Q_k)S_j^K\lambda^{(i)}\|_{L^2(K)}\nonumber\\
&\le &\sum_{K\in
                 \T_h(\Omega_i)}\frac{1}{n+1}\sum_{j=1}^{n+1}C{\nu_{\epsilon,\tau_K,h}}h\interleave
       q^{(i)} \interleave_K \|S_j^K\lambda^{(i)}\|_{L^2(K)}\nonumber\\
&=&\sum_{K\in
                 \T_h(\Omega_i)}C{\nu_{\epsilon,\tau_K,h}}h\interleave
       q^{(i)} \interleave_K \frac{1}{n+1}\sum_{j=1}^{n+1}\|S_j^K\lambda^{(i)}\|_{L^2(K)}\nonumber\\
&\le &C{\nu_{\epsilon,\tau_K,h}}h\interleave
       q^{(i)} \interleave_{\Omega_i} \|\lambda^{(i)}\|_{S(\Omega_i)}\nonumber\\
&\le &C{\nu_{\epsilon,\tau_K,h}}\frac{h}{\sqrt{\epsilon}}
      \| q^{(i)}\|_{B(\Omega_i)} \|\lambda^{(i)}\|_{S(\Omega_i)}.
\end{eqnarray}
We have, from \EQ{SnormCalculation}, \EQ{phig} and \EQ{phigh}, that
\begin{align*}
( \wE - \uE, \lambda )_{S} 
&=\widetilde{a}(\wE, \phitilde_\lambda) - a_h(\uE, \varphi_\lambda) \\
&\quad+\sum_{K\in
                 \T_h(\Omega)}
\frac{1}{n+1}\sum_{j=1}^{n+1}
((1-Q_k)S_j^Kw_{\A,\Gamma}, (1-Q_k)S_j^K\lambda)_K\\
&\quad-\sum_{K\in
                 \T_h(\Omega)}
\frac{1}{n+1}\sum_{j=1}^{n+1}
((1-Q_k)S_j^K\mu_{\A,\Gamma}, (1-Q_k)S_j^K\lambda)_K.
\end{align*}

Since $\varphi_\lambda\in \Lambda$, by  Lemma \LA{propw}, we know that $\widetilde{a}( \wE - \uE, \varphi_\lambda)
= 0$. Therefore,
\begin{eqnarray*}
&&\widetilde{a}(\wE, \phitilde_\lambda) - a_h(\uE, \varphi_\lambda) 
=\widetilde{a}(\wE , \phitilde_\lambda) - \widetilde{a}(\uE ,
\varphi_\lambda) \\
& =& ~ \widetilde{a}(\wE, \phitilde_\lambda) - \widetilde{a}(\wE , \varphi_\lambda)
 = ~ \widetilde{a}(\wE , \phitilde_\lambda
- \varphi_\lambda), 
\end{eqnarray*}
and, by Lemmas \LA{bnorm} and \LA{ErrorBound}, \EQ{residualofSnorm},
\begin{align*}
&\big| ( \wE - \uE, \lambda )_{S} \big|\\
 &\le\left|\widetilde{a}(\wE , \phitilde_\lambda
- \varphi_\lambda)\right|+\left|\sum_{K\in
                 \T_h(\Omega)}
\frac{1}{n+1}\sum_{j=1}^{n+1}
((1-Q_k)S_j^K(w_{\A,\Gamma}-\mu_{\A,\Gamma}), (1-Q_k)S_j^K\lambda)_K\right|\\
& \leq  \frac{\Ctwo^2}{{\epsilon}} \| \wE \|_{\Btilde} \| \varphi_\lambda  -
       \phitilde_\lambda \|_{\Btilde}+C{\nu_{\epsilon,\tau_K,h}}\frac{h}{\sqrt{\epsilon}}
      \| w_{\A,\Gamma}-\mu_{\A,\Gamma}\|_{B} \|\lambda\|_{S}\\ 
& \leq C \frac{\Ctwo^2\mymax}{\epsilon^2} \| \wE\|_{\Btilde} \| g \|_{L^2(\Omega)}+{\frac{\Ctwo h}{\sqrt{\epsilon}}(\|\wE\|_{\Btilde}+\|\mu_{\A,\Gamma}\|_{\widetilde{B}})\|\lambda\|_S}.
\end{align*}

Since $g=\frac{1}{n+1}\sum_{i=1}^{n+1}Q_kS_i^K\lambda$, we have
\begin{eqnarray*}
\| g \|^2_{L^2(\Omega)}&=&\sum_{K\in \T_h(\Omega)}(\frac{1}{n+1}\sum_{i=1}^{n+1}{Q_k}S_i^K\lambda,
                           \frac{1}{n+1}\sum_{i=1}^{n+1}{Q_k}S_i^K\lambda)_{K}\\
&\le &C\sum_{K\in \T_h(\Omega)}\frac{1}{n+1}\sum_{i=1}^{n+1}({Q_k}S_i^K\lambda,
                           {Q_k}S_i^K\lambda)_{K}\le {C\|\lambda\|^2_{S}}.
\end{eqnarray*}
Therefore, we have
\begin{eqnarray*}
\| \wE - \uE \|_{S} & = & \sup_{\lambda \in \Ltilde}
\frac{\big| ( \wE - \uE, \lambda )_{S} \big|}{\| \lambda \|_{S}} \\
& \leq &C( \frac{\Ctwo^2\mymax}{\epsilon^2}+{\frac{\Ctwo h}{\sqrt\epsilon}})~ (\| \wE \|_\Btilde+\|\uG\|_{B_\Gamma}) \\
&\le& C \frac{\Ctwo^4C_{ED}\mymax}{\epsilon^3 }
\|\uG\|_{B_\Gamma}, \mbox{\endproof}
\end{eqnarray*}
where we use Lemma \LA{wnorm} for the last step.

{\bf Proof of Lemma \LA{L2error}:}
\begin{proof}
Given $\mu_\Gamma\in \Lhat_\Gamma$, let $w_\Gamma = \Stilde^{-1}_\Gamma
\Rtilde_{D,\Gamma} S_\Gamma \mu_\Gamma$.  Recall the restriction operators $\widehat{R}_{\Gamma}^{(i)}$ and $\overline{R}_{\Gamma}^{(i)}$ restricts functions 
in $\Lhat_\Gamma$ and $\Ltilde_\Gamma$ to $\Lambda_{\Gamma}^{(i)}$.
To make the notation simpler, let $\H_{\A}
(\mu_\Gamma)=\mu_{\A,\Gamma}$. 

We have
$v_{\Gamma}=T\mu_{\Gamma}-\mu_{\Gamma}=\widetilde{R}^T_{D,\Gamma}w_{\Gamma}-\widetilde{R}^T_{D,\Gamma}\widetilde{R}_{\Gamma}\mu_{\Gamma}$
and
\begin{align*}
&\|v_{\A,\Gamma}\|_{h}
  =\sum_{i=1}^N
                 \|\mathcal{H}_{\mathcal{A}}(
                 \hatR_\Gamma^{(i)}(\widetilde{R}^T_{D,\Gamma}w_{\Gamma}-\mu_{\Gamma}))\|_{h,\Omega_i}\\
&=\sum_{i=1}^N
  \|\mathcal{H}_{\mathcal{A}}({\hatR}_\Gamma^{(i)}\widetilde{R}^T_{D,\Gamma}w_{\Gamma}-\overline{R}^{(i)}_\Gamma
  w_\Gamma)+\mathcal{H}_{\mathcal{A}}(\overline{R}^{(i)}_\Gamma
  w_\Gamma-{\hatR}_\Gamma^{(i)}\mu_{\Gamma}))\|_{h,\Omega_i}\\
&\le \sum_{i=1}^N
  \left(\|\mathcal{H}_{\mathcal{A}}({\hatR}_\Gamma^{(i)}\widetilde{R}^T_{D,\Gamma}w_{\Gamma}-\overline{R}^{(i)}_\Gamma
  w_\Gamma) \|_{h,\Omega_i}+\|\mathcal{H}_{\mathcal{A}}(\overline{R}^{(i)}_\Gamma
  w_\Gamma-{\hatR}_\Gamma^{(i)}\mu_{\Gamma}))\|_{h,\Omega_i}\right).
\end{align*}

For the first term,  since $w_{\Gamma}\in \Ltilde_\Gamma$, $\widetilde{R}^T_{D,\Gamma}w_{\Gamma} $ and
$w_{\Gamma} $ have the same edge average on $\partial\Omega_i$. We use Lemmas \LA{averagezeroscaling},
\LA{averagenorm}, \LA{wnorm} and have 
\begin{eqnarray*}
\sum_{i=1}^N \|\mathcal{H}_{\mathcal{A}}({\hatR}_\Gamma^{(i)}\widetilde{R}^T_{D,\Gamma}w_{\Gamma}-\overline{R}^{(i)}_\Gamma
  w_\Gamma) \|_{h,\Omega_i}
&\le &C \sum_{i=1}^N  H\interleave \mathcal{H}_{\mathcal{A}}({\hatR}_\Gamma^{(i)}\overline{R}^T_{D,\Gamma}w_{\Gamma}-\overline{R}^{(i)}_\Gamma
  w_\Gamma) \interleave_{\Omega_i}\\
&\le& C \sum_{i=1}^N \frac{H}{\sqrt{\epsilon}}\|\mathcal{H}_{\mathcal{A}}({\hatR}_\Gamma^{(i)}\widetilde{R}^T_{D,\Gamma}w_{\Gamma}-\overline{R}^{(i)}_\Gamma
  w_\Gamma)\|_{B(\Omega_i)}\\
&\le& C \sum_{i=1}^N \frac{H}{\sqrt{\epsilon}}\|{\hatR}_\Gamma^{(i)}\widetilde{R}^T_{D,\Gamma}w_{\Gamma}-\overline{R}^{(i)}_\Gamma
  w_\Gamma \|_{B_\Gamma^{(i)}} \le C\frac{H}{\sqrt{\epsilon}}C_{ED}\|w_{\Gamma}\|_{\Btilde_\Gamma}\\
&\le &\frac{\Ctwo^2
       HC^2_{ED}}{\epsilon^{\frac32}}\|\mu_\Gamma\|_{B_\Gamma}={
       \frac{\Ctwo^6 H\left(1+\log\frac{H}{h}\right)^2\gamma_{h,\tau_K}}{\epsilon^{\frac72}}\|\mu_\Gamma\|_{B_\Gamma}.}
\end{eqnarray*}

 For the second term, we use Lemmas \LA{SK}, \LA{approximation} and have
\begin{eqnarray*}
&&\sum_{i=1}^N \|\mathcal{H}_{\mathcal{A}}(\overline{R}^{(i)}_\Gamma
  w_\Gamma-{\hatR}_\Gamma^{(i)}\mu_{\Gamma})\|_{h,\Omega_i}
\le C\|w_{\mathcal{A},\Gamma}-\mu_{\mathcal{A},\Gamma}\|_S\\
&\le &
       C\frac{\Ctwo^4C_{ED}\mymax}{\epsilon^3}\|\mu_\Gamma\|_{B_\Gamma}\\
&=&C\frac{\Ctwo^6\left(1+\log\frac{H}{h}\right)\gamma^{\frac12}_{h,\tau_K}\mymax}{\epsilon^{4}}\|\mu_\Gamma\|_{B_\Gamma}.
\end{eqnarray*}

Therefore, we have
\[
\|v_{\A,\Gamma}\|_{h}\le C\frac{\Ctwo^6\left(1+\log\frac{H}{h}\right)^2\gamma_{h,\tau_K}\mymax}{\epsilon^4}\|\mu_\Gamma\|_{B_\Gamma}.
\]
\end{proof}

\begin{remark}\label{remark:error}
Compare to \cite[Lemma 7.14]{TuLi:2008:BDDCAD}, we estimate the first
and second term separately and therefore our estimate does not
have the factor $\frac{H}{h}$. However, the additional factors $\Ctwo^6$
and $\gamma_{h,\tau_K}$ are still in our estimate
which are from the HDG discretization. An additional
$\frac{1}{\epsilon}$ is due to the estimate of the average
operator $E_D$.  
\end{remark}

\section{Numerical experiments}\label{sec:numerical}

We present some numerical results for two examples, which were considered in \cite{006tos,TuLi:2008:BDDCAD}.
In \cite{TuLi:2008:BDDCAD}, the performance of the BDDC algorithms for
Test Problem I and II are very similar. Therefore, we only present the
results for Test Problem I (thermal boundary layer) and III (Rotating
flow field)  in \cite{TuLi:2008:BDDCAD}.

We decompose our domain $\Omega=[-1,1]\times [-1,1]$ into subdomains and each subdomain
into triangles. We then  discretize \EQ{pdea} with HDG discretization
of degree $0$, $1$, and $2$ polynomials. In \EQ{nqtrace}, we take the stabilizer
$\vbeta_{\E}=\max (\sup_{x\in \E} (\vbeta\cdot {\bf n},0))$  for any
$\E\in \partial K$ and $\forall K\in \T_h$, as defined in \cite[(2.10)]{fu_analysis_2015-1}.
We take $f=0$ in all the examples as in \cite{TuLi:2008:BDDCAD}.

We use the GMRES method without restart and work with the
$L^2$-norm. The iteration is stopped when $L^2$-norm of the initial
residual is reduced by a factor $10^{-10}$. The convergence rates in
$K$-norm are quite similar with using the $L^2$-norm \cite{TuLi:2008:BDDCAD}. 
We present the results of our BDDC algorithms employing different 
 primal continuity constraints. 
For BDDC$_1$, the subdomain 
edge average
constraints are  chosen as coarse level primal
variables, while  in BDDC$_2$, besides the subdomain edge average
constraints, one  additional edge flux weighted average constraint, defined
in Assumption \ref{assump:onWtilde}, 
for each edge, are also added to the coarse level
 subspace.   For BDDC$_3$, the first moment defined in Assumption
 \ref{assump:onWtilde}, is added additional to those used in BDDC$_2$. 


We note that the one-level and
two-level Robin-Robin methods,  which were developed and extended in
\cite{ Achdou:1997:RPA, Achdou:1999:DDN, Achdou:2000:DDAD}, are
closely related to our BDDC algorithms.  In
\cite{TuLi:2008:BDDCAD}, the performance of these  Robin-Robin 
algorithms are compared with similar BDDC algorithms as we discussed
in this paper.  The BDDC with extra flux weighted edge constraints
over-performed.  We do not present any results for
the these algorithms here.  See \cite{TuLi:2008:BDDCAD} for more details.

\subsection{Thermal boundary layer (Test Problem I)}
As in \cite{006tos}, see also \cite{Trotta:1996:MFE,Franca:1992:SFE,TuLi:2008:BDDCAD},
a thermal boundary layer problem is considered. The velocity field be
$\vbeta=\left(\frac{1+y}{2},0\right)$. The boundary condition is:
\[
\begin{array}{ll}
u=1,& \left\{\begin{array}{ll}
       x=-1 & -1< y\le 1, \\
      y=1,  & -1\le x\le 1,
      \end{array}
     \right. \\ [0.6ex]
u=0, & y=-1, -1\le x \le 1,\\ [0.6ex]  u=\frac{1+y}{2},& x=1,-1\le
y\le 1.
\end{array}
\]

In the first set of our experiments,  we change the number of the 
subdomains with a fixed
subdomain local problem. The results are reported in Table \ref{TR1_1}.
We note that, for large values of the viscosity $\epsilon$
($\epsilon\ge 0.01$) with different degrees of the  basis functions in HDG,  the
iteration counts for all BDDC algorithms are 
independent of
the number of the subdomains. However, with the increasing of the
degrees of the basis functions, the system become more ill-conditioned
and those additional constraints make the BDDC perform better. 

With a decreasing value of $\epsilon$ ($\epsilon< 0.01$), we can see
the scalability of the algorithms deteriorates.  But the more constraints used in the BDDC algorithms, the
performance  are  better. 
We also note that the numbers of the iterations are bounded for all
algorithms even when the viscosity $\epsilon$ goes to zero.  These
observations are similar to those for the FETI algorithms in
\cite{006tos}. When $\epsilon\le 10^{-5}$, the iterations are double
as we half $H$, the size of the subdomains, for all BDDC algorithms. 
In our theory, we  require $H$ small enough to obtain the
scalable results. In order to confirm that, we did larger experiments
for those degree zero basis functions, which we can afford. The results
are listed in Table \ref{TR1_0}.   When we increase the number of the
subdomains to $128$ for $\epsilon=10^{-3}$ ($\epsilon\ge 10^{-4}$ for BDDC$_3$), the iterations do not
increase anymore. However, for smaller $\epsilon$, $128$ is still not
larger enough to see the scalability of the algorithms.

In the second set of our experiments, we change the size of the subdomain local 
problems with a fixed
 number of the subdomains. The results are listed in Table \ref{TR1_2}.
From the results, we can see that the number of iterations increases slightly
with increasing 
$\frac{H}{h}$, the size of the subdomain local problems, when the viscosity $\epsilon$ is
large ($\epsilon\ge 10^{-4}$). These results are 
 consistent  with the results of the domain decomposition algorithms for symmetric,
positive definite problems.  Moreover, the more constraints are used in
the BDDC algorithms, the better performance we obtained. 

However, with the decreasing of the values of $\epsilon$, the iterations
for all BDDC algorithms are almost independent of the
values of $\epsilon$, $\frac{H}{h}$,  and the
constraints of the BDDC algorithms.  These results are consistent with
the results in \cite{TuLi:2008:BDDCAD}.

\subsection{Rotating flow field (Test Problem II)}

The second test case is more difficult one, see \cite{006tos,TuLi:2008:BDDCAD}.
The velocity field be
$\vbeta=\left(y,-x\right)$ and the boundary condition is:

\[
\begin{array}{ll}
u=1,& \left\{\begin{array}{ll}
       y=-1 & 0< x\le 1, \\
       y=1, & 0< x\le 1,\\
       x=1, &-1\le y\le 1,
      \end{array}
     \right. \\
 \\
u=0, & \mbox{elsewhere}.
\end{array}
\]

We also have two sets of the numerical experiments for this example.

In the first set of our experiments,  we change the number of subdomains
 with a fixed size of
the subdomain local
 problems.
 The results 
 are reported in Table \ref{TR2_1}.

Again, we note that, for large values of the viscosity $\epsilon\ge 10^{-3}$, the
iteration counts are independent of
the
 number of the subdomains for all three BDDC algorithms with different
 degrees of the basis functions. However,   
 the additional constraints are more significant in this difficult
case.  For the basis functions with degree $0$ and $1$, the iteration
counts with BDDC$_3$ are almost independent of the number of subdomains, even
with very small $\epsilon=10^{-6}$, while this is not the case for
the BDDC$_1$ algorithm with $\epsilon<10^{-4}$ for degree $0$ 
basis function.  The performance of BDDC$_2$ is between these two
algorithms. 

With the degree $2$ basis functions, the iteration counts increase
rapidly  with the decreasing of $\epsilon$. For $\epsilon=10^{-6}$,
the number of iterations for the BDDC$_1$ is almost $20$ times of
those for BDDC$_3$ for large number of subdomains. Even though the
iteration counts for 
BDDC$_3$  is increasing in this case, it is much less than double
when we half $H$, the size of the subdomains.

In the second set of our experiments, we change the size of the subdomain local problems with a
fixed number
 of the subdomains. The results
 are reported in Table \ref{TR2_2}.
From the results, we can see that the number of iterations increases with 
 increasing
$\frac{H}{h}$
for all values of the  viscosity $\epsilon$ for all BDDC
algorithms. The problems become more difficult when we increase the
degrees of the HDG basis functions. 
However, the BDDC$_3$ algorithm works 
significantly better than the others.

\begin{center}
\begin{table}[h]
\caption{Iteration counts for changing number of subdomains and
$H/h=6$ for Test Problem I} 

\begin{tabular}{|c|c|c|c|c||c|c|c|c||c|c|c|c|}\hline\label{TR1_1}
\# of Sub. &  $4^2$ &  $8^2$& $16^2$
& $32^2$ & $4^2$  & $8^2$&
 $16^2$ &  $32^2$&$4^2$ &$8^2$ & $16^2$
 & $32^2$ \\
\hline \hline $\epsilon$(deg=0)
&\multicolumn{4}{c||}{BDDC$_1$}&\multicolumn{4}{c|}{BDDC$_2$}&\multicolumn{4}{c|}{BDDC$_3$}\\
\hline
 $1e0 $ &11   &11  &12   &11       &10   &11     &11      &10        &9  &10  &10  &10 \\
 $1e-1$ &12   &14   &13   &13       &11   &12     &12      &12      &10  &11   &11  &11 \\                   
 $1e-2$ &11   &17 &21  &19         &9   &13   &15    &14           &8  &12  &13  &13 \\                       
 $1e-3$ &7   &13  &23 &33         &6  &11    &18    &25            &5&9   &15 &17 \\      
 $1e-4$ &5   &9   &17  &29        &5   &8    &14     &23           &4  &7  &11  &19  \\                 
 $1e-5$ &4   &7   &12  &24         &4   &6    &12     &22         &3  &6   &11  &19  \\                
 $1e-6$ &3   &6   &11  &20      &3   &6    &10     &20           &3  &5   &10  &19  \\         
\hline \hline $\epsilon$ (deg=1)
&\multicolumn{4}{c||}{BDDC$_1$}&\multicolumn{4}{c|}{BDDC$_2$}&\multicolumn{4}{c|}{BDDC$_3$}\\
\hline
 $1e0 $ &13   &14   &14    &14       &12   &13     &13      &13        &10  &11   &11  &11  \\
 $1e-1$ &16   &17  &16   &16        &13  &15     &15      &15         &12  &13   &13  &13 \\                   
 $1e-2$ &15   &23 &27  &25           &12   &18   &20    &19           &10  &15   &17  &16 \\                       
 $1e-3$ &9   &17  &30 &45             &8   &14     &26    &32          &7  &13  &21 &25 \\                       
 $1e-4$ &6   &11   &22  &38           &5   &10    &19     &35          &5  &9   &16  &29 \\                 
 $1e-5$ &4   &8   &16  &32             &4   &8    &15     &27           &4  &7   &13  &22 \\                
 $1e-6$ &4   &7   &12  &24             &4   &6    &12     &23          &4  &6   &11  &21  \\ 
\hline \hline $\epsilon$ (deg=2)
&\multicolumn{4}{c||}{BDDC$_1$}&\multicolumn{4}{c|}{BDDC$_2$}&\multicolumn{4}{c|}{BDDC$_3$}\\
\hline
$1e0 $ &14   & 15  &15  &15                 & 13  & 15  & 14&14   & 11  &11   & 11 & 11\\
 $1e-1$  &17   & 18  &18  &17                 &14   &17 &17 & 16  &13   &14   &14  &13 \\
 $1e-2$ & 18  & 25  & 29 & 27                & 14  & 20 &23 & 21  & 11  & 17  &18  &17\\
 $1e-3$  & 10  & 20  & 33 &49                 & 9  & 17  & 30&37   & 8  & 15  & 23 &26\\    
 $1e-4$  & 6  & 13  &24  & 44                & 6  &  12   &22 &41   & 6  & 10  & 19 &33\\ 
 $1e-5$  & 5  & 9  & 18 & 33                & 5  & 9  &  18       &30   &4   &8   &15  & 24 \\
 $1e-6$  & 4  &  7 & 14 & 29                & 4  & 7  & 13 &27   &  4 &  7 &13  & 25\\    
 \hline       \end{tabular}

\end{table}

\begin{center}

\begin{table}
\caption{Iteration counts for changing number of subdomains and
$H/h=6$ for Test Problem I with large number of subdomains} 

\begin{tabular}{|c|c|c|c|c||c|c|c|c||c|c|c|c|}\hline\label{TR1_0}
\# of Sub. &  $16^2$ &  $32^2$& $64^2$
& $128^2$&  $16^2$ &  $32^2$& $64^2$
& $128^2$&  $16^2$ &  $32^2$& $64^2$
& $128^2$\\
\hline \hline $\epsilon$ (deg=0)
&\multicolumn{4}{c||}{BDDC$_1$}&\multicolumn{4}{c|}{BDDC$_2$}&\multicolumn{4}{c|}{BDDC$_3$}\\
\hline
  $1e-3$ &23   &33  &48 & 35        &18 &25    &  20  & 17           &15&17   &15 &14 \\      
  $1e-4$ &17   &29   &45  &  81      &14   &23    & 41 &   55        &11  &19  &36  &29  \\                 
 $1e-5$ &12  &24   &45  &75         &12   &22    & 40    & 72        &11 &19   &35  & 67 \\                
 $1e-6$ &11  &20   & 40 &77      &10   &20    &  38   &72           &10  &19   & 35 & 67 \\         
\hline 

      \end{tabular}

\end{table}
\end{center}

\begin{table}[h]
\caption{Iteration counts for $6\times6$ subdomains and changing 
subdomain problem size Test Problem I} \center{
\begin{tabular}{|c|c|c|c|c||c|c|c|c||c|c|c|c|}\hline\label{TR1_2}
H/h. &  $4$ &  $8$& $16$
& $32$ & $4$  & $8$&
 $16$ &  $32$&$4$ &$8$ & $16$
 & $32$ \\\hline \hline $\epsilon$ (deg=0)
&\multicolumn{4}{c||}{BDDC$_1$}&\multicolumn{4}{c|}{BDDC$_2$}&\multicolumn{4}{c|}{BDDC$_3$}\\
\hline
$1e0 $ &10   &12 &14   &15         &10  &11     &13  &14        &9 &10   &10  &10  \\
 $1e-1$ &12  &14  &17  &19          &11  &13     &14  &16       &10 &11   &12  &13 \\                   
 $1e-2$ &13  &16   &19 &22           &10   &13   &15   &18        &9  &11  &13  &14 \\                       
 $1e-3$ &9   &11    &13 &16            &8   &9     &11     &14        &6 &7   &9 &11 \\                       
 $1e-4$ &6   &7    &9   &10            &6   &7    &8      &9             &5  &6  &7&8\\                 
 $1e-5$ &5   &6     &6   &7              &5   &5    &6     &7             &4  &5   &6  &6  \\                
 $1e-6$ &5   &5    &5   &6             &4   &5    &5     &6             &4  &4   &5  &5  \\        
\hline \hline $\epsilon$ (deg=1)
&\multicolumn{4}{c||}{BDDC$_1$}&\multicolumn{4}{c|}{BDDC$_2$}&\multicolumn{4}{c|}{BDDC$_3$}\\
\hline
$1e0 $ &13   &14   &16    &18     &12   &13   &14  &15            &10  &11   &11 &11 \\
 $1e-1$ &16  &17  &19   &21       &14    &15 &17    &17           &13   &13  &14   &14 \\                   
 $1e-2$ &17   &21 &24  &26        &14 &17   &18  &20              &13  &14   &15 &16\\                       
 $1e-3$ &12   &14 &17 &20        &10   &12  &14  &17              &8   &10   &12&13 \\                       
 $1e-4$ &8   &9  &10  &12          &7   &8    &9  &10                 &6   &7      &8  &9\\                 
 $1e-5$ &6   &7   &7  &8            &6   &6    &7     &8                 &5   &6     &7  &7  \\                
 $1e-6$ &5  &5   &6  &7             &5   &5    &6     &6                 &5   &5    &5  &6  \\    
     \hline \hline $\epsilon$ (deg=2)
&\multicolumn{4}{c||}{BDDC$_1$}&\multicolumn{4}{c|}{BDDC$_2$}&\multicolumn{4}{c|}{BDDC$_3$}\\
\hline
$1e0 $ &14   &15   &17  &19                 &13   &15   &16 & 16  &11   &11   & 11 & 11\\
 $1e-1$  &17   &19   &21  & 22                &16   &17   &18 &19   &13   &14   & 14 &14 \\                   
 $1e-2$ & 20  &23   &26  &  28               & 17  & 19  &20 & 22  & 14  & 15  &16  &16\\                       
 $1e-3$  &14   &17   &20  & 23                &12   &14   &17 & 20  & 10  &12   & 13 &15\\                       
 $1e-4$  & 9  & 10  &11  &  13               & 8  & 9  &10 & 12&7   &  8 &9  &10\\                 
 $1e-5$  & 7  &7   & 8 &  9                    &  6 &  7 &8 &  8 & 6  &  7 & 7 &8  \\                
 $1e-6$  & 5  &6   &7  &   7                    &5   &6   &6 & 7  &5   &6   &  6&7 \\    
  \hline    
\end{tabular}
}
\end{table}

\begin{table}[h]
\caption{Iteration counts for changing number of subdomains and
$H/h=6$ for Test Problem II} \center{
\begin{tabular}{|c|c|c|c|c||c|c|c|c||c|c|c|c|}\hline\label{TR2_1}
\# of Sub. &  $4^2$ &  $8^2$& $16^2$
& $32^2$ & $4^2$  & $8^2$&
 $16^2$ &  $32^2$&$4^2$ &$8^2$ & $16^2$
 & $32^2$ \\
\hline \hline $\epsilon$ (deg=0)
&\multicolumn{4}{c||}{BDDC$_1$}&\multicolumn{4}{c|}{BDDC$_2$}&\multicolumn{4}{c|}{BDDC$_3$}\\
\hline
 $1e0 $ &11   &12   &12    &12       &6  &6        &6      &6        &4  &4   &4  &3  \\
 $1e-1$ &13   &14   &14    &13       &7   &7        &7      &7        &5  &4   &4  &4 \\                   
 $1e-2$ &14  &19 &20 &19              &9   &9       &9      &8        &6  &5   &5  &5 \\                       
 $1e-3$ &23   &31  &38 &36            &13  &13     &12    &11      &8  &8   &8 &7 \\                       
 $1e-4$ &25  &39   &56  &62           &14   &16    &15    &16      &9  &9   &9  &10  \\                 
 $1e-5$ &25   &40   &59  &72          &14  &16     &16     &16      &9  &9   &9  &10 \\                
 $1e-6$ &25   &40   &60  &74           &14   &16    &16     &16      &9  &9   &9  &10  \\         
\hline \hline $\epsilon$ (deg=1)
&\multicolumn{4}{c||}{BDDC$_1$}&\multicolumn{4}{c|}{BDDC$_2$}&\multicolumn{4}{c|}{BDDC$_3$}\\
\hline
 $1e0 $ &13   &14   &15    &14               &7   &8     &8      &8              &5  &5   &4  &4  \\
 $1e-1$ &16   &18    &17  &16                &9  &10     &9      &9              &6  &6   & 6 &5 \\                   
 $1e-2$ &20   &26    &27  &25                &13   &13   &13    &11           &8 &8     &8  &7 \\                       
 $1e-3$ &47   &64    &69  &61               &29   &27     &22   &16           &18  &15  &12 &9 \\                       
 $1e-4$ &82   &153   &257  &279          &53   &77    &82    &72           &33  &38 &39 &32\\                 
 $1e-5$ &94   &226   &463  &894          &68   &118    &164 &185         &41  &55   &63  &68  \\                
 $1e-6$ &95   &246   &572  &$>$1000        &70  &132   &213  &277        &41  &61   &75 &80  \\       
\hline \hline $\epsilon$ (deg=2)
&\multicolumn{4}{c||}{BDDC$_1$}&\multicolumn{4}{c|}{BDDC$_2$}&\multicolumn{4}{c|}{BDDC$_3$}\\
\hline
$1e0 $ &14   &16   &16    &16       &8  &9        &9      &9        &5  &5   &5  &5  \\
 $1e-1$ &17   &19   &18    &18       &11   &11        &10      &10        &6  &6   &6  &6 \\                   
 $1e-2$ &22  &29 &30 &27              &15   &15       &14      &13        &9  &9   &9  &8 \\                       
 $1e-3$ &50   &67  &73 &67           &33  &29     &25    &19      &21  &16   &12 &11 \\                       
 $1e-4$ &118  &215   &303  &290     &83   &118    &114    &88      &57  &68   &60  &43  \\                 
 $1e-5$ &158   &417   &907 &$>$1000   &128  &262     &377     &413      &91  &167   &196  &215 \\                
 $1e-6$ &166   &509   &$>$1000 &$>$1000  &140   &371    &755     & $>$1000     &110  &236   &374  &526  \\   
 \hline         \end{tabular}
}
\end{table}

\begin{table}[h]
\caption{Iteration counts for $6\times6$ subdomains and changing 
subdomain problem size 
for Test Problem II} \center{
\begin{tabular}{|c|c|c|c|c||c|c|c|c||c|c|c|c|}\hline\label{TR2_2}
H/h. &  $4$ &  $8$& $16$
& $32$ & $4$  & $8$&
 $16$ &  $32$&$4$ &$8$ & $16$
 & $32$ \\\hline \hline $\epsilon$ (deg=0)
&\multicolumn{4}{c||}{BDDC$_1$}&\multicolumn{4}{c|}{BDDC$_2$}&\multicolumn{4}{c|}{BDDC$_3$}\\
\hline
$1e0 $ &10   &12   &14   &16          &5   &7   &8 &9             &3   &4   &5  &5    \\
 $1e-1$ &12  &15  &17   &19           &6   &8   &9 &11           &4   &5   &6  &6    \\                   
 $1e-2$ &15 &19    &23  &26           &7   &10  &12 &15        &6   &6   &8  &9    \\                       
 $1e-3$ &24  &34  &43   &51           &11   &15 & 20&25       &8   &9   &11  &14    \\                       
 $1e-4$ &26   &41 &59   &81           &12   &19  &29 &42      &8   &10   &16  &22    \\                 
 $1e-5$ &27  &42  & 62  &88           &12  &19  &30 &47       &8   &11   &17  &24     \\                
 $1e-6$ &27  &42  & 62  &90          &12   &19  &30 &47       &8   &11   &17 &24    \\        
\hline \hline $\epsilon$ (deg=1)
&\multicolumn{4}{c||}{BDDC$_1$}&\multicolumn{4}{c|}{BDDC$_2$}&\multicolumn{4}{c|}{BDDC$_3$}\\
\hline
$1e0 $ &13   &15   & 16&18                &7   &9   &9     &10                   &4   &5   &5  & 5\\
 $1e-1$ &16   &18   &20  &22              &8   &10   &11  &11                  &5   &6   &7 &7 \\                   
 $1e-2$ &22   &25  &29  &32               &11   &14   &16 &17                &7   &9   &10 &11\\                       
 $1e-3$  &54  &61   & 63 &64              &26   &31   & 32 &33              &14 &18   &18  &19\\                       
 $1e-4$  &97   &143   &177  &192        &49  &80   &105 &115            &25   &46   &65  &71\\                 
 $1e-5$  &112   &200  & 330 & 448       &61   &120 &208 &305           &29   &71   &127  &187  \\                
 $1e-6$  &114   & 214  & 380 & 636       &63   &137&279 &471          &31   & 80  &175  &288 \\    
    \hline \hline $\epsilon$ (deg=2)
&\multicolumn{4}{c||}{BDDC$_1$}&\multicolumn{4}{c|}{BDDC$_2$}&\multicolumn{4}{c|}{BDDC$_3$}\\
\hline
$1e0 $ &14   &16   & 17 &    19               & 9  &  9 & 10&   10  & 5  & 5  & 5 & 5\\
 $1e-1$  &17   &19   &  21&  23               &10   &11   &12 & 13  &6   & 7  &7  &7 \\                   
 $1e-2$ & 24  & 28  & 31 &    33             & 14  & 16  &17 &  19 &  9 &  10 & 11 &11\\                       
 $1e-3$  & 60  & 62  & 63 &   66              & 32  & 33  &33 & 35  & 18  & 19  & 19 &20\\                       
 $1e-4$  & 147  & 180  &188  &189         &89   & 113  &118 & 118  &51   &72   &73  &69\\                 
 $1e-5$  & 210  & 356  & 466 & 537           &144   & 236  &339 &370  &86   &164   & 222 &209  \\                
 $1e-6$  & 223  &  408 & 693 &  $>$1000         &167  & 330  &596 &832   & 99  & 235  &415  &518 \\    
    \hline
\end{tabular}
}

\end{table}

\end{center}

\bibliography{paper}
\bibliographystyle{plain}
\end{document}